\documentclass[a4paper,12pt]{article}
\usepackage[utf8]{inputenc}
\usepackage[T1]{fontenc}
\usepackage[english]{babel}
\usepackage{mathtools,amssymb,amsthm}
\usepackage{hyperref}
\usepackage{enumitem}
\usepackage{esint}
\usepackage{lmodern}
\usepackage{amsbsy} 
\usepackage{mathrsfs}
\usepackage{geometry} 
\usepackage{authblk}

\usepackage{tikz}

\usepackage{ulem}
\usepackage[usenames, dvipsnames]{xcolor}

\makeatletter
\newcommand{\addresses@store}{} 

\newcommand{\authorinfo}[4]{
  \appto\addresses@store{
    \par\noindent\textsc{#1}\par
    \noindent #2\par
    \ifstrempty{#3}{}{
      \noindent\textit{E-mail address:} \href{mailto:#3}{#3}\par
    }
    \ifstrempty{#4}{}{
      \noindent\textit{URL:} \url{#4}\par
    }
    \medskip
  }
}

\AtEndDocument{
  \section*{Affiliations}
  \begingroup\footnotesize
  \addresses@store
  \endgroup
}
\makeatother

\newcommand \Gn{\mathcal{G}_{\text{norm}}}

\newcommand \Om{\Omega}
\newcommand \QM{\mathscr{QM}}

\newcommand \om{\omega}
\newcommand \eps{\epsilon}
\newcommand \B{\mathcal{B}}
\newcommand \Cy{\mathcal{C}}
\newcommand \D{\mathbb{D}}

\newcommand \Sp{\mathbb{S}^{1}}

\newcommand \G{\mathcal{G}}
\newcommand \ov{\overline}
\newcommand \ovh{\widehat}
\newcommand \N{\mathbb{N}}
\newcommand \Z{\mathbb{Z}}

\newcommand \C{\mathbb{C}}
\newcommand \R{\mathbb{R}}
\newcommand \lin{\mathrm{lin}}
\newcommand \loc{\mathrm{loc}}

\newcommand \W{\mathcal{W}}

\newcommand \M{\mathscr{M}}

\newcommand \Spt{\mathrm{Spt}}
\newcommand \Oo{\mathcal{O}}

\numberwithin{equation}{section}

\newtheorem{theorem}{Theorem} 
\newtheorem{corollary}[theorem]{Corollary} 
\newtheorem{lemma} [theorem]{Lemma} 
\newtheorem{remark} [theorem]{Remark} 
\newtheorem{proposition} [theorem]{Proposition} 
 
\theoremstyle{definition} 
\newtheorem{definition}[theorem] {Definition}

\title{Boundary regularity of a fourth order Alt-Caffarelli problem and applications to the minimization of the critical buckling load}

\author{Jimmy Lamboley, Mickaël Nahon}

\authorinfo{Jimmy Lamboley}{\'Ecole Normale Supérieure, CNRS, PSL University, DMA and Sorbonne Université, Universit\'e Paris Cit\'e, IMJ-PRG, F-75005 Paris, France}{jimmy.lamboley@ens.fr}{}
\authorinfo{Mickaël Nahon}{Univ. Grenoble Alpes, CNRS, Grenoble INP, LJK, 38000 Grenoble, France.}{mickael.nahon@univ-grenoble-alpes.fr}{}

\date{}
\begin{document}
\maketitle

\begin{abstract}
We study a higher order analogue to the Alt-Caffarelli functional that arises in several shape optimization problems, among which the minimization of the critical buckling load of a clamped plate of fixed area. We obtain several regularity results up to the boundary in two dimensions, in particular we prove the full regularity of the boundary (analytic outside angles of opening $\approx 1.43\pi$) near any point of density less than $1$ of the optimal shape. These results are based on the monotonicity formula discovered by Dipierro, Karakhanyan, and Valdinoci in \cite{DKV20}, which we improve with a new epiperimetric inequality.
\end{abstract}

\tableofcontents

\section{Introduction}

\subsection{General context}
The goal of this paper is to study a variational formulation of a fourth order  free boundary problem, which is a natural fourth order version of the classical Alt-Caffarelli free boundary problem from \cite{AC81}.

Let $D$ some open subset of $\R^2$ and let $H^2(D)=\left\{u\in L^2(D):\nabla u,\nabla^2 u\in L^2(D)\right\}$. For $u\in H^2(D)$ we define
\begin{equation}\label{eq_energy}
E(u;D)=\int_{D}\left(|\Delta u|^2+\chi_{u\neq 0}\right),
\end{equation}
where $\chi$ designates the indicator function, meaning $\chi_{u\neq 0}(x)=\begin{cases}1\text{ if }u(x)\neq 0\\ 0\text{ if }u(x)=0\end{cases}$. We say $u$ is a minimizer of $E(\cdot\ ;D)$ if for every $v\in H^2(D)$ such that $\{u\neq v\}$ is compactly contained in $D$, we have
$$E(u;D)\leq E(v;D).$$
We will denote the set of minimizers of $E(\cdot\ ;D)$ by $\M(D)$. 

When the $|\Delta u|^2$ term in the energy is replaced by $|\nabla u|^2$, and if one assumes $u\geq 0$ near $\partial D$, as mentioned above we obtain the classical (one-phase) Alt-Caffarelli problem and it is well-known in this context that solutions $u$ are nonnegative, locally Lipschitz in $D$ (this regularity is called the optimal regularity, as one expects the gradient to be discontinuous between a region where $u=0$ and a region where $u>0$), and that the free boundary $\partial\{u>0\}$ is analytic. In higher dimension ($D\subset\R^d$ for some $d\geq 2$), it is known that the free boundary is fully regular for $d=3,4$, that singularities may appear in high dimension ($d\geq 7$), and the situation is still open in dimension $d=5,6$ (see for instance  \cite[Sec. 1.4]{V23}).

The situation with this fourth order version is more involved, and the existing literature is more scarce. In \cite{DKV20}, the authors studied a similar problem where $\{u\neq 0\}$ is replaced by $\{u>0\}$: as mentioned in \cite{GM24}, while these two cases lead to the same problem when considering the Alt-Caffarelli functional (as the energy of $u_+$ is lower than the energy of $u$), when dealing with a fourth order analogue, the two resulting problems have very different behaviour. The problem under study here is therefore closer to the one studied in \cite{GM24}, although the mathematical tools will have similarities with \cite{DKV20}, particularly with the use of a monotonicity formula as explained later in this introduction.

\begin{remark}\label{rem:ipp}
By integration by parts, $u$ belongs to the set of minimizers $\M(D)$ if and only if it is a minimizer (in the same sense) of one of the following functionals:
$$u\mapsto \int_{D}\left(|\nabla^2 u|^2+\chi_{u\neq 0}\right)\text{ or }u\mapsto \int_{D}\left(\left(\partial_{x,x}u-\partial_{y,y}u\right)^2+4\left(\partial_{x,y}u\right)^2+\chi_{u\neq 0}\right).$$
\end{remark}

The purpose of this paper is twofold:
\begin{itemize}
    \item On one hand, we study the regularity properties of the free boundary related to minimizers of $E(\cdot,D)$: in particular, we classify the blow-ups for this problem, highlighting the possibility of angular singularities, and we provide regularity properties near flat and angular points of the free boundary, using an approach based on a monotonicity formula (in the spirit of \cite{DKV20}) and an epiperimetric inequality (in the spirit of \cite{R64,SV19} and subsequent works). Some questions remain open, in particular the behavior of the free boundary near points of density 1 of the support of $u$, and the optimal regularity for $u$ which is expected to be $\mathcal{C}^{1,1}$. 
    
    We describe these results in Section \ref{ssect:results} below.
    \item On the other hand, we extend our results to more general energies, by introducing a suitable notion of quasi-minimizer, and we show that this applies in particular to the minimization of buckling eigenvalues. This is described in Section \ref{ssect:motivations} below.
\end{itemize}

Note that in the whole paper, we only consider  the two-dimensional case (though as seen in the next paragraph, this includes several rotationally symmetric three-dimensional problems).
Another natural variant would be to study minimizers of
\begin{equation}\label{eq_Stokespb}
\vec{v}\in H^1(D,\R^d)\mapsto \int_{D}\left(|\nabla \vec{v}|^2+\chi_{\vec{v}\neq 0}\right)
\end{equation}
under the constraint $\nabla\cdot \vec{v}=0$. Indeed for $d=2$, the minimization of the functional \eqref{eq_Stokespb} is equivalent (by identifying $\vec{v}=\nabla^\bot u$) to the minimization of the functional 
\begin{equation}\label{eq_Stokes_scalar}
u\in H^2(D,\R)\mapsto \int_{D}\left(|\Delta u|^2+\chi_{\nabla u\neq 0}\right).
\end{equation}
When $\{\nabla u=0\}$ is fully included (up to a negligible set) in a single level set of $u$, then a minimizer of \eqref{eq_Stokes_scalar} is (up to a translation) a minimizer of \eqref{eq_energy}, however the reciprocal is not true in general.

A minimizer $\vec{v}$ of \eqref{eq_Stokespb} is expected to verify Stokes' equation with an overdetermined boundary condition $\left(\vec{v}=0,\Vert\nabla v\Vert=1\right)$ on the regular free boundary.

This is reminiscent of some vectorial versions of the one phase Alt-Caffarelli problem, see for instance \cite{KL18,CSY18,MTV17,MTV20,MTV23,BMMTV24,BLNP23}, and the recent survey \cite{TV25}. A notable difference is that in each of these works, some version of the maximum principle is involved to prove the regularity of the boundary, something which is not available for the Stokes system.

\subsection{Motivations}\label{ssect:motivations}

Generalized versions of the functional $E(\cdot\ ;D)$ arise in several settings: we give examples coming both from spectral optimization and fluid mechanics.\\

\noindent{\bf Towards a Faber-Krahn inequality for the first buckling eigenvalue}:
Consider the first buckling eigenvalue of a clamped plate of shape $\Om$
\begin{equation}\label{expr_Lambda1}
\Lambda_1(\Om)=\inf_{u\in H^2_0(\Om)}\frac{\int_{\Om}|\Delta u|^2}{\int_{\Om}|\nabla u|^2}
\end{equation}
defined for any open set $\Om\subset\R^2$ of finite area (here $H^2_0(\Om)$ is the $H^2$-closure of compact support functions in $\Om$, see subsection \ref{subsec_notations} for notations). It was conjectured by P{\'o}lya and Szeg{\"o} in \cite{S50,PS51} that among open sets of given area, $\Lambda_1(\Om)$ is minimal when $\Om$ is a disk, by analogy with Rayleigh's since solved conjectures on the principal frequency of membranes or of clamped plates. It is known by an argument of Weinberger and Willms (reproduced in \cite[Prop 4.4]{K00} and \cite{AB03}) that under the assumption that there exists a minimizer $\Om^{\mathrm{opt}}$ that is bounded, of class $\mathcal{C}^2$, and simply connected, then $\Om^{\mathrm{opt}}$ must be a disk. The regularity hypothesis is central as it implies (by a shape derivative argument) that $\Delta u$ is locally constant on $\partial\Om$.

One of the motivations in this work is to partially weaken the regularity hypotheses of Weinberger \& Willms's argument: we will prove that minimizers of $\Lambda_1(\Om)$ under area constraint are necessarily bounded, and that they satisfy some partial regularity properties, stated in Theorem \ref{mr_total_quasi}.
In other words, we lower the hypothesis on Weinberger and Willms's argument on buckling eigenvalue as follows:
\begin{center}\textit{
Assume $\Om\subset\R^2$ is a minimizer of $|\Om|\Lambda_1(\Om)$ among open sets of finite area, such that $\Om$ is simply connected and for all $p\in\partial\Om$ we have
\begin{equation}\label{eq_upper_density_estimate}
\liminf_{r\to 0}\frac{|\D_{p,r}\cap\Om|}{|\D_{p,r}|}<\frac{t_1}{2\pi}(\approx 0.715).
\end{equation}
Then $\Om$ is a disk.}\end{center}
Here $\D_{p,r}$ is the disk centered at $p$ of radius $r$ and $t_1\approx 1.43\pi$ is uniquely defined by $\tan(t_1)=t_1$ and $t_1\in(\pi,2\pi)$.  In particular, if $\Om$ is a minimizer and $\Om$ is convex, then it must be a disk (this follows from the fact that $t_1>\pi$). These results are discussed in subsection \ref{subsec_buckling}, as well as some alternative conditions on $\Om$.

\begin{remark}\label{rk:regsimplyconnected}
Note that the topological hypothesis of Weinberger \& Willms ($\Om$ being simply connected) may be considered as a constraint, in an attempt to prove the minimality of the disk only among simply connected sets of given area. However, proving the regularity of the boundary under this type of topological constraint goes beyond the scope of this work: as far as the authors know, this is not well-understood even for the classical Alt-Caffarelli problem.
\end{remark}

Let us now mention some other problems that lead to similar free boundary problems and where our results may apply, though they all exhibit some specific difficulties, which is why we leave these applications for future work.

For a rotationally symmetric set $K\subset\R^3$, the drag of $K$ in the direction $\vec{e_x}$ is defined as
\begin{equation}\label{eq:drag}\mathcal{D}(K)=\inf\left\{\int_{\R^3\setminus K}\left|\frac{\nabla \vec{v}+(\nabla \vec{v})^*}{2}\right|^2,\vec{v}\in \vec{e_x}+H^1(\R^3,\R^3):\nabla\cdot \vec{v}=0,\ \vec{v}|_{\partial K}\equiv 0 \right\}.
\end{equation}
The minimization of $D(K)$ under the constraint $|K|=1$ has been widely studied, see for instance in \cite{P73,B74} the study of a conjectured optimal shape with two sharp ends of angular opening $\frac{2\pi}{3}$. This has also been considered in two dimensions in \cite{R95}, with an explicit conjectured optimal shape for the drag with two angles of opening $\approx 0.57\pi$ (which corresponds to $2\pi-t_1$ where $t_1$ is as defined in \eqref{eq_upper_density_estimate}: the reason for this precise angle is explained in the classification of $2$-homogeneous minimizers below).
Let us also mention the works  \cite{P08,BP13} for the minimization of the energy dissipated by a fluid following Stokes' equation in a pipe.

In the three-dimensional rotationally invariant case, we can define a stream function $u\in H^2_\loc(\R\times\R_{>0},\R)$ by $\vec{v}(x,y,0)=y^{-1}\nabla^\bot\left(y^\frac{1}{2}u\right)$, and it can be checked that the minimality of the drag $\mathcal{D}(K)$ under the constraint $|K|=|B|$ corresponds to the minimality of the energy
$$u\mapsto \int_{D}\left(|\Delta u|^2+\frac{3}{2}y^{-2}|\nabla u|^2-\frac{63}{16}y^{-4}u^2\right)$$
for any $D\Subset\R\times\R_{>0}$, under the constraint $2\pi\int_{\R^2}y\chi_{\nabla u\neq 0}=1$: this is a variant of the energy \eqref{eq_Stokes_scalar}, and we consider the study of \eqref{eq_energy} to be a first step towards understanding this free boundary problem.

\subsection{Main results}\label{ssect:results}

Let us first remind several known results from the literature, essentially based on \cite{DKV20,GM24} (although the reference \cite{DKV20} deals with a different free boundary problem, we will see that several tools developed in \cite{DKV20} find applications here).

We start by studying the regularity of $u\in \M(D)$. In \cite{GM24} it is proved that any minimizer $u$ of $E(\cdot\ ;D)$ is in $\mathcal{C}^{1,\alpha}_\loc(D,\R)$ for any $\alpha\in (0,1)$, and more precisely that
\[\Delta u\in \mathrm{BMO}_\loc(D,\R).\]
The expected optimal regularity for $u$ is $u\in\mathcal{C}^{1,1}_\loc(D,\R)$, which we will not obtain here.  Let us mention the work \cite{S16} where the regularity of the optimal shape for $\Om\mapsto\Lambda_1(\Om)|\Om|$ (which is essentially the same free boundary problem) is analyzed, with an argument for the $\mathcal{C}^{1,1}$ regularity that appears erroneous to the authors.

Then we note that a function $u$ is a solution in $D$ if and only if $x\mapsto r^{-2}u(rx)$ is also a solution in $D/r$, for any $r>0$, so the expected regularity of the boundary is linked to the study of $2$-homogeneous minimizers. As will be seen in subsection \ref{subsec_classification}, we prove that these $2$-homogeneous minimizers belong to one of the four following types:
\begin{itemize}
\item[(I) ]Flat solutions:
\begin{equation}\label{def_uI}
u=s u^{\mathrm{I}}\circ \mathrm{rot},\ u^{\mathrm{I}}\left(x,y\right)=\frac{y_+^2}{2},
\end{equation}
where $s\in\{-1,+1\}$ and $\mathrm{rot}$ is a rotation.
\item[(II) ]Angular solutions:
\begin{equation}\label{def_uII}
u=s u^{\mathrm{II}}\circ \mathrm{rot},\ u^{\mathrm{II}}\left(re^{i\theta}\right)=\frac{r^2}{4}\left(1-\cos(2\theta)-\frac{2}{t_1}\left(\theta-\frac{\sin(2\theta)}{2}\right)\right)\chi_{0\leq \theta\leq t_1},
\end{equation}
where $t_1\approx 1.43\pi$ is the unique fixed point of $\tan$ in $(\pi,2\pi)$, $s\in\{-1,+1\}$ and $\mathrm{rot}$ is a rotation.
\item[(III) ]Nodal solutions:
\begin{equation}\label{def_uIII}
u=\lambda u^{\mathrm{III}}\circ \mathrm{rot},\ u^{\mathrm{III}}\left(x,y\right)=\frac{y^2}{2},
\end{equation}
where $|\lambda|\geq 1$ and $\mathrm{rot}$ is a rotation.
\item[(IV) ]Isolated points:
\begin{equation}\label{def_uIV}
u(x,y)=u_{a,b,c}\left(x,y\right)=a (x^2+y^2)+b(x^2-y^2)+2cxy,
\end{equation}
for some $(a,b,c)\in\R^3$ such that $a^2\neq b^2+c^2$ ($a^2=b^2+c^2$ corresponds to the previous case).
\end{itemize}
Note that we do not claim each of these functions are minimizers, only that all homogeneous minimizers are among these. A detailed discussion on this issue may be found in Remark \ref{rem_minimality_blowup_conjecture}.

A central property of elements of $\M(\D_1)$ is the monotonicity formula discovered in a different context by Dipierro, Karakhanyan, and Valdinoci in \cite[Th. 1.12]{DKV20}, which may be seen as a higher-order analogue of the monotonicity formula introduced by Weiss in \cite{W99} for the classical Alt-Caffarelli problem. We define
\begin{equation}\label{eq_WDKV}
W(u,r)=\frac{1}{r^2}\int_{\D_r}\left\{|\Delta u|^2+\chi_{u\neq 0}\right\}+D(u,r),
\end{equation}
where
\[D(u,r)=\frac{1}{r}\int_{\partial \D_r}\left\{\frac{2\partial_r u\Delta u}{r}-\frac{10 (\partial_ru)^2}{r^2}-\frac{4u\Delta u}{r^2}+\frac{24u \partial_ru}{r^3}+\frac{4\partial_{\theta}u \partial_{r,\theta}u}{r^3}-\frac{16u^2}{r^4}-\frac{6(\partial_{\theta}u)^2}{r^4}\right\}\]
is a ``corrective term'' that is well-defined for almost every $r$. Then we show (see Theorem \ref{Th_Monotonicity_disk}) that $r\mapsto W(u,r)$ is continuous, nondecreasing, and it is constant if and only if $u$ is one of the $2$-homogeneous minimizers described above. Moreover, in this case we identify
\[W(u,r)=|\{u\neq 0\}\cap \D_1|\in \left\{\frac{\pi}{2},\frac{t_1}{2},\pi\right\},\]
the value $\frac{\pi}{2}$ corresponding to type I, $\frac{t_1}{2}$ corresponding to type II, and $\pi$ corresponding to types III and IV.
For a general minimizer $u$, we write $W(u,0)\in \R\cup\{-\infty\}$ the limit of $W(u,r)$ as $r\to 0$ (which always exists since $r\mapsto W(u,r)$ is nondecreasing).

When $u\in\M(D)$, it is known from \cite[Th. 2]{GM24} (see also section \ref{sec_prel_est} below) that $u$ and $\nabla u$ are continuous, and we define the support of $u$ as the open set
$$ \Spt(u)=\{u\neq 0\}\cup\{\nabla u\neq 0\}.$$
Defining the support of a function as an open set may seem  non-standard, however it is necessary to capture some part of the free boundary that would otherwise be in the interior of the support. Our main result is as follows:

\begin{theorem}\label{mr_total}
Let $D$ be an open set of $\R^2$ and $u\in\M(D)$. Then $u\in\mathcal{C}_\loc^{1,\alpha}(D)$ for every $\alpha\in (0,1)$ and  there is a partition
$$D\cap\partial \Spt(u)=\mathcal{R}_u\sqcup\mathcal{A}_u\sqcup\mathcal{N}_u\sqcup\mathcal{J}_u\sqcup\mathcal{E}_u,$$
where 
\begin{itemize}[label=\textbullet]
\item $\mathcal{R}_u$ (\textit{regular} points) is the regular boundary of $\Spt(u)$, meaning every point for which $\Spt(u)$ is (after rotation) the epigraph of an analytic function in a neighbourhood, $\mathcal{R}_u$ is relatively open in $\partial\Spt(u)$, $u|_{\Spt(u)}$ is analytic up to $\mathcal{R}_u$, and the trace of $\Delta u|_{\Spt(u)}$ on $\mathcal{R}_u$ takes values in $\{-1,+1\}$.
\item $\mathcal{A}_u$ (\textit{angular} points) is a discrete set where two connected components of $\mathcal{R}_u$ (where $\Delta u|_{\mathcal{R}_u}$ takes opposite signs) join with an angular opening $t_1$.
\item $\mathcal{N}_u$ (\textit{nodal} points) is the set of points of $\partial\Spt(u)$ such that $u$ is biharmonic in a neighbourhood. It is a disjoint union of singletons, and of real analytic curves in $\partial\Spt(u)$ on which $|\Delta u|\geq 1$.
\item $\mathcal{J}_u$ (\textit{junction} points) is the set of points $p$, such that $r^{-2}u(p+r\cdot)$ converges as $r\to 0$ in $H^2_\loc(\R^2)$ to a homogeneous solution of type III and such that for any $r>0$, we have $|\{u=0\}\cap\D_{p,r}|>0$.
\item $\mathcal{E}_u$ (\textit{explosion} points) is a set of points $p\in\partial\Spt(u)$ such that $$\limsup_{r\to 0}\frac{|\{u=0\}\cap \D_{p,r}|}{\pi r^2}=0,\quad W(u(p+\cdot),0)<\pi,\quad \Vert u_{p,r}\Vert_{H^1(\D_1)}\underset{r\to 0}{\longrightarrow}+\infty,$$
where $u_{p,r}:=r^{-2}u(p+r\cdot)$. Moreover, any accumulation point of $\frac{u_{p,r}}{\Vert u_{p,r}\Vert_{H^1(\D_1)}}$ in $H^1(\D_1)$ as $r\to 0$ is a non-zero $2$-homogeneous biharmonic polynomial.
\end{itemize}
\end{theorem}
More details on the behavior of $u$ at points of $\mathcal{A}_u$ and $\mathcal{J}_u$ may be found in Theorems \ref{mr_typeII} and \ref{mr_typeIII_IV}.
In order to close the regularity question in this context, the two remaining open problems are:
\begin{itemize}
    \item Whether the explosion set $\mathcal{E}_u$ is empty, which is strongly linked to the conjectured $\mathcal{C}^{1,1}$ regularity of $u$.
    \item A finer understanding of the junction set $\mathcal{J}_u$. A key step here would be to first prove the blow-up of $u$ at a point of $\mathcal{J}_u$ is necessarily $\pm u^\mathrm{III}\circ \mathrm{rot}$ for some rotation $\mathrm{rot}$ (in other words, proving that $|\lambda|=1$ in \eqref{def_uIII}).
\end{itemize}

The proof of Theorem \ref{mr_total} can be found in section \ref{sec_reg}, and is split into several results: Theorem \ref{mr_blowup} gives the existence and uniqueness of blow-ups at any boundary point that is not of density $1$ in the support of $u$, or at points for which $W(u,0)$ is sufficiently large, and it provides a speed of convergence to these blow-ups as well. Then Theorems  \ref{mr_typeI},  \ref{mr_typeII} and \ref{mr_typeIII_IV} are $\eps$-regularity results for minimizers that are assumed to be close to one of the homogeneous solutions: for type I and II we prove that any solution that is close enough to these is itself a smooth perturbation of a homogeneous solution, whereas for type III and IV we prove a quantified speed of convergence of the blow-up sequence.

\begin{theorem}[Existence and uniqueness of blow-ups]\label{mr_blowup}
There exists $\gamma\in (0,1)$ such that the following holds: let $D\subset\R^2$ be an open set and $u\in\M(D)$. Let $p\in D\cap\partial  \Spt(u)$ such that one of the following hypothesis is verified:
\begin{itemize}
\item[(i) ]$\liminf_{r\to 0}\frac{|\Spt(u)\cap\D_{p,r}|}{\pi r^2}<1$.
\item[(ii) ]$W(u(p+\cdot),0)\geq \pi$.
\end{itemize}
Define the blow-up sequence $u_{p,r}(z)=\frac{u(p+rz)}{r^2}$, then there exists $u_{p,0}\in H^2_\loc(\R^2)$ a $2$-homogeneous minimizer of type I or II in case (i), III or IV in case (ii), such that
$$\Vert u_{p,r}-u_{p,0}\Vert_{H^1(\D_1)}=\mathcal{O}_{r\to 0}\left(r^\gamma\right).$$
\end{theorem}

In particular, this implies that any point that is not of Lebesgue density $1$ in $\Spt(u)$ has density $\frac{1}{2}$ (corresponding to type I blow-ups) or $\frac{t_1}{2\pi}$ (corresponding to type II blow-ups).\\

We now state the first $\eps$-regularity theorem: any minimizer that is $H^1$-close to a flat solution is itself smooth.
\begin{theorem}[$\eps$-regularity : flat case]\label{mr_typeI}
There exist $\alpha,\kappa\in (0,1)$, $\eps,C>0$ with the following property: for any $u\in\M(\D_1)$ such that
\[\Vert u-u^{\mathrm{I}}\Vert_{H^1(\D_1)}\leq\eps,\]
 there exists an analytic function $h:\left(-\frac{1}{2},\frac{1}{2}\right)\to \left(-\frac{1}{2},\frac{1}{2}\right)$ such that
$$ \Vert h\Vert_{\mathcal{C}^{1,\alpha}\left(\left(-\frac{1}{2},\frac{1}{2}\right)\right)}\leq C\Vert u-u^{\mathrm{I}}\Vert_{H^1(\D_1)}^\kappa$$
and 
$$\D_{\frac{1}{2}}\cap\Spt(u)=\left\{(x,y)\in\D_{\frac{1}{2}}:y> h(x)\right\}.$$
\end{theorem}

We now state the second $\eps$-regularity theorem at angular boundary points: one notable difference in the statement is that we now assume that $W(u,0)\geq \frac{t_1}{2}$. This hypothesis is verified for instance when there exists one blow-up of type other than I at the origin, by Lemma \ref{lem_conv_blowup}.
\begin{theorem}[$\eps$-regularity : angular case]\label{mr_typeII}
There exist $\nu,\mu\in (0,1)$, $\eps,C>0$ with the following property: for any $u\in\M(\D_1)$ such that
\[W(u,0)\geq\frac{t_1}{2},\qquad \Vert u-u^{\mathrm{II}}\Vert_{H^1(\D_1)}\leq\eps,\]

then there exists a diffeomorphism $\Phi\in\mathcal{C}^{1,\gamma}\left(\D_{\frac{1}{2}},\D_{\frac{1}{2}}\right)$ such that
$$\Phi(0)=0,\ D\Phi(0)=I_2,\ \Vert \Phi-\mathrm{id}\Vert_{\mathcal{C}^{1,\nu}\left(\D_{\frac{1}{2}}\right)}\leq C\Vert u-u^{\mathrm{II}}\Vert_{H^1(\D_1)}^\mu$$
and
$$\Spt(u)\cap\D_{1/2}=\Phi\left(\Spt(u^{\mathrm{II}})\cap\D_{1/2}\right),$$
where $\Spt(u^{\mathrm{II}})$ is the union of two half-lines meeting with the angle $t_1$.
\end{theorem}

The last $\eps$-regularity result concerns homogeneous solutions of type III and IV (which we do not differentiate at this stage), that we remind are of the form
$$u_{a,b,c}(x,y)=a(x^2+y^2)+b(x^2-y^2)+2cxy$$
for some $(a,b,c)\in\R^3$. The main feature of this case (compared to Theorem \ref{mr_typeI} and \ref{mr_typeII}) is that the set of homogeneous solutions of type III, IV is non-compact.
\begin{theorem}[$\eps$-regularity : nodal and isolated case]\label{mr_typeIII_IV}
There exist $\eps,\gamma,C>0$ such that the following holds: let $a,b,c\in\R$ and $u\in\M(\D_1)$ such that
$$0\in\partial\Spt(u),\quad W(u,0)\geq\pi,\quad \Vert u-u_{a,b,c}\Vert_{H^1(\D_1)}\leq \eps \Vert u_{a,b,c}\Vert_{H^1(\D_1)}.$$
Then there exists $a',b',c'\in\R$ such that
$$|a'-a|+|b'-b|+|c'-c|\leq C\Vert u-u_{a,b,c}\Vert_{H^1(\D_1)}^\gamma\Vert u_{a,b,c}\Vert_{H^1(\D_1)}^{1-\gamma}$$
and such that for any $r\in (0,1]$:
$$\Vert r^{-2}u(r\cdot)-u_{a',b',c'}\Vert_{H^1(\D_1)}\leq C\min\left(r,\frac{\Vert u-u_{a,b,c}\Vert_{H^1(\D_1)}}{\Vert u_{a,b,c}\Vert_{H^1(\D_1)}}\right)^\gamma\Vert u_{a,b,c}\Vert_{H^1(\D_1)}.$$
\end{theorem}

Finally, we state a generalization of Theorem \ref{mr_total} to a notion a quasi-minimizer.

\begin{theorem}\label{mr_total_quasi}
Let $D\subset\R^2$ be an open set, and $u\in H^2(D,\R)$ such that for any $v\in H^2(D,\R)$ with $\{u\neq v\}\Subset D$ we have
$$\int_{D}\left(|\Delta u|^2+\chi_{u\neq 0}\right)\leq\int_{D}\left(|\Delta v|^2+\chi_{v\neq 0}\right)+\Vert v-u\Vert_{L^2(D)}.$$
Then $u\in \mathcal{C}_\loc^{1,\gamma}(D)$ for every $\gamma\in (0,1)$, and there is a partition
$$D\cap\partial \Spt(u)=\mathcal{R}_u\sqcup\mathcal{A}_u\sqcup\mathcal{N}_u\sqcup\mathcal{J}_u\sqcup\mathcal{E}_u,$$
where 
\begin{itemize}
    \item 
$\mathcal{R}_u$ is the regular boundary of $\Spt(u)$, meaning every point for which $\Spt(u)$ is the epigraph of a $\mathcal{C}^{1,\alpha}$ function in a neighbourhood,  $\mathcal{R}_u$ is relatively open in $\partial\Spt(u)$, and $\Delta u|_{\Spt(u)}$ is continuous up to $\mathcal{R}_u$ with value $\pm 1$,
\item the sets $\mathcal{A}_u$, $\mathcal{N}_u$, $\mathcal{J}_u$, $\mathcal{E}_u$ satisfy the same conclusions as in Theorem \ref{mr_total}.
\end{itemize}
\end{theorem}
The discussion of this notion of quasi-minimizer as well as applications to the buckling problem can be found in section \ref{sec_quasireg}.

\subsection{Notations and outline of the paper}\label{subsec_notations}

In all the paper, we make use of the following notation: we write $a\lesssim b$ when $a,b$ are two quantities such that $b>0$ and $a\leq Cb$ for some universal constant $C>0$, i.e. some constant that does not depend on any parameter.

We denote $\D_r$  the centered disk of radius $r$ in $\R^2$, and $\D_{p,r}=p+\D_r$ for $p\in\R^2$. We also let $\Sp=\partial\D_1$, that we identify with $\frac{\R}{2\pi\Z}$. 

For any $k\in\N$, $q\in [1,+\infty[$, any smooth bounded open set $D\subset\R^2$, and any measurable function $u:D\to \R$ with $k$ derivative in the distributional sense we let by convention
$$\Vert u\Vert_{W^{k,q}(D)}=\left(\sum_{\alpha\in \N^2 : \alpha_1+\alpha_2\leq k}\int_{D}|\partial_x^{\alpha_1}\partial_y^{\alpha_2} u|^q\right)^{\frac{1}{q}}.$$
We denote $H^k(D)=W^{k,2}(D)$, and $H^k_0(D)$ the closure of $\mathcal{C}^\infty_c(D,\R)$ functions with respect to the $H^k(D)$ norm.

For a function $u:D\to\R$, $p\in D$, $r>0$, we let
\begin{equation}\label{eq_defupr}
u_{p,r}(z)=\frac{u(p+rz)}{r^2}.
\end{equation}
the associated rescaling defined on $\frac{D-p}{r}$. When $p=0$ we simply write $u_r:=u_{0,r}$.
In particular, if $u$ is a local minimizer of $E(\cdot;D)$ (which we remind is written $u\in\M(D)$), then $u_{p,r}\in\M\left(\frac{D-p}{r}\right)$.

In all the paper, when $u$ is a $\mathcal{C}^1$ function, we denote by $\Spt(u)$ the set of points $p$ where either $u(p)$ or $\nabla u(p)$ is non-zero.

The plan of the paper is as follows.
\begin{itemize}[label=\textbullet]
\item In section \ref{sec_prel_est}, we prove several low regularity estimate on $u$ that will be repeatedly used later. This includes a log-lipschitz bound on $\nabla u$ depending on the $H^1$ norm of $u$ in a neighbourhood, as well as a nondegeneracy lemma that states that a minimizer cannot be ``too  small'' near a point of the support. The methods of proofs are relatively standard with respect to the classical Alt-Caffarelli problem, based on \cite{AC81}, \cite{DT15}, \cite{V23}. The main difference is that there is no maximum principle nor notion of viscosity solutions in our setting.
\item In section \ref{sec_mono} we introduce and prove the monotonicity formula for the renormalized energy $W(u,r)$ (defined in \eqref{eq_WDKV}) and several of its consequences, namely a classification of $2$-homogeneous minimizers as well as a proof of Theorem \ref{mr_blowup} about existence of blow-ups. Even though the monotonicity formula was already found in \cite{DKV20}, we propose a new way to tackle the computation, which most importantly leads to a new interpretation of the corrective term. The proof is based on the criticality of the solution for the energy, with respect to inner variation (i.e. comparison of $E(u;\D_1)$ with $E(u\circ (\mathrm{id}+t\zeta);\D_1)$ at $t\to 0$, for any vector field $\zeta\in\mathcal{C}^\infty_c(\D_1,\R^2)$) similarly to \cite{W98}. The monotonicity implies in particular a maximal speed of growth of the $H^1(\D_r)$ norm of $u$ with respect to $r$ (given by Lemma \ref{lem_boundedgrowth}). This property is central in the proof of existence of blow-ups under weak geometrical hypothesis (being a non-Lebesgue point of the support of $u$).
\item In section \ref{sec_epi} we prove an epiperimetric inequality for the renormalized energy $W$, that takes the form
$$W(u,e^{-1})\leq (1-\eta)W(u,1)+\eta\frac{\Theta}{2}$$
when $u$ is sufficiently close to a blow-up of opening $\Theta\in\{\pi,t_1,2\pi\}$. This is obtained by constructing, for almost every $r$ in the interval $[e^{-1},1]$, a competitor $v$ for $u$ in $\D_r$ such that $v-u\in H^2_0(\D_r)$ while the energy of $v$ is well-controlled by the boundary data $u|_{\partial\D_r}$, $\partial_r u|_{\partial\D_r}$.
\item In section \ref{sec_reg} we prove Theorems  \ref{mr_blowup}, \ref{mr_typeI}, \ref{mr_typeII}, and \ref{mr_typeIII_IV}. This is obtained by iterating the epiperimetric inequality to deduce geometrical information on the boundary (see Proposition \ref{prop_convpolyblow_up}). The main difficulty compared to the classical Alt-Caffarelli problem is that blow-ups are only known to exist under some a priori condition, either on the density of the support or on a lower bound of $W(u,r)$: we obtain the $\mathcal{C}^{1,\alpha}$ regularity of the boundary near type I blow-ups by a careful analysis of contact points of the boundary with a cone from outside the support. Higher regularity is obtained by a conformal hodograph transform on the associated overdetermined Stokes equation. The combination of these results implies Theorem \ref{mr_total}.
\item In section \ref{sec_quasireg}, we define a suitable notion of quasi-minimizers of $E(\cdot;D)$ and prove Theorem \ref{mr_total_quasi}. We then apply these results to the minimization of the first buckling eigenvalue \eqref{expr_Lambda1}. Our quasiminimality is similar to the quasiminimality condition from \cite{MTV17}; a central point is that a suitable quasiminimality condition must be sufficiently strong to extract information from small inner variations, which is not the case for weaker quasiminimality conditions (like the one found in \cite{DT15} for classical Alt-Caffarelli problem). The reason for this is that the monotonicity formula is proved via inner variations of $u$, and it is unclear whether it can be accessed differently.
\end{itemize}

\section{Preliminary estimates}\label{sec_prel_est}

In this section we state and prove several regularity estimates on minimizers: a Cacciopoli-type estimate (Lemma \ref{lem_Cacciopoli}) to show that the $H^1$ norm of $u$ locally controls its energy, a $\mathrm{BMO}$ estimate (Lemma \ref{lem_BMO}) on $\Delta u$ that implies the gradient of $u$ is almost Lipschitz (Lemma \ref{lem_C1log}), a non-degeneracy lemma (Lemma \ref{lem_nondegeneracy}) that shows a minimizer always carry some minimal threshold of energy near the boundary of the support, and finally a lemma showing that a bounded sequence of minimizers converges (after extraction) to some minimizer in a strong way. The first two results already appear in some form in \cite{GM24} or \cite{S16} and our proofs in this section are not original in this regard, however we carry them here with a more explicit statement for later uses.

In this section, we will need to consider the following slightly more general version: for any $\lambda\geq 0$ we let $\M_\lambda(D)$ be the minimizers (in the sense given in the beginning of the introduction) of the energy
$$u\in H^2(D)\mapsto E_\lambda (u;D):=\int_{D}\left(|\Delta u|^2+\lambda\chi_{u\neq 0}\right).$$

We start by a remark on our definition of minimality, showing that it is equivalent to consider compactly supported perturbations, or perturbations up to the boundary with fixed traces of $u,\nabla u$.
\begin{lemma}
Let $u\in H^2(\D_1)$, $\lambda\geq 0$, then $u\in\M_\lambda(\D_1)$ if and only if for every $v\in u+H^2_0(\D_1)$ we have $E_\lambda(u;\D_1)\leq E_\lambda(v;\D_1)$.
\end{lemma}
\begin{proof}
Assume $u\in \M(\D_1)$, let $w\in H^2_0(\D_1)$, $r\in (0,1)$, and $\eta\in\mathcal{C}^\infty_c(\D_1)$ such that $\eta\equiv 1$ on $\ov{\D_r}$. Then $(1-\eta)w\in H^2_0(\D_1\setminus\ov{\D_r})$, so there exists some sequence $\chi_n\in\mathcal{C}^\infty_c(\D_1\setminus\ov{\D_r})$ that converges to $(1-\eta)w$ in $H^2(\D_1\setminus\D_r)$. Since $u\in\M(\D_1)$, we have $E_\lambda(u;\D_1)\leq E_\lambda(u+\eta w+\chi_n;\D_1)$, and
$$\limsup_{n\to +\infty} E_\lambda(u+\eta w+\chi_n;\D_1)\leq E_\lambda(u+w;\D_1)+\lambda\pi(1-r^2)$$
by the $H^2$ convergence of $\chi_n$ to $(1-\eta)w$. We conclude by taking $r$ arbitrarily close to $1$.
\end{proof}

We then prove a Cacciopoli-type estimate: in this result we only use the fact that $u$ is biharmonic on its support, with a competitor of the form $e^{t\phi}u$ for some smooth function $\phi$ with compact support.
\begin{lemma}\label{lem_Cacciopoli}
Let $r\in \left[\frac{1}{2},1\right)$, let $\lambda\geq 0$ and $u\in \M_\lambda(\D_1)$. Then
\[\Vert  u\Vert_{H^2(\D_r)}\lesssim \frac{1}{(1-r)^2}\Vert u\Vert_{H^1(\D_1\setminus\D_r)}.\]
\end{lemma}
Note that this estimate does not depend on $\lambda$.
\begin{proof}
For this proof, it is more convenient to see $u$ as a local minimizer of $u\mapsto \int_{\D_1}\left(|\nabla^2u|^2+\lambda\chi_{u\neq 0}\right)$ (see remark \ref{rem:ipp}). Let $\eta\in\mathcal{C}^{\infty}_c(\D _1,[0,1])$ to be fixed later. The optimality condition on $u$ gives

\[\left.\frac{d}{dt}\right|_{t=0}\int_{\D_1}\left(|\nabla^2(e^{\eta^4 t}u)|^2+\lambda\chi_{e^{\eta^4 t}u\neq 0}\right)=0,\]
which after computations develops as $\int_{\D_1}\nabla^2(\eta^4 u):\nabla^2 u=0$, with
$$\nabla^2(\eta^4 u)=\eta^4\nabla^2 u+8\eta^3\nabla\eta\otimes \nabla u+4\left(3\eta^2\nabla\eta\otimes\nabla\eta+\eta^3\nabla^2\eta\right)u.$$
The previous relation becomes
\begin{align*}
\int_{\D _1}\eta^4|\nabla^2  u|^2&=-\int_{\D _1}\eta^2\nabla^2 u : \left(8\eta\nabla\eta\otimes\nabla u+4\left(3(\nabla\eta\otimes \nabla \eta)+\eta\nabla^2\eta\right)u\right)\\
&\leq \frac{1}{2}\int_{\D_1}\eta^4|\nabla^2u|^2+\frac{1}{2}\int_{\D_1}\left|8\eta\nabla\eta\otimes\nabla u+4\left(3(\nabla\eta\otimes \nabla \eta)+\eta\nabla^2\eta\right)u\right|^2,
\end{align*}

which simplifies to
\[\int_{\D _{1}}\eta^4|\nabla^2 u|^2\lesssim\int_{\D _1}\left(\left[|\nabla\eta|^4+|\eta\nabla^2\eta|^2\right]| u|^2+\eta^2|\nabla\eta|^2|\nabla  u|^2\right).\]
Now, choosing an appropriate profile $\eta$ equal to $1$ in $\D_r$, with $|\nabla^k \eta|\lesssim (1-r)^{-k}$ for $k=0,1,2$, we obtain $\Vert\nabla^2 u\Vert_{L^2(\D_r)}\lesssim (1-r)^{-2}\Vert u\Vert_{H^1(\D_1\setminus\D_r)}$.

The bound on the full $H^1(\D_r)$ norm is then obtained by the general inequality 
$$\Vert v\Vert_{H^2(\D_r)}^2\lesssim\Vert \nabla^2 v\Vert_{L^2(\D_r)}^2+\Vert v\Vert_{L^2(\partial\D_r)}^2\lesssim \Vert \nabla^2 v\Vert_{L^2(\D_r)}^2+\frac{1}{1-r}\Vert v\Vert_{H^1(\D_1\setminus\D_r)}^2,$$
for any $v\in H^2(\D_r)\cap H^1(\D_1)$, applied to $v=u$. 
\end{proof}
Next we prove an estimate on the BMO semi-norm of $\nabla^2 u$. We remind our definition of this seminorm on a smooth set $\Om$: a function $f\in L^1_\loc(\Om)$ is in $\mathrm{BMO}(\Om)$ if for every disk $\D_{p,r}\subset\Om$ we have
$$\fint_{\D_{p,r}}\left|f-\fint_{\D_{p,r}}f\right|^2\leq M^2$$
for some constant $M>0$, and the best constant $M$ is denoted $[f]_{\mathrm{BMO}(\Om)}$. Here $\fint_{U} g:=\frac{\int_{U}g}{|U|}$ by convention.

We recall that the BMO space embeds in every $L^p$ space, for $p\in [1,+\infty)$, by the John-Nirenberg Lemma \cite[Lem. 1]{JN61}, but not in $L^\infty(\Om)$.
\begin{lemma}\label{lem_BMO}
Let $ u\in\M(\D_1)$, then
\[[\nabla^2  u]_{\mathrm{BMO(\D _{1/2})}}\lesssim 1+\Vert u\Vert_{H^2(\D_1)}.\]
\end{lemma}

\begin{proof}
In this proof we will only use the following property: if $ v$ is the biharmonic extension of $ u$ on some open set $\Om\subset\D_1$ - meaning $u-v\in H_0^2(\Om)$ and $\Delta^2 v=0$ in $\Om$ - then by minimality of $u$ we have
\[\fint_\Om|\Delta( u- v)|^2=\frac{1}{|\Om|}\int_{\Om}\left(|\Delta u|^2-|\Delta v|^2\right)\leq 1.\]
We recall (see for instance \cite[Lem. 1.41]{HL11}) that for any harmonic function $h:\D_{R}\to\R$ and $r<R$ we have the decay property
\begin{equation}\label{eq_decay}
\fint_{\D _{r}}\left|h-\fint_{\D _{r}}h\right|^2\lesssim \left(\frac{r}{R}\right)^{2}\fint_{\D _{R}}\left|h-\fint_{\D _{R}}h\right|^2.
\end{equation}
Let now $\D _{p,r}\subset \D _1$, $ v$ the biharmonic extension of $ u$ in $\D _{p,r}$ and $\tau\in (0,1)$: then we have 
\begin{align*}
\fint_{\D _{p,\tau r}}|\Delta  u-\langle \Delta  u\rangle_{p,\tau r}|^2&\lesssim \fint_{\D _{p,\tau r}}\left(|\Delta( u- v)|^2+|\langle\Delta  u\rangle_{p,\tau r}-\langle \Delta  v\rangle_{p,\tau r}|^2+|\Delta  v-\langle \Delta  v\rangle_{p,\tau r}|^2\right)\\
&\lesssim \tau^{-2}\fint_{\D _{p,r}}|\Delta( u- v)|^2+\tau^2\fint_{\D _{p,r}}|\Delta v-\langle \Delta v\rangle_{p,r}|^2\text{ by \eqref{eq_decay}}\\
&\lesssim \tau^{-2}+\tau^2\fint_{\D _{p,r}}\left(|\Delta( u- v)|^2+|\langle\Delta  u\rangle_{p, r}-\langle \Delta  v\rangle_{p, r}|^2+|\Delta  u-\langle \Delta  u\rangle_{p, r}|^2\right)\\
&\lesssim \tau^{-2}+\tau^2\fint_{\D _{p,r}}|\Delta  u-\langle \Delta  u\rangle_{p, r}|^2.
\end{align*}
As a consequence, letting $g(p,r):=\fint_{\D _{p,r}}|\Delta  u-\langle \Delta  u\rangle_{p, r}|^2$, we may fix a sufficiently small $\tau\in (0,1)$ such that for some (universal) constant $C\gtrsim 1$:
\[g(p,\tau r)\leq C+\frac{1}{2}g(p,r).\]
By induction on $g(p,\tau^k r)$, we obtain that for any $x\in \D _{\frac{3}{4}}$ and for any $r\in \left(0,\frac{1}{4}\right)$ we have 
$$g(p,\tau^kr)\leq \left(1+\frac{1}{2}+\hdots+\frac{1}{2^{k-1}}\right) C+\frac{1}{2^k}g(p,r).$$
Thus, for any $r\in \left(0,\frac{1}{4}\right)$, we have
\[\fint_{\D _{p,r}}|\Delta  u-\langle \Delta  u\rangle_{p, r}|^2\lesssim 1+\int_{\D _1}|\Delta  u|^2,\]
Hence $[\Delta u]_{\mathrm{BMO}(\D_{\frac{3}{4}})}\lesssim 1+\Vert \Delta u\Vert_{L^2(\D_1)}$. By \cite[Cor. 2]{FS72}, this implies $[\nabla^2 u]_{\mathrm{BMO}(\D_{\frac{1}{2}})}\lesssim 1+\Vert u\Vert_{H^2(\D_1)}$.

\end{proof}

\begin{lemma}\label{lem_C1log}
Let $ u\in\M(\D_1)$, then for any $p,q\in \D _{1/2}$:
\[|\nabla u(p)-\nabla u(q)|\lesssim |p-q|\log\left(\frac{1}{|p-q|}\right)\left(1+\fint_{\D _1}|\nabla u|^2+ u^2\right)^\frac{1}{2}.\]
\end{lemma}
\begin{proof}
This is a consequence of Lemmas \ref{lem_Cacciopoli}, \ref{lem_BMO}, to which we apply general Orlicz space embeddings (see for instance \cite[Th. 3]{C96}).
\end{proof}

\begin{remark}
This lemma implies that when $u\in\M(\D_1)$ with $u(0)=|\nabla u(0)|=0$, then $$\forall p\in\D_{1/2},\ |\nabla u(p)|+\frac{|u(p)|}{|p|}\lesssim \left(1+\Vert u\Vert_{H^1(\D_1)}\right)|p|\log\left(\frac{1}{|p|}\right).$$
\end{remark}

The non-degeneracy property is central in the classical Alt-Caffarelli problem (see \cite[Lem. 3.4]{AC81}): it implies that blow-ups, when they exist, are not zero. The next lemma is a non-degeneracy property for minimizers of \eqref{eq_energy}, and our proof is similar to the proof of non-degeneracy in \cite{DT15}, by an induction on the energy of $u$ relying on the three following estimates: 
\begin{itemize}
\item[a) ]The $H^2$ norm of $u$ is locally controlled by its $H^1$ norm by our Cacciopoli-type estimate.
\item[b) ]The area of the support of $u$ is locally controlled by its $H^2$ norm, by minimality of $u$.
\item[c) ]The $H^1$ norm of $u$ is locally controlled by a \textbf{product} of the area of its support and its $H^2$ norm. 
The nonlinear nature of this estimate is what makes the induction work.
\end{itemize}

\begin{lemma}\label{lem_nondegeneracy}
There exists $\eps>0$ such that for any $u\in\M(\D_1)$, if $\Vert u\Vert_{H^1(\D_1)}\leq \eps$, then $u|_{\D_{1/2}}\equiv 0$.
\end{lemma}
\begin{proof}
It is enough to prove that for any minimizer in $\D _1$ and for a small enough $\eps$,
\[\int_{\D _1}| u|^2+|\nabla  u|^2 \leq \eps\quad\text{ implies }\quad u(0)=0,\]
because this can be applied to $u_{p,1/2}$ for every $p\in\D_{1/2}$ (the definition of $u_{p,r}$ is given in \eqref{eq_defupr}).
\begin{itemize}[label=\textbullet]
\item $\int_{\D _1}|\nabla u|^2+| u|^2 \leq \eps$ implies $E(u;\D _{1/2})\lesssim \eps$. Indeed, by the Cacciopoli-type inequality from Lemma \ref{lem_Cacciopoli}, we have 
\[\int_{\D _{\frac{9}{10}}}|\Delta  u|^2\leq \|u\|^2_{H^2(\D_{\frac{9}{10}})}\lesssim \int_{\D _1\setminus\D_{\frac{9}{10}}}\left(| u|^2+|\nabla  u|^2\right)\lesssim  \eps.\]
Let now $\eta\in\mathcal{C}^\infty_c(\D _{\frac{9}{10}},[0,1])$ such that $\eta\equiv 1$ in $\D _{\frac{1}{2}}$ with $\Vert\eta\Vert_{\mathcal{C}^2(\D_1)}\lesssim 1$, we then compare the energies of $ u$ and $(1-\eta)  u$:
\[|\D _{1/2}\cap\{ u\neq 0\}|+\int_{\D _1}|\Delta u|^2\leq \int_{\D _1}|\Delta( u-\eta u)|^2.\]
Since $u$ is biharmonic on its support (see \cite[Theorem 2]{GM24}) we have
\[\int_{\D _1}|\Delta(u-\eta u)|^2=\int_{\D _1}|\Delta u|^2+|\Delta(\eta u)|^2,\]
so
\begin{align*}
|\D _{\frac{1}{2}}\cap\{ u\neq 0\}|&\leq \int_{\D _1}|\Delta(\eta u)|^2\lesssim \Vert u\Vert_{H^2(\D_{\frac{9}{10}})}^2\lesssim \eps,
\end{align*}

and thus $E(u;\D_{\frac{1}{2}})\lesssim \eps$.

\item $E( u;\D _{1})+\int_{\D _1}(|\nabla u|^2+| u|^2)\leq \eps$ implies $\int_{\D _{\frac{1}{2}}}(|\nabla u|^2+| u|^2)\lesssim \eps^{\frac{3}{2}}$.

Indeed, let $\eta$ be defined as previously, then 
\begin{align*}
\int_{\D _{\frac{1}{2}}}|\nabla  u|^2&\leq \int_{\D _1}|\nabla(\eta u)|^2\leq |\{\eta u\neq 0\}|^{\frac{1}{2}}\left(\int_{\D _1}|\nabla(\eta u)|^{4}\right)^{1/2}\\
&\lesssim  E( u;\D _1)^{\frac{1}{2}}\int_{\D _1}|\Delta(\eta u)|^{2}\text{ by the Sobolev inclusion }W^{2,2}_0(\D_1)\hookrightarrow W^{1,4}_0(\D_1)\\
&\lesssim E( u;\D _1)^{\frac{1}{2}}\left(E( u;\D _1)+\int_{\D _1}(|\nabla u|^2+| u|^2)\right)\\
&\lesssim \eps^{\frac{3}{2}}
\end{align*}
and similarly 
$$\int_{\D _{\frac{1}{2}}} u^2\leq \int_{\D_1}(\eta u)^2\leq|\{\eta u\neq 0\}|^\frac{1}{2}\left(\int_{\D_1}(\eta u)^4\right)^\frac{1}{2}\lesssim \eps^\frac{3}{2}.$$
\end{itemize}
We now proceed by induction. We let $ u_r(x):=\frac{ u(rx)}{r^2}$. Then for some constant $C_1\gtrsim 1$ we have (by the first point above):
$$E(u_{\frac{1}{2}};\D_1)+\Vert u_{\frac{1}{2}}\Vert_{H^1(\D_1)}^2\leq C_1\Vert u\Vert_{H^1(\D_1)}^2.$$
By the second point applied to $u_{\frac{1}{2}}$, for some constant $C_2\gtrsim 1$ we have
$$\Vert u_{\frac{1}{4}}\Vert_{H^1(\D_1)}^2\leq C_2\left(C_1\Vert u\Vert_{H^1(\D_1)}^2\right)^\frac{3}{2}.$$
In particular suppose that $\Vert u\Vert_{H^1(\D_1)}^2\leq \frac{1}{4}C_1^{-3}C_2^{-2}$, then we obtain by induction
$$\forall k\in\N,\ \Vert u_{4^{-k}}\Vert_{H^1(\D_1)}^2\leq 2^{-k}\Vert u\Vert_{H^1(\D_1)}^2.$$
This implies $ u(0)=0$.
\end{proof}

Finally, we prove a result about sequences of minimizers. 
\begin{lemma}\label{lem_sequence}
Let $(\lambda^{(n)})_{n\in\N^*}$ be a sequence of nonnegative numbers converging to some limit $\lambda\geq 0$, and $(u^{(n)})_{n\in\N^*}$ a sequence of $H^2(\D_1)$ such that for any $n\in\N^*$, $u^{(n)}\in \M_{\lambda^{(n)}}(\D_1)$. Suppose moreover that
$$\sup_{n\in\N^*} \Vert u^{(n)}\Vert_{H^1(\D_1)}<\infty.$$
Then there exists some extraction $(n_i)_{i\in\N^*}$, and some $u\in\M_\lambda(\D_1)$ such that
\begin{align*}
u^{(n_i)}&\underset{i\to +\infty}{\longrightarrow}u\text{ in }H^2_\loc(\D_1),\\
\lambda^{(n_i)}\chi_{u^{(n_i)}\neq 0}&\underset{i\to +\infty}{\longrightarrow}\lambda\chi_{u\neq 0}\text{ in }L^1(\D_1).
\end{align*}
\end{lemma}
\begin{proof}The proof is similar to its equivalent in the Alt-Caffarelli problem and is fairly standard (see for instance \cite[Lem. 6.3]{V23}).
Since $(u^{(n)})_n$ is bounded in $H^1(\D_1)$, by the Cacciopoli inequality from Lemma \ref{lem_Cacciopoli} it is bounded in $H^2(\D_r)$ for every $r\in (0,1)$. Up to extraction we may suppose that for every $r\in (0,1)$, $(u^{(n)})_{n\in\N^*}$ converges strongly in $H^1(\D_r)$, weakly in $H^2(\D_r)$, as well as almost everywhere, to some limit $u$. By the weak $H^2$ convergence we have
$$\int_{\D_1}|\Delta u|^2\leq \liminf_{n\to+\infty}\int_{\D_1}|\Delta u^{(n)}|^2.$$

Since $u^{(n)}$ converges almost everywhere to $u$, then by Fatou lemma:
$$\int_{\D_1}\chi_{u\neq 0}\leq \int_{\D_1}\liminf_{n\to\infty}\chi_{u^{(n)}\neq 0}\leq\liminf_{n\to\infty}\int_{\D_1}\chi_{u^{(n)}\neq 0}. $$

Thus, ${E_\lambda}(u;\D_1)\leq\liminf_{n\to\infty}{E_{\lambda^{(n)}}}(u^{(n)};\D_1)$. Let now $v$ be a competitor for $u$, meaning $v\in H^2(\D_1)$ and $\{u\neq v\}\subset\D_s$ for some $s<1$. Let $r\in (s,1)$, and $\zeta\in\mathcal{C}^\infty_c(\D_1)$ a function such that $\zeta=1$ in $\D_s$, $\zeta=0$ in $\D_1\setminus\D_{r}$. Then
$$v^{(n)}:=(1-\zeta)u^{(n)}+\zeta v$$
is a valid competitor for $u^{(n)}$. By our assumption on $u^{(n)}$ we have
$${E_{\lambda^{(n)}}}(v^{(n)};\D_1)-{E_{\lambda^{(n)}}}(u^{(n)};\D_1)\geq 0.$$
We first compute
\begin{align*}
\int_{\D_1}\left(|\Delta v^{(n)}|^2-|\Delta u^{(n)}|^2\right)&=\int_{\D_1}\left(\left|(1-\zeta)\Delta u^{(n)}+\zeta\Delta v+2\nabla(v-u^{(n)})\cdot \nabla\zeta +(v-u^{(n)})\Delta \zeta\right|^2-|\Delta u^{(n)}|^2\right)\\
&=\int_{\D_1}\left(\left((1-\zeta)^2-1\right)|\Delta u^{(n)}|^2+\zeta^2|\Delta v|^2+2\zeta(1-\zeta)\Delta u\Delta v\right)+o_{n}(1)\\
&=\int_{\D_1}\left(|\Delta v|^2-|\Delta u|^2+\left((1-\zeta)^2-1\right)\left(|\Delta u^{(n)}|^2-|\Delta u|^2\right)\right)+o_{n}(1)
\end{align*}
where the $o_n(1)$ term is due to the weak $H^2(\D_1)$ (and strong $H^1(\D_1)$) convergence $v^{(n)}\to v$, which is equal to $u$ on the support of $\nabla\eta,\Delta \eta$. Next,
\begin{align*}
\lambda^{(n)}\int_{\D_1}\left(\chi_{v^{(n)}\neq 0}-\chi_{u^{(n)}\neq 0}\right)&\geq \lambda^{(n)}\int_{\D_1}\left(\chi_{\D_s}\left(\chi_{v\neq 0}-\chi_{u^{(n)}\neq 0}\right)-\chi_{\D_{r}\setminus\D_s}\right)\\
&=\lambda^{(n)}\int_{\D_1}\left(\left(\chi_{v\neq 0}-\chi_{u\neq0}\right)+\chi_{\D_s}\left(\chi_{u\neq 0}-\chi_{u^{(n)}\neq 0}\right)-\chi_{\D_r\setminus\D_s}\right),
\end{align*}
so we may rearrange $E(u^{(n)};\D_1)\leq E(v^{(n)};\D_1)$ into
\begin{align*}
E_{\lambda^{(n)}}(v;\D_1)-E_{\lambda^{(n)}}(u;\D_1)\geq &\int_{\D_1}\left(\left(1-(1-\eta)^2\right)\left(|\Delta u^{(n)}|^2-|\Delta u|^2\right)+\lambda^{(n)}\chi_{\D_s}\left(\chi_{u^{(n)}\neq 0}-\chi_{u\neq 0}\right)\right)\\
&+o_n(1)-2\pi\lambda^{(n)}(r-s).
\end{align*}
Taking $n\to +\infty$, we get
\begin{eqnarray}
E_{\lambda}(v;\D_1)-E_{\lambda}(u;\D_1)&\geq& \liminf_{n\to\infty}\int_{\D_1}\left(\left(1-(1-\eta)^2\right)\left(|\Delta u^{(n)}|^2-|\Delta u|^2\right)+\lambda^{(n)}\chi_{\D_s}\left(\chi_{u^{(n)}\neq 0}-\chi_{u\neq 0}\right)\right)\notag\\\notag
&&-2\pi\lambda (r-s)\\
&\geq &-2\pi \lambda(r-s).\label{eq_aux_sequenceproof}
\end{eqnarray}
Taking $r\searrow s$, we obtain $E_\lambda(v,\D_1)\geq E_\lambda(u;\D_1)$ so $u\in\M_\lambda(\D_1)$. Moreover, the convergence is strong in $H^2_\loc$: indeed taking $v=u$, inequality \eqref{eq_aux_sequenceproof} gives
\begin{align*}
2\pi\lambda(r-s)+\int_{\D_1}\left(1-(1-\eta)^2\right)|\Delta u|^2&\geq \liminf_{n\to\infty}\int_{\D_1}\left(1-(1-\eta)^2\right)|\Delta u^{(n)}|^2,\\
\end{align*}
which implies
\begin{align*}
2\pi\lambda(r-s)+\int_{\D_r}|\Delta u|^2&\geq \liminf_{n\to\infty}\int_{\D_s}|\Delta u^{(n)}|^2.\\
\end{align*}
By taking $r$ arbitrarily close to $s$, we obtain $\int_{\D_s}|\Delta u|^2\geq \liminf_{n\to\infty}\int_{\D_s}|\Delta u^{(n)}|^2$ for every $s$, so $u^{(n)}$ converges strongly (in $H^2(\D_s)$) to $u$. Similarly, for the convergence of the area term, inequality \eqref{eq_aux_sequenceproof} with $v=u$ gives 
$$\int_{\D_s}\lambda\chi_{u\neq 0}\geq\liminf\int_{\D_s}\lambda^{(n)}\chi_{u^{(n)}\neq 0},$$
for every $s$ (and the opposite inequality is a consequence of Fatou's lemma). By dominated convergence, $\lambda^{(n)}\chi_{u^{(n)}\neq 0}\underset{n\to +\infty}{\longrightarrow}\lambda\chi_{u\neq 0}$ in $L^1(\D_1)$.
\end{proof}

\section{Monotonicity formula and blow-ups}\label{sec_mono}

Monotonicity formulas are roughly speaking a way to express the notion that, as we zoom in on a solution, the solution can only look simpler and simpler. For the Alt-Caffarelli problem, i.e. for minimizers of
$$ u\mapsto\int_{\D_1}\left(|\nabla u|^2+\chi_{u>0}\right),$$
Weiss found in \cite{W98,W99} that the quantity
$$r\mapsto \frac{1}{r^2}\int_{\D_r}\left(|\nabla u|^2+\chi_{u>0}\right)-\frac{1}{r^3}\int_{\partial\D_r}u^2$$
is always non-decreasing, and it is constant if and only if $u$ is $1$-homogeneous. From this monotonicity alone one can deduce that blow-ups are always $1$-homogeneous, and that for any solution $u$, the $H^1(\D_r)$ norm is controlled by the $L^2(\partial\D_r)$ norm (with appropriate scaling) for small $r$. This monotonicity formula was later refined in several settings (see \cite{Weiss99}, \cite{SV19}, and also the monograph \cite{V23}) to prove the convergence to a unique blow-up by an \textit{epiperimetric} inequality (which originates from \cite{R64} with a similar estimate for the perimeter of a minimal surface, and later in \cite{Weiss99} in the setting of obstacle problems). Our interest in these techniques is due to the fact that even for classical Alt-Caffarelli problem, these epiperimetric methods use no maximum principle or notion of viscosity solutions, something which is not available in our setting.

In \cite[Th. 1.12]{DKV20}, the authors found a similar monotone quantity for higher order problem (note that their functional is not exactly the same as ours, however it is applied to points where the solution has homogeneity $2$). We start by stating again the monotonicity of this quantity in our setting, with a proof that is made slightly shorter by an exponential change of variable.

The monotonicity formula of Weiss for the classical Alt-Caffarelli problem may be obtained in several ways. The first is by comparing the energy of the solution $u$ with its $1$-homogeneous extension $|x|u\left(\frac{rx}{|x|}\right)$ in $\D_r$. The second is by using that $u$ is critical for the energy with respect to small deformation $u\left(x+t f(|x|) x\right)$ for a smooth function $f$ that approximates $\chi_{[0,r]}$.
In our setting, the $2$-homogeneous extension is not always a valid competitor (since it may create a jump in the gradient), and we do not obtain the monotonicity by the comparison with a single competitor as in the classical Alt-Caffarelli problem: however we are able to adapt the second method to our setting.

\subsection{Monotonicity formula of Dipierro, Karakhanyan and Valdinoci recast in exponential coordinates}\label{subsec_mono}

In this paragraph, we deal with the monotonicity formula. We start choosing a specific set of coordinates that is convenient to express the energy. We deduce in this set of coordinates a monotonicity formula (Theorem \ref{Th_Monotonicity_exp}) and then express it again in the usual coordinates (Theorem \ref{Th_Monotonicity_disk}).

We will denote  $\Cy_{T}=(T,+\infty)\times\Sp$ the half-infinite cylinder, such that $(t,\theta)\in \Cy_T\mapsto e^{-t+i\theta}\in\D_{e^{-T}}\setminus\{0\}$ is a bijective anticonformal map. We denote by $H^2_\lin(\Cy_T)$ the set of functions $v\in H^2_\loc(\Cy_T)$ such that for some constant $C>0$ (that may depend on $v$), we have
\begin{equation}\label{eq_def_growth}
\forall t\geq T,\ \Vert  v\Vert_{H^2(\Cy_t\setminus \Cy_{t+1})}\leq C(t-T+1).
\end{equation}

For any $\tau\in\R$, $v\in H^2_\lin(\Cy_0)$, we denote
\begin{equation}\label{defGronde}
\G(v,\tau)=\int_{\Cy_\tau}e^{-2(t-\tau)}\left\{\left(\partial_{t,t}v\right)^2+2\left(\partial_{t,\theta}v\right)^2+\left(\partial_{\theta,\theta}v\right)^2-4\left(\partial_{\theta}v\right)^2+\chi_{v\neq 0}\right\}
\end{equation}
\begin{lemma}\label{lem_change_variable_diskexp}
Let $u\in H^2(\D_1)$ such that $v\in H^2_\loc(\Cy_0)$ defined by
$$v(t,\theta)=e^{2t}u\left(e^{-t+i\theta}\right)$$
is in $H^2_\lin(\Cy_0)$. Then for every $\tau\geq 0$, we have 
\begin{align*}
\frac{1}{(e^{-\tau})^2}\int_{\D_{e^{-\tau}}}\left\{|\Delta u|^2+\chi_{u\neq 0}\right\}&=\G(v,\tau)+2\int_{\partial \Cy_\tau}\left\{2\left(\partial_tv\right)^2-4v \partial_tv+4v^2-\left(\partial_\theta v\right)^2+\partial_{t,\theta}v\partial_\theta v\right\}.
\end{align*}
\end{lemma}

\begin{proof}
We express $u$ in terms of $v$ with

\[u(re^{i\theta})=r^2v(-\log(r),\theta).\]
Then
\[\Delta u(re^{i\theta})=\left(\partial_{\theta,\theta}v+\partial_{t,t}v-4\partial_tv+4v\right)(-\log(r),\theta),\]
so we may rewrite the energy of $u$ in the disk $\D_{e^{-\tau}}$ as
\begin{equation}\label{eq_lien_energie_u_v}
\int_{\D_{e^{-\tau}}}\left\{|\Delta u|^2+\chi_{u\neq 0}\right\}=\int_{\Cy_{\tau}}e^{-2t}\left\{|\partial_{\theta,\theta}v+\partial_{t,t}v-4\partial_tv+4v|^2+\chi_{v\neq 0}\right\}.
\end{equation}
We develop this expression:
\begin{align*}
\int_{\Cy_{\tau}}e^{-2(t-\tau)}&\left\{|\partial_{t,t}v-4\partial_tv+4v+\partial_{\theta,\theta}v|^2+\chi_{v\neq 0}\right\}\\
&=\int_{\Cy_\tau}e^{-2(t-\tau)}\left\{\left(\partial_{t,t}v\right)^2+16 \left(\partial_tv\right)^2+16v^2+\left(\partial_{\theta,\theta}v\right)^2+\chi_{v\neq 0}\right\}\\
&\;\;\;+\int_{\Cy_\tau}e^{-2(t-\tau)}\left\{-8 \partial_{t,t}v\partial_tv+8 \partial_{t,t}vv +2 \partial_{t,t}v\partial_{\theta,\theta}v-32 \partial_tvv-8 \partial_tv\partial_{\theta,\theta}v+8 v \partial_{\theta,\theta}v\right\}.
\end{align*}
By integration by parts we have:
\begin{align*}
\int_{\Cy_\tau}e^{-2(t-\tau)}(-8\partial_{t,t}v\partial_tv)&=\int_{\Cy_\tau}e^{-2(t-\tau)}(-8\left(\partial_tv\right)^2)+\int_{\partial \Cy_\tau}4\left(\partial_tv\right)^2\\
\int_{\Cy_\tau}e^{-2(t-\tau)}8\partial_{t,t}vv&=\int_{\Cy_\tau}e^{-2(t-\tau)}8(-\partial_tv+2v)\partial_tv+\int_{\partial \Cy_\tau}-8v\partial_tv\\
&=\int_{\Cy_\tau}e^{-2(t-\tau)}(-8\left(\partial_tv\right)^2+16v^2)+\int_{\partial \Cy_\tau}(-8v\partial_tv-8v^2)\\
\int_{\Cy_\tau}e^{-2(t-\tau)}(-32\partial_tvv)&=\int_{\Cy_\tau}e^{-2(t-\tau)}(-32v^2)+\int_{\partial \Cy_\tau}16v^2\\
\int_{\Cy_\tau}e^{-2(t-\tau)}(-8\partial_tv\partial_{\theta,\theta}v)&=\int_{\Cy_\tau}e^{-2(t-\tau)}8\partial_{t,\theta}v\partial_{\theta}v\\
&=\int_{\Cy_\tau}e^{-2(t-\tau)}8\left(\partial_{\theta}v\right)^2+\int_{\partial \Cy_\tau}-4 \left(\partial_\theta v\right)^2.
\end{align*}
In each integration by parts we used the growth estimate \eqref{eq_def_growth} so that there is no boundary term at $t=+\infty$. Finally, for almost every $\tau\geq 0$ we have
\begin{align*}
\int_{\Cy_\tau}e^{-2(t-\tau)}2\partial_{t,t}v\partial_{\theta,\theta}v&=\int_{\Cy_\tau}e^{-2(t-\tau)}(2\left(\partial_{t,\theta}v\right)^2-4\left(\partial_\theta v\right)^2)+\int_{\partial \Cy_\tau}(2\left(\partial_{\theta}v\right)^2+2\partial_{t,\theta}v\partial_{\theta}v).
\end{align*}
Summing all these contributions, we indeed get for almost every $\tau\geq 0$:

\begin{equation}\label{eq_aux_changevariable}
\begin{split}
\int_{\Cy_{\tau}}&e^{-2(t-\tau)}\left\{|\partial_{t,t}v+4\partial_tv+4v+\partial_{\theta,\theta}v|^2+\chi_{v\neq 0}\right\}\\
&=\int_{\Cy_\tau}e^{-2(t-\tau)}\left\{\left(\partial_{t,t}v\right)^2+2\left(\partial_{t,\theta}v\right)^2+\left(\partial_{\theta,\theta}v\right)^2-4\left(\partial_{\theta}v\right)^2+\chi_{v\neq 0}\right\}\\
&\;\;\;+2\int_{\partial \Cy_\tau}\left\{2\left(\partial_tv\right)^2-4v \partial_tv+4v^2-\left(\partial_\theta v\right)^2+\partial_{t,\theta}v\partial_\theta v\right\}d\theta.
\end{split}
\end{equation}

\end{proof}

\begin{definition}\label{def_minimizerG}
A function $v\in H^2_\loc(\Cy_\tau)$ is called a minimizer of $\G(\cdot ,\tau)$ if $v\in H^2_\lin(\Cy_\tau)$ and if letting $$u\left(re^{i\theta}\right)=r^2v\left(-\log(r),\theta\right),$$
we have $u\in\M(\D_{e^{-\tau}})$.
\end{definition}
In particular, for any $w\in H^2_\lin(\Cy_\tau)$ such that $\{v\neq w\}\subset \Cy_{\tau+\eps}$ (for some $\eps>0$), we have (by Lemma \ref{lem_change_variable_diskexp}):
\[\G(v,\tau)\leq \G(w,\tau).\]
The reason we do not define minimizers of $\G(\cdot,\tau)$ only by this condition, is because this does not take into account competitors for $u$ for which either $u(0)$ or $\nabla u(0)$ do not vanish. Indeed, $\G(\cdot,\tau)$ is \textit{a priori} only defined for functions that grow sufficiently slow as $t\to +\infty$.

We now define a ``corrected'' energy, for any $v\in H^2_\lin(\Cy_\tau)$:
\begin{equation}\label{def_W_ronde}
\W(v,\tau)=\G(v,\tau)+2\int_{\partial \Cy_\tau}\partial_tv(\partial_tv-\partial_{t,t}v),
\end{equation}
where the corrective term is chosen so that $\W'$ has a sign.
Due to the last term,  $\W(v,\tau)$ is \textit{a priori}  defined for almost every $\tau$ only. The next result implies in fact that $\W(v,\tau)$ is continuous in $\tau$ when $v$ is a minimizer, so it is well-defined for every $\tau> 0$.

\begin{theorem}\label{Th_Monotonicity_exp}
Let $v$ be a minimizer of $\G(\cdot\ ;0)$, then $\tau\mapsto \W(v,\tau)$ belongs to $W^{1,1}_{\loc}(\R_+)$, and its weak derivative (defined for almost every $\tau\geq 0$) is
\[\W'(v,\tau)=-4\int_{\partial \Cy_\tau}\left\{\left(\partial_{t,t}v\right)^2+\left(\partial_{t,\theta}v\right)^2\right\}.\]
\end{theorem}
Here $\W'(v,\tau)$ designates the derivative of $\W(v,\tau)$ along the variable $\tau$, which exists in $L^1_\loc(\R_+)$ in the weak sense since $\W(v,\cdot)$ is absolutely continuous. In particular, $\W(v,\tau)$ is constant in $\tau$ if and only if $\partial_t v(t,\theta)$ is constant. 
\begin{proof}
We write $\Vert\cdot\Vert$ (resp. $\langle\cdot,\cdot\rangle$) the $L^2(\Sp)$ norm (resp. the $L^2(\Sp)$ scalar product). It is sufficient to prove the following: for any $\varphi\in\mathcal{C}^\infty_c(\R_{>0},\R)$, we have
\begin{equation}\label{eq_proof_monotonie}
\int_{0}^{+\infty}\varphi'(t)\W(v,t)dt=4\int_{0}^{+\infty}\varphi(t)\left\{\Vert \partial_{t,t}v(t)\Vert^2+\Vert \partial_{t,\theta}v(t)\Vert^2 \right\}dt.
\end{equation}

Let now $f\in\mathcal{C}^\infty(\R_+,\R)$, such that $f=0$ in a neighbourhood of $0$, and $f,f',f''$ are bounded in $L^\infty(\R_+)$ (note that $f$ does \textbf{not} have compact support in general). For any small enough $\eps$ (positive of negative), the function $t\mapsto t+\eps f(t)$ is a diffeomorphism of $\R_+$, with inverse $\psi_\eps$. Let
$$v_\eps(t,\theta)=v(\psi_\eps(t),\theta).$$
Then $v_\eps\in H^2_\lin(\Cy_0)$ is a valid competitor for $v$. We compute
\begin{align*}
\partial_{\theta}v_\eps(t,\theta)&=\partial_{\theta}v(\psi_\eps(t),\theta),\qquad \partial_{\theta,\theta}v_\eps(t,\theta)=\partial_{\theta,\theta}v(\psi_\eps(t),\theta)\\
\partial_{t}v_\eps(t,\theta)&=(1-\eps f'(t)+o(\eps))\partial_{t}v(\psi_\eps(t),\theta),\\
\partial_{t,\theta}v_\eps(t,\theta)&=(1-\eps f'(t)+o(\eps))\partial_{t,\theta}v(\psi_\eps(t),\theta),\\
\partial_{t,t}v_\eps(t,\theta)&=(1-2\eps f'(t)+o(\eps))\partial_{t,t}v(\psi_\eps(t),\theta)-(\eps f''(t)+o(\eps))\partial_{t}v(\psi_\eps(t),\theta),
\end{align*}
so
\begin{align*}
\G(v_\eps,0)=\G(v,0)+o(\eps)&-\eps\int_{\Cy_0}e^{-2t}\left\{4f'(\partial_{t,t}v)^2+2f''\partial_t v\partial_{t,t}v+4f'(\partial_{t,\theta}v)^2\right\}\\
&+\eps \int_{\Cy_0}e^{-2t}(f'-2f)\left\{(\partial_{t,t}v)^2+2(\partial_{t,\theta}v)^2+(\partial_{\theta,\theta}v)^2-4(\partial_{\theta}v)^2+\chi_{v\neq 0}\right\}.
\end{align*}
Taking the derivative at $\eps=0$ above, we have for any such $f$ that
\begin{align*}
\int_{\Cy_0}e^{-2t}(f'-2f)&\left\{(\partial_{t,t}v)^2+2(\partial_{t,\theta}v)^2+(\partial_{\theta,\theta}v)^2-4(\partial_{\theta}v)^2+\chi_{v\neq 0}\right\}-2e^{-2t}f''\partial_t v\partial_{t,t}v \\
&=4\int_{\Cy_0}e^{-2t}f'\left\{( \partial_{t,t}v)^2+( \partial_{t,\theta}v)^2\right\}.
\end{align*}
Let $\varphi\in\mathcal{C}^\infty_c(\R_{>0},\R)$, and $f(t)=\int_{0}^{t}e^{2s}\varphi(s)ds$. The function $f$ verifies the previous hypothesis (in particular it is constant for any large enough $t$, meaning it is bounded), and $f'(t)=e^{2t}\varphi(t)$, $f''(t)=e^{2t}\left(\varphi'(t)+2\varphi(t)\right)$, so the previous relation becomes
\begin{multline*}
\int_{\Cy_0}\left(\varphi(t)-2e^{-2t}\int_{0}^{t}e^{2s}\varphi(s)ds\right)\left\{(\partial_{t,t}v)^2+2(\partial_{t,\theta}v)^2+(\partial_{\theta,\theta}v)^2-4(\partial_{\theta}v)^2+\chi_{v\neq 0}\right\} \\
-2\int_{\Cy_0}\left(\varphi'(t)+2\varphi(t)\right)\partial_t v\partial_{t,t}v =4\int_{\Cy_0}\varphi(t)\left\{( \partial_{t,t}v)^2+( \partial_{t,\theta}v)^2\right\}. 
\end{multline*}
On one hand,
$$-2\int_{\Cy_0}\left(\varphi'(t)+2\varphi(t)\right)\partial_t v\partial_{t,t}v =2\int_{0}^{+\infty}\varphi'(t)\langle \partial_t v(t),\partial_t v(t)-\partial_{t,t}v(t)\rangle dt,$$
and on the other hand,
\begin{align*}
\int_{\Cy_0}&\left(\varphi(t)-2e^{-2t}\int_{0}^{t}e^{2s}\varphi(s)ds\right)\left\{(\partial_{t,t}v)^2+2(\partial_{t,\theta}v)^2+(\partial_{\theta,\theta}v)^2-4(\partial_{\theta}v)^2+\chi_{v\neq 0}\right\} \\
&=\int_{\Cy_0}\left(e^{2t}\varphi(t)-2\int_{0}^{t}e^{2s}\varphi(s)ds\right)e^{-2t}\left\{(\partial_{t,t}v)^2+2(\partial_{t,\theta}v)^2+(\partial_{\theta,\theta}v)^2-4(\partial_{\theta}v)^2+\chi_{v\neq 0}\right\} \\
&=\int_{0}^{+\infty}\left(e^{2t}\varphi(t)-2\int_{0}^{t}e^{2s}\varphi(s)ds\right)\frac{d}{dt}\left\{-e^{-2t}\G(v,t)\right\}dt\\
&=\int_{0}^{+\infty}\varphi'(t)\G(v,t)dt,
\end{align*}
by integration by parts (we remind that $\varphi$ has compact support). As a consequence, we obtain that for any $\varphi\in\mathcal{C}^\infty_c(\R_{>0},\R)$, we have
$$\int_{0}^{+\infty}\varphi'(t)\left\{\G(v,t)+2\langle\partial_t v(t),\partial_t v(t)-\partial_{t,t}v(t)\rangle\right\}dt=4\int_{0}^{+\infty}\varphi(t)\left\{\Vert \partial_{t,t}v(t)\Vert^2+\Vert\partial_{t,\theta}v(t)\Vert^2\right\}dt.$$
By the definition of $\W$ in \eqref{def_W_ronde}, we have obtained that $t\mapsto \W(v,t)$ admits a weak derivative, which is of the form given in \eqref{eq_proof_monotonie}.
\end{proof}

We may now state this theorem in usual coordinates (both will be useful). Suppose $u\in\M(\D_1)$ such that $u(0)=|\nabla u(0)|=0$, and we remind that for any $r\in (0,1)$ we let
$$u_r(x)=\frac{u(rx)}{r^2}$$
the natural $2$-homogeneous rescaling of $u$, defined in $\D_{1/r}$. To avoid notational ambiguity, we always denote $\partial_r$ the radial derivative of a function, and $\frac{d}{dr}$ the derivative with respect to the parameter $r$, in other words
\begin{align*}
\forall x\in\Sp,\ \frac{d}{dr}u_r(x)&=\frac{1}{r}(\partial_r u_r(x)-2u_r(x)).
\end{align*}
We denote
\begin{equation}\label{eq_def_N}
N(u,r)=\int_{\Sp}\left(r^2\left|\frac{d}{dr}u_r\right|^2+\left|\partial_\theta u_r\right|^2+|u_r|^2\right).
\end{equation}
In particular $N(u,r)$ verifies
$$ \Vert u_r\Vert_{L^2(\Sp)}^2+\Vert \nabla u_r\Vert_{L^2(\Sp)}^2\lesssim N(u,r)\lesssim \Vert u_r\Vert_{L^2(\Sp)}^2+\Vert \nabla u_r\Vert_{L^2(\Sp)}^2,$$
$r\mapsto N(u,r)$ is absolutely continuous, and its (weak) derivative defined almost everywhere is
\begin{equation}\label{eq_exprN}
\begin{split}
N'(u,r)=\frac{2}{r}\int_{\Sp}\left(\left(\partial_r u_r-2u_r\right)\left(\partial_{r,r}u_r-3\partial_r u_r+4u_r\right)+\left(\partial_{r,\theta}u_r-2\partial_{\theta}u_r\right)\partial_\theta u_r+\left(\partial_{r}u_r-2u_r\right) u_r\right)\\
=\frac{2}{r}\int_{\Sp}\left(\partial_r u_r\partial_{r,r}u_r-2u_r\partial_{r,r}u_r+\partial_{\theta}u_r\partial_{r,\theta}u_r-3(\partial_ru_r)^2+11 u_r\partial_r u_{r}-2(\partial_\theta u_r)^2-10u_r^2\right).
\end{split}
\end{equation}
We will also denote
\begin{equation}\label{eq_def_R}
R(u,r):=2\int_{\Sp}\left\{(\partial_\theta u_r)^2-(\partial_r u_r)^2+2u_r^2-u_r\partial_r u_r\right\}.
\end{equation}

\begin{theorem}\label{Th_Monotonicity_disk}
Let $u\in\M(\D_1)$ such that $u(0)=|\nabla u(0)|=0$, define
\begin{equation}\label{def_W_disk}
W(u,r):=\frac{1}{r^2}E(u;\D_r)+rN'(u,r)+R(u,r).
\end{equation}
Then $r\mapsto W(u,r)$ is non-decreasing, and it is constant in $r$ if and only if $u$ is $2$-homogeneous and in this case
$$W(u,r)=W(u,1)=\left|\{u\neq 0\}\cap\D_1\right|.$$

Moreover, if we denote $v(t,\theta)=r^{-2}u\left(re^{i\theta}\right)$ where $r=e^{-t}$, then $v$ is a minimizer of $\G(\cdot\ ;0)$ and for any $t\geq 0$:
$$\W(v,t)=W(u,r).$$
\end{theorem}
Although the corrective term looks different, once everything is developed we find that the normalized energy of \eqref{def_W_disk} is equal to the one defined in \cite[Th. 1.12]{DKV20} (up to a scalar factor), also recalled in \eqref{eq_WDKV}.
\begin{proof}
By Lemma \ref{lem_change_variable_diskexp}, we have
$$\G(v,\tau)=e^{2\tau}E\left(u;\D_{e^{-\tau}}\right)-2\int_{\partial \Cy_\tau}\left\{2\left(\partial_tv\right)^2-4v \partial_tv+4v^2-\left(\partial_\theta v\right)^2+\partial_{t,\theta}v\partial_{\theta} v\right\}d\theta.$$
We replace $v(t,\theta)=e^{2t}u\left(e^{-t+i\theta}\right)$, so
\begin{align*}
\partial_t v(t,\theta)&=2e^{2t}u\left(e^{-t+i\theta}\right)-e^{t}\partial_ru\left(e^{-t+i\theta}\right),\\
\partial_{t,\theta} v(t,\theta)&=2e^{2t}u_{\theta}\left(e^{-t+i\theta}\right)-e^{t}\partial_{r,\theta}u\left(e^{-t+i\theta}\right),\\
\partial_{t,t} v(t,\theta)&=4e^{2t}u\left(e^{-t+i\theta}\right)-3e^{t}\partial_ru\left(e^{-t+i\theta}\right)+\partial_{r,r}u\left(e^{-t+i\theta}\right),\\
\partial_\theta v(t,\theta)&=e^{2t}\partial_\theta u\left(e^{-t+i\theta}\right).
\end{align*}
Denoting $r=e^{-\tau}$, the boundary integral term above becomes
\begin{align*}
-2\int_{\partial \Cy_\tau}&\left\{2\left(\partial_tv\right)^2-4v \partial_tv+4v^2-\left(\partial_\theta v\right)^2+\partial_{t,\theta}v\partial_{\theta} v\right\}d\theta\\
&=-\frac{2}{r}\int_{\partial\D_r}\left(2\left(\frac{2u}{r^2}-\frac{\partial_r u}{r}\right)^2-4\frac{u}{r^2}\left(\frac{2u}{r^2}-\frac{\partial_r u}{r}\right)+4\frac{u^2}{r^4}-\frac{(\partial_\theta u)^2}{r^4}+\frac{\partial_{\theta}u}{r^2}\left(\frac{2\partial_\theta u}{r^2}-\frac{\partial_{r,\theta} u}{r}\right)\right)\\
&=-\frac{2}{r}\int_{\partial\D_r}\left(8\frac{u^2}{r^4}+2\frac{(\partial_r u)^2}{r^2}-8\frac{u\partial_ru}{r^3}-8\frac{u^2}{r^4}+4\frac{u\partial_r u}{r^3}+4\frac{u^2}{r^4}+\frac{(\partial_\theta u)^2}{r^4}-\frac{\partial_\theta u\partial_{r,\theta}u}{r^3}\right)\\
&=-\frac{2}{r}\int_{\partial\D_r}\left(4\frac{u^2}{r^4}+2\frac{(\partial_r u)^2}{r^2}-4\frac{u\partial_ru}{r^3}+\frac{(\partial_\theta u)^2}{r^4}-\frac{\partial_\theta u\partial_{r,\theta}u}{r^3}\right).
\end{align*}
We obtain
$$\G(v,\tau)=\frac{1}{r^2}E(u;\D_r)-\frac{2}{r}\int_{\partial\D_r}\left(4\frac{u^2}{r^4}+2\frac{(\partial_r u)^2}{r^2}-4\frac{u\partial_ru}{r^3}+\frac{(\partial_\theta u)^2}{r^4}-\frac{\partial_\theta u\partial_{r,\theta}u}{r^3}\right).$$
Then from the definition of $\W(v,\tau)$ \eqref{def_W_ronde},
\begin{align*}
\W(v,\tau)&=\G(v,\tau)+2\int_{\partial\Cy_\tau}\partial_t v (\partial_{t}v-\partial_{t,t}v)\\
&=\G(v,\tau)+\frac{2}{r}\int_{\partial \D_r}\left(\frac{\partial_r u}{r}-2\frac{u}{r^2}\right)\left(\partial_{r,r}u-2\frac{\partial_r u}{r}+2\frac{u}{r^2}\right)\\
&=\G(v,\tau)+\frac{2}{r}\int_{\partial \D_r}\left(\frac{\partial_{r}u\partial_{r,r}u}{r}-2\frac{u\partial_{r,r}u}{r^2}-2\frac{(\partial_r u)^2}{r^2}+6\frac{u\partial_r u}{r^3}-4\frac{u^2}{r^4}\right).
\end{align*}
So replacing the expression of $\G(v,\tau)$ found above we obtain
\begin{equation}\label{eq_explicit_expressionW}
\W(v,\tau)=\frac{1}{r^2}E(u;\D_r)+\frac{2}{r}\int_{\partial\D_r}\left(\frac{\partial_r u\partial_{r,r}u}{r}-2\frac{u\partial_{r,r}u}{r^2}+\frac{\partial_{\theta}u\partial_{r,\theta}u}{r^3}-4\frac{(\partial_r u)^2}{r^2}+10\frac{u\partial_r u}{r^3}-\frac{(\partial_\theta u)^2}{r^4}-8\frac{u^2}{r^4}\right).
\end{equation}
This is the same expression as in \eqref{eq_WDKV}. Moreover, the boundary integral term on the right-hand side is exactly the sum of $rN'(u,r)$ (see \eqref{eq_exprN}) and $R(u,r)$ (see \eqref{eq_def_R}).

Suppose now that  $r\mapsto W(u,r)$ is constant. Then $t\mapsto \W(v,t)$ is constant, so by Theorem \ref{Th_Monotonicity_exp}, $\partial_{t,\theta}v$ and $\partial_{t,t}v$ are identically zero. Thus for some constant $c\in\R$, we have
$$v(t,\theta)=v(0,\theta)+ct.$$
We denote $g(\theta)=v(0,\theta)$, then $u$ may be rewritten as
$$u(re^{i\theta})=cr^2\log(r)+r^2g(\theta).$$
Assume by contradiction $c$ is non-zero. Since $u\in\mathcal{C}^{1}(\D_1,\R)$ (by Lemma \ref{lem_C1log}), in particular $g$ is bounded, so there is some sufficiently small radius $\ov{r}$ such that $u\neq 0$ in $\D_{\ov{r}}\setminus\{0\}$.
Let $h$ be the biharmonic extension of $u$ in $\D_{\ov{r}}$. Since $u$ has full support in $\D_{\ov{r}}$ we have by minimality of $u$
$$\int_{\D_{\ov{r}}}|\Delta(u-h)|^2=\int_{\D_{\ov{r}}}\left(|\Delta u|-|\Delta h|^2\right)\leq 0,$$
so $u$ is biharmonic in $\D_{\ov{r}}$. This is in contradiction with the fact that $c\neq 0$.

As a consequence, $u$ is of the form
$$u(re^{i\theta})=r^2g(\theta),$$
meaning it is $2$-homogeneous. Reciprocally if $u$ is of this form, then we see that $$W(u,r)=\W(v,-\log(r))=\W(v,0)=W(u,1)$$ for every $r<1$, so $W(u,\cdot)$ is constant.
\end{proof}

Although the boundary term in the definition of $W$ \eqref{def_W_disk} is  not continuous with respect to $H^2$, due to its structure the energy $W$ still verifies a type of Cacciopoli inequality:

\begin{lemma}\label{lem_W_cacciopoli}
Let $u\in\M(\D_1)$ such that $u(0)=|\nabla u(0)|=0$, then
$$\left|W\left(u,\frac{1}{2}\right)\right|\lesssim \Vert u\Vert_{H^1(\D_1)}^2.$$
\end{lemma}

\begin{proof}
If $\Spt(u)\cap\D_{\frac{1}{2}}=\emptyset$ there is nothing to prove: without loss of generality we assume this is not the case. By the non-degeneracy Lemma \ref{lem_nondegeneracy} this implies $\Vert u\Vert_{H^1(\D_1)}\gtrsim 1$.

Let $\varphi\in\mathcal{C}^\infty_c\left(\left]\frac{1}{2},\frac{3}{4}\right[,\R_+\right)$ be such that $\int_{0}^{1}\varphi(s)ds=1$ and $|\varphi'|\lesssim 1$, then by the monotonicity of $r\mapsto W(u,r)$:
\begin{align*}
W\left(u,\frac{1}{2}\right)&\leq \int_{1/2}^{3/4}\varphi(r)W(u,r)dr\\
&=\int_{1/2}^{3/4}\varphi(r)\left(\frac{1}{r^2}E(u;\D_r)+rN'(u,r)+R(u,r)\right)dr\\
&=\int_{1/2}^{3/4}\left(\frac{\varphi(r)}{r^2}E(u;\D_r)-(r\varphi(r))'N(u,r)+\varphi(r)R(u,r)\right)dr\\
&\lesssim 1+\Vert u\Vert_{H^2(\D_{\frac{3}{4}})}^2\\
&\lesssim 1+ \Vert u\Vert_{H^1(\D_{1})}^2\text{ by lemma }\ref{lem_Cacciopoli} \\
&\lesssim \Vert u\Vert_{H^1(\D_{1})}^2\text{ since }\Vert u\Vert_{H^1(\D_1)}\gtrsim 1.
\end{align*}
Similarly, choosing $\varphi\in\mathcal{C}^{\infty}_c\left(\left]\frac{1}{4},\frac{1}{2}\right[,\R_+\right)$ such that $\int_{0}^{1}\varphi(s)ds=1$ and $|\varphi'|\lesssim 1$ we get the opposite estimate.
\end{proof}

The monotonicity formula extends to minimizers of $E_0(u,D):=\int_{D}|\Delta u|^2$ i.e. biharmonic functions. We denote 
\begin{equation}\label{eq_defW0}
W_0(u,r):=\frac{1}{r^2}E_0(u;\D_r)+rN'(u,r)+R(u,r).
\end{equation}
Let $u\in H^2(\D_1)$ be a biharmonic function: by using its Goursat decomposition ($h_1+(x^2+y^2)h_2$ for $h_1,h_2$ harmonic in $\D_1$), $u$ can be decomposed as
\begin{equation}\label{eq_Goursat}
u\left(re^{i\theta}\right)=\sum_{n\in\Z}\left(a_n r^{|n|+2}+b_nr^{|n|}\right)e^{in\theta}
\end{equation}
for some complex coefficients $(a_n)_{n\in\Z},(b_n)_{n\in\Z}$, such that the sum converges absolutely in $\D_1$. 
\begin{lemma}\label{lem_computationWbih}
Let $u\in H^2(\D_1)$ be a biharmonic function, and let $(a_n)_{n\in\Z}$, $(b_n)_{n\in\Z}$ be the coefficients defined in the decomposition \eqref{eq_Goursat}, then for any $r\in (0,1)$ we have
$$
W_0(u,r)=8\pi\sum_{n\in\Z}\Big\{|n|^3r^{2|n|}|a_n|^2+2|n|^2(|n|-2)r^{2|n|-2}\Re(\ov{a_n}b_n)+(|n|-2)(|n|^2-2|n|+2)r^{2|n|-4}|b_n|^2\Big\}.
$$
\end{lemma}

\begin{proof}
For this computation, we will directly use definition \eqref{eq_WDKV}.\\
Let us first notice that for any $r<1$, $W_0(u,r)=W_0(u_r,1)$, and the coefficients associated to the Goursat decomposition of $u_r$ are $(r^{|n|}a_n,r^{|n|-2}b_n)_{n\in\Z}$: we lose no generality in proving the result at $r=1$ only, for a function $u$ that is biharmonic in a neighbourhood of $\ov{\D_1}$.

For a function $f:\ov{\D_1}\to\C$  we define
$$c_n[f]=\frac{1}{2\pi}\int_{\partial\D_1}e^{-in\theta}f(e^{i\theta})d\theta$$
the Fourier coefficients of $f$.
Then we compute 
\begin{align*}
c_n[u]&=a_n+b_n,\\
c_n[\partial_r u]&=(|n|+2)a_n+|n|b_n,\\
c_n[\partial_{\theta,\theta}u]&=-|n|^2(a_n+b_n),\\
c_n[\Delta u]&=4(|n|+1)a_n.
\end{align*}
We deduce $E_0(u;\D_1)=\int_{\D_1}\left|\sum_{n\in\Z}4(|n|+1)a_n r^{|n|}e^{in\theta}\right|^2=8\pi\sum_{n\in\Z}(2|n|+2)|a_n|^2$, and 
\begin{align*}
&N'(u,1)+R(u,1)=\int_{\partial \D_1}\left\{2\partial_r u\Delta u-10 (\partial_ru)^2-4u\Delta u+24u \partial_ru+4\partial_{\theta}u \partial_{r,\theta}u-16u^2-6(\partial_{\theta}u)^2\right\}\\
&=2\pi\sum_{n\in\Z}\Big\{2((|n|+2)a_n+|n|b_n)4(|n|+1)\ov{a_n}-10\big|(|n|+2)a_n+|n|b_n\big|^2-16(|n|+1)a_n(\ov{a_n}+\ov{b_n})\Big\}\\
&~\qquad+2\pi\sum_{n\in\Z}\Big\{(24+4|n|^2)((|n|+2)a_n+|n|b_n)(\ov{a_n}+\ov{b_n})-(16+6|n|^2)|a_n+b_n|^2\Big\}\\
&=8\pi\sum_{n\in\Z}\Big\{\left(|n|^3-2|n|-2\right)|a_n|^2+2|n|^2(|n|-2)\Re(\ov{a_n}b_n)+|n|(|n|-2)(|n|^2-2|n|+2)|b_n|^2\Big\}.
\end{align*}
Then $W_0(u,1)=E_0(u;\D_1)+N'(u,1)+R(u,1)$ simplifies to the expression of the lemma with $r=1$.
\end{proof}

\begin{remark}
Let $u\in H^2(\D_1)$ be a biharmonic function, and we write $(a_n,b_n)_{n\in\Z}$ the coefficient in the Goursat decomposition \eqref{eq_Goursat} of $u$. Then for any $r\in (0,1)$ we have
\begin{align*}
W_0'(u,r)=&\frac{16\pi}{r}\sum_{n\in\Z}|n|^4r^{2|n|}|a_n|^2+2|n|^2(|n|-1)(|n|-2)r^{2|n|-2}\Re(\ov{a_n}b_n)\\
&+\frac{16\pi}{r}\sum_{n\in\Z}(|n|-2)^2(|n|^2-2|n|+2)r^{2|n|-4}|b_n|^2.
\end{align*}
The discriminant of the $n$-th term (seen as a Hermitian form of $r^{|n|}a_n,r^{|n|-1}b_n$) is $-4|n|^4(|n|-2)^2$: for $n\notin \{-2,0,2\}$, the $n$-th term is positive definite, and for $n\in\{-2,0,2\}$, it is positive semi-definite. We find again the monotonicity of $W_0$ in this special case.
Another observation is that 
$$\frac{1}{8\pi}W_0(u,r)=-4r^{-4}|b_0|^2-r^{-2}\left(|b_1|^2+|b_{-1}|^2\right)+\mathcal{O}_{r\to 0}(1).$$
In particular, $W_0(u,0)$ is finite if and only if $u(0)=|\nabla u(0)|=0$. This implies that if $v\in\M(\D_1)$ and $0\in\Spt(v)$, then necessarily $W(v,r)\to -\infty$ as $r\to 0$.
\end{remark}

\subsection{Classification of $2$-homogeneous minimizers}\label{subsec_classification}

From the monotonicity formula of Theorem \ref{Th_Monotonicity_disk}, we expect that limits of $r\mapsto u_r:=r^{-2}u(r\cdot)$ as $r\to 0$, when they exist with a strong enough notion of convergence, should be $2$-homogeneous. For this reason we classify every $2$-homogeneous minimizers; we will see there are only four types. We start with a lemma on biharmonic functions in angular sectors. For any $\om\in (0,2\pi]$ we denote by
$$S_\om:=\{re^{i\theta}:0<r<1,\ 0<\theta<\om\}$$
the angular sector. As mentioned previously, we will denote by $t_1$ the unique fixed point of $t\in (\pi,2\pi)\mapsto \tan(t)$.
\begin{lemma}\label{lem_angularsector}
Let $\om\in (0,2\pi]$, and $u\in H^2\left(S_\om\right)\setminus\{0\}$ such that $u=\partial_\theta u=0$ on $\partial S_\om\setminus\partial\D_1$ and $\Delta^2 u=0$ on $S_\om$. Suppose $u$ is of the form
$u(re^{i\theta})=r^2b(\theta)$ (for some $b\in H^2_0([0,\om])$). Then either
\begin{itemize}
\item[(i) ] $\om=\pi$ and $b$ is a multiple of
\begin{equation}\label{def_bI}
b^{\mathrm{I}}(\theta)=\frac{1}{2}\sin(\theta)_+^2.
\end{equation}
\item[(ii) ] $\om=t_1$ and $b$ is a multiple of
\begin{equation}\label{def_bII}
b^{\mathrm{II}}(\theta)=\frac{1}{4}\left(1-\cos(2\theta)-\frac{2}{t_1}\left(\theta-\frac{\sin(2\theta)}{2}\right)\right)\chi_{0\leq \theta\leq t_1}.
\end{equation}
\item[(iii) ] $\om=2\pi$ and $b$ is a multiple of
\begin{equation}\label{def_bIII}
 b^{\mathrm{III}}(\theta)=\frac{1}{2}\sin(\theta)^2.
\end{equation}
\end{itemize}
\end{lemma}
\begin{proof}
The equation verified by $u$ rewrites as
$$
\begin{cases}
\partial_{\theta}^4b+4\partial_\theta^2 b=0 \qquad \text{ in }(0,\om),\\
b(0)=b'(0)=b(\om)=b'(\om)=0,
\end{cases}
$$
so $b$ must be a linear combination of the solutions of the homogeneous differential equation, that is of $\cos(2\theta)$, $\sin(2\theta)$, $1$, $\theta$. Denote $\chi(\theta)=\theta-\frac{\sin(2\theta)}{2}$, the conditions $b(0)=b'(0)=0$ implies that $b$ must be a linear combination of $\chi(\theta)$ and $\chi'(\theta)=1-\cos(2\theta)$, so for some $(\alpha,\beta)\in\R^2\setminus\{(0,0)\}$:
$$
\begin{cases}
\alpha\chi(\om)+\beta\chi'(\om)&=0\\
\alpha\chi'(\om)+\beta\chi''(\om)&=0.
\end{cases}
$$
There is a nonzero solution if and only if $\chi''(\om)\chi(\om)=\chi'(\om)^2$, meaning
$$\sin(\om)\left(\om\cos(\om)-\sin(\om)\right)=0.$$
The only solutions in $(0,2\pi]$ are $\om=\pi,t_1,2\pi$, and in each case the space of solutions has dimension $1$ and is generated by the function
$$\theta\mapsto \chi(\om)\chi'(\theta)-\chi'(\om)\chi(\theta),$$
which gives $b^{\mathrm{I}},b^{\mathrm{II}},b^{\mathrm{III}}$ (up to some multiplicative constant).
\end{proof}
\begin{lemma}\label{lem_shapeder}
Let $u\in\M(\D_1)$ such that $\partial\Spt(u)=\{(x,y)\in\D_1:y=0\}$. 
\begin{itemize}[label=\textbullet]
\item If $\Spt(u)$ is a half-disk, then  the trace of $\Delta u|_{\Spt(u)}$ on $\partial\Spt(u)$ is $+1$ or $-1$.
\item If $\Spt(u)$ is two half-disks, then $u$ is biharmonic in $\D_1$ and $|\Delta u|\geq 1$ on $\partial\Spt(u)$.
\end{itemize}
\end{lemma}

\begin{proof}
If $u\in\M(\D_1)$, then for any $\phi=(\phi^x,\phi^y)\in\mathcal{C}^\infty_c(\D_1,\R^2)$ we have
$$\left.\frac{d}{dt}\right|_{t=0^+}E\left(u\circ \left(\mathrm{Id}+t\phi\right)^{-1};\D_1\right)\geq 0,$$
which is equivalent to the following inner variation relation (see \cite[Lem. 4.4.]{DKV20} for a detailed computation):
\begin{equation}\label{eq_innervariation}
\forall\phi\in\mathcal{C}^\infty_c(\D_1,\R^2),\ \int_{\D_1}\left(|\Delta u|^2+\chi_{u\neq 0}\right)\nabla\cdot\phi\geq 2\int_{\D_1}\Delta u \left(2 \nabla^2 u : \nabla\phi+\nabla u\cdot \Delta\phi\right)
\end{equation}
and when we replace $\phi$ with $-\phi$, we have an equality above. Suppose that $\Spt(u)=\{(x,y)\in\D_1:y>0\}$, then $u$ is smooth up to the boundary (since it is biharmonic with Dirichlet condition for $u$ and $\nabla u$ along a smooth boundary), so $$u_r\underset{r\to 0}{\longrightarrow}\frac{\Delta^+ u(0)}{2}y_+^2$$ in every $\mathcal{C}^k(\D_{1}\cap\{y>0\})$ space (here $\Delta^+$ designate the trace of $\Delta u|_{\{y>0\}}$ on $\{y=0\}$). Replacing $u$ with $u_r$ in the inner variation relation \eqref{eq_innervariation} gives, as $r\to 0$, and for any $\phi\in\mathcal{C}^\infty_c(\D_1,\R^2)$,
$$\int_{\D_1\cap\{y>0\}}\left(|\Delta^+ u(0)|^2+1\right)\nabla\cdot\phi=2\int_{\D_1\cap\{y>0\}}\Delta^+ u(0) \left(2 \Delta^+ u(0) e_y\otimes e_y : \nabla\phi+\Delta^+ u(0)(ye_y)\cdot \Delta \phi\right),$$
which after integration by parts gives
$$\left((\Delta^+ u(0))^2-1\right)\int_{\D_1\cap\{y=0\}}\phi^y=0.$$
And so $(\Delta^+ u)^2=1$ on $\{(x,y)\in\D_1:y=0\}$.

Suppose now that $\Spt(u)=\{(x,y)\in\D_1:y\neq 0\}$. Let $\phi=(\phi^x,\phi^y)\in\mathcal{C}^\infty_c(\D_1,\R^2)$ such that $\phi^y(x,0)\geq 0$ for any $x$, and let
$$u^t=\begin{cases}
u\circ \left(\mathrm{Id}+t\phi\right)^{-1}&\text{ in }\left(\mathrm{Id}+t\phi\right)(\{y>0\})\\
u&\text{ in }\{y<0\}\\
0&\text{ elsewhere}.
\end{cases}$$

Then $u^t$ is a valid competitor for $u$ (note that this would not be true when replacing $\phi$ with $-\phi$ due to overlapping of $\left(\mathrm{Id}+t\phi\right)(\{y>0\})$ and $\{y<0\}$): by the inner variation \eqref{eq_innervariation}, this gives $|\Delta^+ u(0)|\geq 1$. Finally, let $v$ be the biharmonic extension of $u$ in $\D_1$: then by minimality we have $E(u;\D_1)\leq E(v;\D_1)$ and since $u$ has full support in $\D_1$, then $\int_{\D_1}|\Delta u|^2\leq\int_{\D_1}|\Delta v|^2$, so $u$ must be equal to its own biharmonic extension.
\end{proof}

\begin{proposition}\label{lem_carac_blowup}
Let $u\in\M(\D_1)\setminus\{0\}$, suppose that $u$ is $2$-homogeneous. Then $u$ belongs to one of the four types of functions described in \eqref{def_uI} \eqref{def_uII}, \eqref{def_uIII}, \eqref{def_uIV}.
\end{proposition}

\begin{proof}

Since $u\in\M(\D_1)$, by Lemma \ref{lem_C1log}, $\nabla u$ is continuous. Therefore, $\Spt(u)(=\{u\neq 0\}\cup\{\nabla u \neq 0\})$ is an open cone, and on each connected component, $u$ must be a biharmonic function with Dirichlet boundary conditions. By Lemma \ref{lem_angularsector}, each angular sector (that is not the full disk $\D_1$) must have an opening $\pi,t_1$ or $2\pi$. This reduces to five different cases, up to rotation:
\begin{itemize}[label=\textbullet]
\item $u=cu^{\mathrm{I}}$ for some $c\in\R$, and Lemma \ref{lem_shapeder} gives $c=\pm 1$.
\item $u=cu^{\mathrm{II}}$ for some $c\in\R$, and Lemma \ref{lem_shapeder} gives $c=\pm 1$.
\item $u=cu^{\mathrm{III}}$ for some $c\in\R$, and Lemma \ref{lem_shapeder} gives $|c|\geq 1$.
\item $u(x,y)=c^+u^{\mathrm{I}}(x,y)\chi_{y>0}+c^-u^{\mathrm{I}}(x,-y)\chi_{y<0}$ for some $c^+,c^-\in\R$, and Lemma \ref{lem_shapeder} gives $c^+=c^-$, so $u=c^+u^{\mathrm{III}}$, and this is already covered in the previous case.
\item $\partial\Spt(u)$ is reduced to a point. In this case, $u$ is a general $2$-homogeneous biharmonic function: using it Goursat decomposition, we may write $u=(x^2+y^2)v+w$ where $v,w$ are harmonic and respectively $0$-homogeneous (meaning it is constant) and $2$-homogeneous (meaning it is a linear combination of $xy$ and $x^2-y^2$), so it is written as in \eqref{def_uIV}. Note that when $a^2=b^2+c^2$, the solution is in fact of type III. 
\end{itemize}
\end{proof}

\begin{definition}\label{def_Mhom}
We let $\M_{\mathrm{hom}}$ be the set of non-zero $2$-homogeneous functions $u$ in $\M(\D_1)$. We have a splitting
$$\M_{\mathrm{hom}}=\M_{\mathrm{hom}}^{\pi}\sqcup \M_{\mathrm{hom}}^{t_1}\sqcup \M_{\mathrm{hom}}^{2\pi},$$
such that $u\in\M_{\mathrm{hom}}^{\Theta}$ if and only if $u\in\M_{\mathrm{hom}}$ and $\Theta=2|\Spt(u)\cap\D_1|$. Similarly, we write $\mathscr{B}_{\mathrm{hom}}$ the set of non-zero minimizers of $\G(\cdot,0)$ that are invariant with respect to the $t$ variable, and we have a splitting
$$\mathscr{B}_{\mathrm{hom}}=\mathscr{B}_{\mathrm{hom}}^{\pi}\sqcup \mathscr{B}_{\mathrm{hom}}^{t_1}\sqcup \mathscr{B}_{\mathrm{hom}}^{2\pi}$$
such that $u\in \M_{\mathrm{hom}}^{\Theta}$ if and only if $u(re^{i\theta})=r^2b(\theta)$ where $b$ (seen as a function of $t,\theta$ that is constant with respect to $t$) belongs to $\mathscr{B}_{\mathrm{hom}}^\Theta$.
\end{definition}
\begin{lemma}\label{lem_separation_blowup}
There exists $d>0$ such that for any $u,v\in\M_{\mathrm{hom}}$ with $|\Spt(u)\cap\D_1|\neq |\Spt(v)\cap\D_1|$, we have $\Vert u-v\Vert_{H^1(\D_1)}\geq d$.
\end{lemma}
\begin{proof}
Suppose $|\Spt(u)\cap\D_1|<|\Spt(v)\cap\D_1|$. Since the supports of $u$ and $v$ have an opening $\pi,t_1$ (for $u$), $t_1,2\pi$ (for $v$), there exists a sector $S\subset\D_1$ of opening at least $\frac{t_1-\pi}{2}$ such that $u=0$ on $S$ and $S\subset\Spt(v)$. In particular there exists $p\in\partial\D_{\frac{1}{2}}$ such that $\D_{p,\frac{1}{10}}\subset S$: by the nondegeneracy Lemma \ref{lem_nondegeneracy} applied to $v$ on this disk we get $$\Vert u-v\Vert_{H^1(\D_1)}\geq \Vert v\Vert_{H^1(\D_{p,\frac{1}{10}})}\gtrsim 1.$$
\end{proof}

\begin{remark}\label{rem_minimality_blowup_conjecture}
We will show below that functions of type I described in \eqref{def_uI} are in fact minimizers, see remark \ref{rem_minimality_blowup}. 
We conjecture (but do not prove) that every functions of type II and III described in \eqref{def_uII}, \eqref{def_uIII} also belong to $\M(\D_1)$.

On the contrary, by the non-degeneracy Lemma \ref{lem_nondegeneracy}, we know that the function $u_{a,b,c}$ described in \eqref{def_uIV} cannot be a minimizer when the parameters $a,b,c$ are too small. In fact, in the work in preparation \cite{GM26}, H.-C. Grunau and M. Müller prove that for any $(a,b,c)\in\R^3$, the function $u_{a,b,c}$ belong to $\M(\D_1)$ when $|a|\geq \frac{1}{4}$, which includes the case of type III minimizers. The situation for $u^{\mathrm{II}}$ and $u_{a,b,c}$ with $|a|<\frac{1}{4}$ and $(b,c)\neq (0,0)$ remains open.
\end{remark}

\subsection{Existence of blow-ups}

The first use of the monotonicity formula of Theorem \ref{Th_Monotonicity_disk} is to show that any blow-up is necessarily $2$-homogeneous.

\begin{lemma}\label{lem_conv_blowup}
Let $u\in \M(\D_1)$, assume $0\in\partial\Spt(u)$, and let $r_1\geq r_2\geq r_3\geq \hdots \to 0$ such that
$$\sup_{n\in\N^*}\Vert u_{r_n}\Vert_{H^1(\D_1)}<\infty.$$
Then there exists some extraction $(n_i)_{i\geq 1}$, and some $2$-homogeneous minimizer $v$ as characterized in Lemma \ref{lem_carac_blowup} such that
\begin{align*}
u_{r_{n_i}}&\underset{i\to+\infty}{\longrightarrow}v\text{ in }H^2_\loc(\D_1),\\
\end{align*}
and $W(v,r)=W(u,0)$ for any $r\in (0,1)$. 
\end{lemma}
\begin{proof}
Since $u_{r_n}$ is bounded in $H^1(\D_1)$, in particular $u(0)=|\nabla u(0)|=0$. The sequence $(u_{r_{n}})_n$ satisfies the hypothesis of Lemma \ref{lem_sequence}, so after extraction we may suppose $u_{r_n}$ converges in $H^2_\loc(\D_1)$ to some limit $v$.
Let $\phi\in\mathcal{C}^\infty_c(]1/2,1[,\R_+)$ be such that $\int_{0}^{1}\phi(s)ds=1$, and
$$W_\phi(v,s):=\int_{0}^{1}\phi(\rho)W(v,s\rho)ds.$$
Then $W_\phi(v,s)$ may be rewritten as
\begin{align*}
W_\phi(v,s)&=\int_{0}^{1}\phi(\rho)\left(\frac{E(v;\D_{\rho s})}{(\rho s)^2}+\rho s N'(v,\rho s)+R(v,\rho s)\right)d\rho\\
&=\int_{0}^{1}\left\{\phi(\rho)\left(\frac{E(v;\D_{\rho s})}{(\rho s)^2}- N(v,\rho s)+R(v,\rho s)\right)-\rho\phi'(\rho)N(v,\rho s)\right\}d\rho.\\
\end{align*}
Since $u_{r_{n}}$ converges to $v$
in $H^2_\loc(\D_1)$ and $\chi_{u_{r_n}\neq 0}$ converges in $L^1(\D_1)$ to $\chi_{u\neq 0}$ (by Lemma \ref{lem_sequence}), then for fixed $s$:
$$W_\phi(v,s)=\lim_{n\to +\infty}W_\phi(u_{r_n},s)=\lim_{n\to +\infty}W_\phi(u,sr_n),$$
and we see the expression on the right-hand side is independent of $s$ (by the monotonicity of $W$), so $W_\phi(v,\cdot)$ is constant. Since $W(v,s/2)\leq W_\phi(v,s)\leq W(v,s)$ for any $s\in (0,1)$, this implies $W(v,\cdot)$ is constant.

By Theorem \ref{Th_Monotonicity_disk},  $v$ is therefore $2$-homogeneous. Moreover, $v$ cannot be zero, since otherwise $u_{r_n}$ would be zero in $\D_{1/2}$ for some large enough $n$ by the nondegeneracy Lemma \ref{lem_nondegeneracy}, which would contradict the fact that $0\in\partial\Spt(u)$. As a consequence, by Lemma \ref{lem_carac_blowup}, $v$ is equal to one of the four types of solutions described in \eqref{def_uI}, \eqref{def_uII}, \eqref{def_uIII}, \eqref{def_uIV}.
\end{proof}

Our goal in the rest of the section is to prove the two following results, on the existence of blow-ups under some a priori information on the solution.

\begin{proposition}\label{prop_existence_blowup}
Let $u\in \M(\D_1)$ such that 
$$\limsup_{r\to 0}\frac{|\{u=0\}\cap\D_{r}|}{|\D_{r}|}>0.$$
Then there exists a blow-up of type I or II at the origin.
\end{proposition}

When there is no such density estimate of $\{u=0\}$, we do not know yet whether a blow-up exists. However, we can prove that there is a nonzero ``renormalized'' blow-up in the following sense.

\begin{proposition}\label{prop_renormalized_blowup}
Let $u\in\M(\D_1)$ such that $u(0)=|\nabla u(0)|=0$ and $(r_n)_{n\in\N^*}$ a sequence converging to $0$ such that
$$\lim_{r\to 0}\Vert u_{r_n}\Vert_{H^1(\D_1)}=+\infty.$$
Then there exists some subsequence $r_{n_k}$, and some nonzero $2$-homogeneous biharmonic function $v$, such that
$$\frac{u_{r_{n_k}}}{\Vert u_{r_{n_k}}\Vert_{H^1(\D_1)}}\underset{k\to +\infty}{\longrightarrow}v \text{ in }H^2_\loc(\R^2).$$
\end{proposition}

\begin{remark}\label{rem_minimality_blowup}
We may now prove that the function defined in \eqref{def_uI} belongs to $\M(\D_1)$ by a blow-up argument. Indeed, for every small $\eps>0$ consider $u^{\eps}\in\M(\D_1)$ verifying $(u^{\eps},\partial_r u^{\eps})=(\eps,0)$ on $\partial\D_1$. Such minimizer exists by standard arguments. Let $\eta\in\mathcal{C}^\infty_c(\R,\R)$ with $\eta\equiv 1$ near $0$, then
$$E(u^{\eps};\D_1)\leq E\left(\eps\eta\left(\frac{1-|\cdot|}{\sqrt{\eps}}\right);\D_1\right)=\mathcal{O}_{\eps\to 0}\left(\sqrt{\eps}\right).$$
So by the nondegeneracy Lemma \ref{lem_nondegeneracy}, $u^{\eps}=0$ in $\D_{\frac{1}{2}}$ for a sufficiently small $\eps>0$ that we fix. Let $r>0$ be the largest radius such that $u^{\eps}=0$ in $\D_r$, and $p\in\partial\D_{r}\cap\partial\Spt(u^\eps)$. Since $\Spt(u^\eps)$ has density at most $\frac{1}{2}$ at $p$, then by Proposition \ref{prop_existence_blowup}, $u^\eps(p+\cdot)$ admits a blow-up of type I, which must then be a minimizer.
\end{remark}

The proofs of Proposition \ref{prop_existence_blowup}, \ref{prop_renormalized_blowup} rely on the structure of $W(u,r)$. We remind that by Theorem \ref{Th_Monotonicity_disk},  $W(u,r)$ is decomposed as
$$W(u,r):=\frac{1}{r^2}E(u;\D_r)+rN'(u,r)+R(u,r),$$
where $N,R$ are defined in \eqref{eq_def_N}, \eqref{eq_def_R}. Since $N(u,r)\gtrsim \Vert u_r\Vert_{L^2(\partial\D_1)}^2+\Vert \nabla u_r\Vert_{L^2(\partial\D_1)}^2\gtrsim R(u,r)$, there exists some universal constant $\kappa>0$ (which could be made explicit, although we will not need it) such that for any function $u\in H^2(\D_1)$, $r\in (0,1)$:
\begin{equation}\label{def_kappa}
 |R(u,r)|\leq \kappa N(u,r).
\end{equation}

The next lemma uses this decomposition of $W(u,r)$ to ensure that for any $u\in\M(\D_1)$, $u_{r}$ cannot be too small compared to $u$ for $r$ close to $1$, in a uniform way.
\begin{lemma}\label{lem_boundedgrowth}
Let $u\in \M(\D_1)$ such that $u(0)=|\nabla u(0)|=0$, $r\in (0,1)$, $\kappa$ the constant from \eqref{def_kappa}. Then
\[N(u,r)\geq r^\kappa N(u,1)-\frac{1-r^\kappa}{\kappa}W(u,1).\]
\end{lemma}
A scale-invariant statement of this lemma is that for any $s<r$, $u\in\M(\D_r)$ such that $u(0)=|\nabla u(0)|=0$:
\[N(u,s)\geq \left(\frac{s}{r}\right)^\kappa N(u,r)-\frac{1-\left(\frac{s}{r}\right)^\kappa}{\kappa}W(u,r).\]
This is obtained by applying Lemma \ref{lem_boundedgrowth} to $u_r$ with a radius $\frac{s}{r}$.
\begin{proof}[Proof of Lemma \ref{lem_boundedgrowth}]
We start from $W(u,r)\leq W(u,1)$, (given by the monotonicity from Theorem \ref{Th_Monotonicity_disk}), and we decompose $W(u,r)$ as:
\[\frac{1}{r^2}E(u;\D_r)+rN'(u,r)+R(u,r)\leq W(u,1).\]
Using the estimate $|R|\leq\kappa N$ (see \eqref{def_kappa}), and the fact that the energy $E$ is nonnegative, we get
\[rN'(u,r)-\kappa N(u,r)\leq W(u,1)\]
which may be rewritten as
\[\frac{d}{dr}\frac{N(u,r)}{r^\kappa}\leq \frac{W(u,1)}{r^{\kappa+1}}.\]
We integrate this from $r(<1)$ to $1$:
\[N(u,1)-\frac{N(u,r)}{r^\kappa}\leq \left(\frac{1}{r^\kappa}-1\right)\frac{W(u,1)}{\kappa},\]
which is the result.
\end{proof}

\begin{lemma}\label{lem_nondegsequences}
There exists $\sigma>0$ such that the following holds: let $(u^{(n)})_n$ be a sequence in $\M(\D_1)$ such that
$$u^{(n)}(0)=|\nabla u^{(n)}(0)|=0,\qquad N(u^{(n)},1)\underset{n\to +\infty}{\longrightarrow}+\infty,\qquad W(u^{(n)},1)\leq \sigma N(u^{(n)},1).$$
Then there exists a subsequence $(n_k)_{k\geq 1}$ and a nonzero biharmonic function $v:\D_1\to\R$ such that
$$\frac{u^{(n_k)}}{\Vert u^{(n_k)}\Vert_{H^1(\D_1)}}\underset{k\to +\infty}{\longrightarrow}v\text{ in }H^2_\loc(\D_1).$$
\end{lemma}

\begin{proof}
Let $(u^{(n)})$ be such a sequence, for a value $\sigma>0$ that will be fixed sufficiently small. By Lemma \ref{lem_boundedgrowth} we have, for all $r\in \left[\frac{1}{2},1\right]$, 
\[N(u^{(n)},r)\geq 2^{-\kappa}N(u^{(n)},1)-\frac{1-2^{-\kappa}}{\kappa}W(u^{(n)},1)\geq 2^{-\kappa-1}N(u^{(n)},1)\text{ for a small enough }\sigma.\]
Let $M_n=2\int_{\frac{1}{2}}^{1}N(u^{(n)},r)dr$. Then $M_n\geq 2^{-\kappa-1}N(u^{(n)},1)\to +\infty$ and by the Cacciopoli inequality of Lemma \ref{lem_Cacciopoli} we have $$\Vert u^{(n)}\Vert_{H^1(\D_1)}^2\lesssim M_n\lesssim \Vert u^{(n)}\Vert_{H^1(\D_1)}^2.$$
It is enough to prove that $M_n^{-1/2}u^{(n)}$ converges to some non-zero biharmonic function. Since this sequence is bounded in $H^1(\D_1)$, and $M_n^{-1/2}u^{(n)}\in\M_{M_n^{-1}}(\D_1)$, by Lemma \ref{lem_sequence} it admits (after extraction of a subsequence $n_k$) some limit $v$ in the $H^2_\loc(\D_1)$ sense such that $v\in\M_0(\D_1)$ (meaning $v$ is biharmonic).

We now prove that $v$ is not identically zero. There exists some $r_n\in \left[\frac{1}{2},1\right]$ such that $M_n=N(u^{(n)},r_n)$, and by application of Lemma \ref{lem_boundedgrowth} we have, for all $r\in \left[\frac{1}{4},\frac{1}{2}\right](\subset \left[\frac{r_n}{4},r_n\right])$,
\[N(u^{(n)},r)\geq 4^{-\kappa}M_n-\frac{1-4^{-\kappa}}{\kappa}W(u^{(n)},1)\geq 4^{-\kappa-1}M_n\text{ for a small enough }\sigma.\]
In particular, $\Vert M_n^{-1/2}u^{(n)}\Vert_{H^1(\D_{1/2})}^2$ is bounded from below by a positive constant. Since the convergence of $M_{n_k}^{-1/2}u^{(n_k)}$ to $v$ is strong in $H^1(\D_{1/2})$, necessarily $v$ is non-zero.
\end{proof}

The next lemma is central for the existence of blow-ups: it states that at points where $\Spt(u)^c$ has a positive density, $W(u,r)$ controls $N(u,r)$ from above. Since $r\mapsto W(u,r)$ is nondecreasing, this will imply a uniform bound of $N(u,r)$ as $r\to 0$.
\begin{lemma}\label{lem_Ncontrol}
Let $\delta\in (0,1)$, there exists $C_\delta>0$ such that for any $u\in \M(\D_1)$ with $u(0)=|\nabla u (0)|=0$ and
\[|\{u\neq 0\}\cap\D_1|\leq (1-\delta)|\D_1|,\]
we have
\[N(u,1)\leq C_\delta(1+W(u,1)_+).\]
\end{lemma}

Note that this lemma fails without the hypothesis on the support: we expect $u:(x,y)\mapsto \lambda y^2$ to be a solution in $\M(\D_1)$ for any $\lambda\geq 1$, but $W(u,1)=\pi$, whereas $N(u,1)= c\lambda^2$ for some $c>0$, so the estimate cannot hold for large enough $\lambda$.

\begin{proof}
Suppose that for some sequence of minimizers $u^{(n)}$ verifying this density estimate we have $N(u^{(n)},1)\geq n(1+W(u^{(n)},1)_+)$. In particular Lemma \ref{lem_nondegsequences} applies for large enough $n$: up to the extraction of a subsequence, $\frac{u^{(n)}}{\Vert u^{(n)}\Vert_{H^1(\D_1)}}$ converges in $H^2_\loc(\D_1)$ to a non-zero biharmonic function $v$.

By Fatou's lemma we have $|\{v=0\}|\geq \limsup_{n\to +\infty}|\{u^{(n)}=0\}|\geq \delta$, which is a contradiction since, for instance, any biharmonic function is analytic, so its support must have full measure.
\end{proof}

We now deduce the proofs of Proposition \ref{prop_existence_blowup} and \ref{prop_renormalized_blowup}.
\begin{proof}[Proof of Proposition \ref{prop_existence_blowup}]
Let $r_n\to 0$ be a sequence such that for some $\delta\in\left( 0,\frac{1}{2}\right)$, we have
$$\frac{|\{u=0\}\cap\D_{r_n}|}{|\D_{r_n}|}>2\delta.$$
Let $\tau=\sqrt{1-\delta}$, such that for any $r\in (\tau r_n,r_n)$, we have
$$\frac{|\{u=0\}\cap\D_{r}|}{|\D_{r}|}>\delta.$$

Then for every such $n$, $r\in (\tau r_n,r_n)$, we have by Lemma \ref{lem_Ncontrol}
$$N(u,r)\leq C_\delta (1+W(u,r)_+)\leq C_\delta (1+W(u,1)_+),$$
where $C_\delta$ is the constant from Lemma \ref{lem_Ncontrol}. We integrate this in $r\in (\tau r_n,r_n)$:
$$\Vert u_{r_n}\Vert_{H^1(\D_{1}\setminus\D_\tau)}^2\lesssim \int_{\tau }^{1}N(u,sr_n)ds\lesssim \delta C_\delta(1+W(u,1)_+),$$
and by the Cacciopoli inequality of Lemma \ref{lem_Cacciopoli}, $\Vert u_{r_n}\Vert_{H^1(\D_\tau)}$ is also bounded independently of $n$. So we can apply Lemma \ref{lem_conv_blowup} to the sequence $u_{r_n}$, and there exists some subsequence $u_{r_{n_k}}$ that converges in $H^2_\loc(\D_1)$ to some $v\in\M_{\mathrm{hom}}$. Since the density estimate is still verified at the limit, necessarily $v\in\M_{\mathrm{hom}}^{\mathrm{I}}\sqcup \M_{\mathrm{hom}}^{\mathrm{II}}$.
\end{proof}

\begin{proof}[Proof of Proposition \ref{prop_renormalized_blowup}]
By the Cacciopoli inequality of Lemma \ref{lem_Cacciopoli},
$$\Vert u_{r_n}\Vert_{H^1(\D_1)}^2\lesssim \int_{r_n/2}^{r_n}N(u,s)\frac{ds}{s}\lesssim \Vert u_{r_n}\Vert_{H^1(\D_1)}^2,$$
so there exists some $s_n\in \left[\frac{r_n}{2},r_n\right]$ such that
$$\Vert u_{r_n}\Vert_{H^1(\D_1)}^2\lesssim N(u,s_n)\lesssim \Vert u_{r_n}\Vert_{H^1(\D_1)}^2\underset{n\to+\infty}{\longrightarrow}+\infty.$$
Since $W(u_{s_n},1)=W(u,s_n)\leq W(u,1)$, Lemma \ref{lem_nondegsequences} applies to $u_{s_n}$ for large enough $n$: up to extracting a subsequence, we may suppose that  $u_{s_{n}}/\Vert u_{s_{n}}\Vert_{H^1(\D_1)}$ converges in $H^2_\loc(\D_1)$ to some nonzero biharmonic function $v$.

Let $R>2$: $u_{s_n}$ is well-defined in $\D_R$ for large enough $n$, and we now prove that up to extracting another subsequence, $u_{s_n}$ converges to $v$ in $H^2_\loc(\D_R)$. By Lemma \ref{lem_boundedgrowth} applied to $u$ with the radii $s_n$, $s_nr$ for $r\in[R/2,R]$, we have 
\[N(u,s_n)\geq R^{-\kappa} N(u,s_nr)-\frac{1-R^{-\kappa}}{\kappa}W(u,s_nr)_+.\]
By integrating this estimate in $r\in\left[\frac{R}{2},R\right]$ we get
\[\Vert u_{s_n}\Vert_{H^1(\D_{R}\setminus\D_{\frac{R}{2}})}^2\lesssim R^\kappa\left( N(u,s_n)+W(u,1)_+\right).\]
Thus $u_{s_n}$ is locally bounded in $H^1(\D_R)$ for any $R>0$: up to extracting a subsequence, we obtain that $u_{s_n}/\Vert u_{s_n}\Vert_{H^1(\D_1)}$ converges in $H^2_\loc(\R^2)$ to $v$.

Next we prove that $v$ is a sum of $k$-homogeneous biharmonic functions for $k\in\{0,1,2\}$. Since $v$ is biharmonic, it admits a decomposition (as in \eqref{eq_Goursat}):
$$
v\left(re^{i\theta}\right)=\sum_{n\in\Z}\left(a_n r^{|n|+2}+b_nr^{|n|}\right)e^{in\theta},
$$
for some complex coefficients $(a_n)_{n\in\Z},(b_n)_{n\in\Z}$.
Proving $v$ is a sum of $k$-homogeneous function for $k\leq 2$ is thus equivalent to proving that every coefficient other than $a_0$, $b_0$, $b_{\pm 1}$, $b_{\pm 2}$ vanishes.

For any $r>0$, we have
\begin{align*}
W_0(v,r)&=\lim_{n\to +\infty} W_0\left(\frac{u_{s_n}}{\Vert u_{s_n}\Vert_{H^1(\D_1)}},r\right)=\lim_{n\to +\infty}\frac{W_0(u_{s_n},r)}{\Vert u_{s_n}\Vert_{H^1(\D_1)}^2}\\
&\leq\limsup_{n\to +\infty} \frac{W(u,s_nr)}{\Vert u_{s_n}\Vert_{H^1(\D_1)}^2}\leq\limsup_{n\to +\infty}\frac{W(u,1)}{\Vert u_{s_n}\Vert_{H^1(\D_1)}^2}\\
&=0\text{ since }\Vert u_{s_n}\Vert_{H^1(\D_1)}\to +\infty.
\end{align*}
By Lemma \ref{lem_computationWbih} we have the explicit expression
$$
W_0(v,r)=8\pi\sum_{n\in\Z}H_n\left(r^{|n|}a_n,r^{|n|-2}b_n\right),
$$
where $H_n$ is the Hermitian form
$$H_n(z,w)=|n|^3|z|^2+2|n|^2(|n|-2)\Re(\ov{z}w)+(|n|-2)(|n|^2-2|n|+2)|w|^2.$$
Note that $H_n$ is positive definite for $|n|\geq 3$; using this, we get
\begin{align*}
\forall r>0,\ 0\geq& \frac{W_0(v,r)}{8\pi}\geq\sum_{|n|\leq 2}H_n\left(r^{|n|}a_n,r^{|n|-2}b_n\right)\\
=&-4r^{-4}|b_0|^2+2r^2\left(|a_1|^2+|a_{-1}|^2\right)-2\left(\Re(\ov{a_1}b_1)+\Re(\ov{a_{-1}}b_{-1})\right)\\
&-r^{-2}\left(|b_1|^2+|b_{-1}|^2\right)+8r^4\left(|a_2|^2+|a_{-2}|^2\right),
\end{align*}
so $a_{\pm 2}=a_{\pm 1}=0$, and
$$\sum_{|n|\leq 2}H_n\left(r^{|n|}a_n,r^{|n|-2}b_n\right)=-4r^{-4}|b_0|^2-r^{-2}\left(|b_1|^2+|b_{-1}|^2\right).$$
Now, for any $N\in\Z$ such that $|N|\geq 3$ we have 
$$0\geq \frac{W_0(v,r)}{8\pi}\geq H_N\left(r^{|N|}a_N,r^{|N|-2}b_N\right)-4r^{-4}|b_0|^2-r^{-2}\left(|b_1|^2+|b_{-1}|^2\right).$$
If $a_N\neq 0$, the highest degree term of the right-hand side is $|N|^3|a_N|^2r^{2|N|}$, which is a contradiction as $r\to +\infty$. So necessarily $a_N=0$, and the previous inequality becomes
$$0\geq (|N|-2)(|N|^2-2|N|+2)|b_N|^2r^{2(|N|-2)}-4r^{-4}|b_0|^2-r^{-2}\left(|b_1|^2+|b_{-1}|^2\right).$$
Since $|N|\geq 3$, then necessarily $b_N=0$ (otherwise the inequality above does not hold as $r\to +\infty$). This ends the proof that $v$ is a sum of $k$-homogeneous biharmonic polynomials for $k\leq 2$.

Finally, we prove that $v$ is $2$-homogeneous. Indeed, for any $n\in\N^*$, for any $p\in\D_{\frac{1}{2}}$, by Lemma \ref{lem_C1log} applied to $u_{r_n}$ we have
$$|u_{r_n}(p)|\leq C |p|^2\log\left(\frac{1}{|p|}\right)\Vert u_{r_n}\Vert_{H^1(\D_1)},$$
for some universal constant $C>0$. As a consequence, we have for any $p\in\D_{\frac{1}{2}}$:
$$|v(p)|\leq C |p|^2\log\left(\frac{1}{|p|}\right).$$
Since $v$ is already known to be a polynomial of degree $2$, this implies that $v$ is $2$-homogeneous.
\end{proof}

\section{Epiperimetric inequality}\label{sec_epi}

Our goal is now to prove a quantitative version of the monotonicity formula of Theorems \ref{Th_Monotonicity_exp} and \ref{Th_Monotonicity_disk}. 
In the classical Alt-Caffarelli problem (see for instance \cite{SV19}, or \cite[Lem. 12.14.]{V23}), this is based mainly on two observations: first, the derivative of the normalized energy $$W^{\mathrm{AC}}(u,r)=\frac{1}{r^2}\int_{\D_r}\left(|\nabla u|^2+\chi_{u>0}\right)-\frac{1}{r^3}\int_{\partial\D_r}u^2,$$
with respect to $r$ is bounded from below by $\frac{W^{\mathrm{AC}}(u^r,r)-W^{\mathrm{AC}}(u,r)}{r}$, where $u^r$ is the $1$-homogeneous extension of $u$ in $\D_r$. The second observation is that for some universal constant $\eta>0$, we have
$$W^{\mathrm{AC}}(u,r)<(1-\eta) W^{\mathrm{AC}}(u^r,r)+\eta\frac{\pi}{2}.$$
Here $\frac{\pi}{2}$ is the energy of homogeneous minimizers. This leads to a differential inequality of the form
$$\frac{d}{dr}\left(W^{\mathrm{AC}}(u,r)-\frac{\pi}{2}\right)\geq \frac{\eta}{(1-\eta)r}\left(W^{\mathrm{AC}}(u,r)-\frac{\pi}{2}\right).$$
When $W^{\mathrm{AC}}(u,0)\geq\frac{\pi}{2}$, this integrates to an explicit rate of convergence 
$$\forall r\in (0,1),\ W^{\mathrm{AC}}(u,r)-\frac{\pi}{2}\leq r^\frac{\eta}{1-\eta} \left(W^{\mathrm{AC}}(u,1)-\frac{\pi}{2}\right).$$
See \cite[Lem. 12.14]{V23} for a more extensive discussion on the epiperimetric inequality and on how to deduce geometric information from this decay.

Neither of these observations is known in our setting: the (2-)homogeneous extension of $u$ in $\D_r$ is generally not a valid competitor for $u$ (since the radial derivatives may not match on each side of $\partial\D_r$, so the $2$-homogeneous extension does not belong to $H^2$), and it is unclear whether the derivative of $W(u,r)$ is bounded from below by $\frac{W(u^r,r)-W(u,r)}{r}$, where $u^r$ would be some other ``natural'' competitor for $u$ in $\D_r$.

Instead of proving a differential inequality on $W(u,r)$, we prove directly the growth rate estimate
$$W(u,e^{-1})-\frac{\Theta}{2}\leq (1-\eta)\left(W(u,1)-\frac{\Theta}{2}\right),$$
for some universal constant $\eta\in (0,1)$, when $u$ is sufficiently close to some homogeneous solution of opening $\Theta\in\{\pi,t_1,2\pi\}$. Here $e^{-1}$ does not play any particular role and could be replaced by any constant strictly lower than $1$.

We also obtain a control of the variation $\Vert u_r-u_s\Vert_{H^1(\D_1)}$ at different scales $0<s<r<1$ by $\sqrt{W(u,1)-\frac{\Theta}{2}}$. This will be enough to obtain the uniqueness and rate of convergence of blow-ups.

\begin{theorem}\label{th_epiperimetry}
Let $\Theta\in\{\pi,t_1,2\pi\}$. There exist constants $c_1>0$, $\eta\in(0,1)$ such that the following holds: let $u\in \M(\D_1)$, $
\ov{u}\in \M_{\mathrm{hom}}^{\Theta}$ if $\Theta\in\{\pi,t_1\}$ (resp. $\ov{u}\in\mathrm{Span}(x^2+y^2,x^2-y^2,xy)$ if $\Theta=2\pi$) such that
$$\Vert u-\ov{u}\Vert_{H^1(\D_1\setminus\D_{e^{-1}})}\leq c_1\Vert \ov{u}\Vert_{H^1(\D_1)},\quad\text{ and }\quad W(u,0)\geq \frac{\Theta}{2}.$$
Then
$$
W(u,e^{-1})-\frac{\Theta}{2}\leq (1-\eta)\left( W(u,1)-\frac{\Theta}{2}\right),$$
and
$$
\Vert u-u_{e^{-1}}\Vert_{H^1(\D_1\setminus\D_{e^{-1}})}\leq C_1\sqrt{W(u,1)-\frac{\Theta}{2}}.
$$
\end{theorem}
Similarly to Section \ref{sec_mono}, it is more convenient to first prove the epiperimetric inequality in exponential coordinates. In all that follows, a function of $\theta$ will also be seen, by an abuse of notation, as a function of $(t,\theta)$ that is constant with respect to $t$.
\begin{theorem}\label{th_epiperimetry_exp}
Let $\Theta\in\{\pi,t_1,2\pi\}$. There exist constants $c_0>0$, $\eta\in(0,1)$, such that the following holds: let $u\in H^2_\lin(\Cy_0)$ be a minimizer of $\G(\cdot,0)$, assume that there exists $b\in\mathscr{B}_{\mathrm{hom}}^{\Theta}$ if $\Theta\in\{\pi,t_1\}$ (resp. $b\in\mathrm{Span}(1,\cos(2\theta),\sin(2\theta))$ if $\Theta=2\pi$) such that
\begin{equation}\label{epiperimetry_firsthypothesis}
\Vert u-b\Vert_{H^1(\Cy_0\setminus\Cy_1)}\leq c_0\Vert b\Vert_{H^1(\Sp)}\quad\text{ and }\quad\W(u,1)\geq \frac{\Theta}{2}.
\end{equation}
Then
\[\W(u,1)\leq \W(u,0)-\eta\left(\left(\W(u,0)-\frac{\Theta}{2}\right) +\Vert \partial_t u\Vert_{H^1(\Cy_0\setminus\Cy_1)}^2\right).\]
\end{theorem}

Theorem \ref{th_epiperimetry} will be obtained as a direct consequence of Theorem \ref{th_epiperimetry_exp}. A central point in the statement of Theorem \ref{th_epiperimetry_exp} is that the monotonicity formula of Theorem \ref{Th_Monotonicity_exp} only gives a control
$$\Vert\partial_{t,t} u\Vert_{L^2(\Cy_0\setminus\Cy_1)}^2+\Vert\partial_{t,\theta} u\Vert_{L^2(\Cy_0\setminus\Cy_1)}^2\leq \frac{\W(u,0)-\W(u,1)}{4},$$
and it is unclear whether $\Vert\partial_t u\Vert_{L^2(\Cy_0\setminus\Cy_1)}^2$ is also controlled by the right-hand side, which explains the presence of the term $\|\partial_t u\|_{H^1(\Cy_0\setminus\Cy_1)}^2$ in the conclusion of Theorem \ref{th_epiperimetry_exp}). When $u$ vanishes on a sufficiently large subset  of  $\Cy_0\setminus\Cy_1$ we may apply a Poincaré inequality to $\partial_t u$, whereas the case where the support of $u$ is \textit{close} to $\Cy_0\setminus\Cy_1$ will be handled separately.

The proof of Theorem \ref{th_epiperimetry_exp} relies on a comparison of the solution $u$ with an appropriate competitor. The construction of the competitor is summarized in the two lemmas below, and their proof will occupy subsections \ref{subsec_singlemode} to \ref{subsec_biharmonic}.

We recall that $\Spt(u)$ designates every point where either $u$ or $\nabla u$ does not vanish, and we restrict it to $\Sp$ for functions only depending on $\theta$.
\begin{lemma}\label{lem_competitor}
Let $\Theta\in\{\pi,t_1,2\pi\}$. There exist $c,\eps,C>0$ such that for any $u\in H^2(\Sp)$, $v\in H^1(\Sp)$, with $\{v\neq 0\}\subset\Spt(u)$ and
\begin{equation}\label{eq_supportcondition}
\begin{cases}
(c,\Theta-c)\subset \Spt(u)\subset (-c,\Theta+c)&\text{ when }\Theta\in\{\pi,t_1\},\\
(c,\pi-c)\cup(\pi+c,2\pi-c)\subset \Spt(u)\subset (0,2\pi)&\text{ when }\Theta=2\pi,\\
\end{cases}
\end{equation}
 there exists $U\in H^2_\lin(\Cy_0)$ such that $U(0,\theta)=u(\theta)$, $\partial_{t}U(0,\theta)=v(\theta)$ and
\begin{equation}\label{eq_epiest}
\G(U,0)\leq \G(u,0)-\eps\left(\G(u,0)-\frac{\Theta}{2}\right)_++C\int_{\Sp}(\partial_\theta v)^2.
\end{equation}
\end{lemma}

Here $\G(u,0)= \frac{1}{2}\int_{\Sp}\left(( \partial_{\theta,\theta}u)^2-4( \partial_\theta u)^2+\chi_{\{u\neq 0\}}\right)$, since $(t,\theta)\mapsto u(\theta)$ is constant with respect to $t$ ; in a way we are comparing the energy of $U(t,\theta)$ with the energy of the homogeneous extension $(t,\theta)\mapsto u(\theta)$, and the last term takes into account the effect of the boundary condition $\partial_t U(0,\theta)=v(\theta)$ on the estimate.

The support condition \eqref{eq_supportcondition}, and more particularly the case disjunction on the value of $\Theta$, is based on the following observation (see the proof of Theorem \ref{th_epiperimetry_exp}): assume $u\in H^2_\lin(\Cy_0)$ is a minimizer of $\G(\cdot,0)$, such that $\frac{\Vert u-b\Vert_{H^1(\Cy_0\setminus\Cy_1)}}{\Vert b\Vert_{H^1(\Sp)}}$ is small for some $b\in\mathscr{B}_{\mathrm{hom}}^\Theta$ with $\Theta\in\{\pi,t_1,2\pi\}$:
\begin{itemize}
    \item 
if $\Theta\in\{\pi,t_1\}$, then up to a translation the support of $b$ is $(0,\Theta)$, and we may prove that the support of $u(t,\cdot)$ for $t\in \left(\frac{1}{4},\frac{3}{4}\right)$ verifies the support assumption \eqref{eq_supportcondition}.
\item
if $\Theta=2\pi$, then up to a translation the support of $b$ is either $\Sp$ or $(0,\pi)\cup(\pi,2\pi)$; when $u$ is sufficiently close to $b$, the support of $u(t,\cdot)$ contains $(c,\pi-c)\cup(\pi+c,2\pi-c)$. Lemma \ref{lem_competitor} treats this case under the additional assumption that $\Spt(u(t,\cdot))$ is not $\Sp$. The case where $\Spt(u)=\Sp$ is purposefully set aside, and is treated separately by the following lemma.
\end{itemize}

\begin{lemma}\label{lem_bihext}
Let $v\in H^2_\lin(\Cy_0)$ be a nonzero minimizer of $\G(\cdot,0)$ such that $|\Spt(v(0,\cdot))|=2\pi$, then
$$
\G(v,0)\leq  \G(v(0,\cdot),0)-\frac{1}{10}\left(\G(v(0,\cdot),0)-\pi\right)_++6\Vert\partial_{t,\theta}v(0,\cdot)\Vert^2-2\Vert\partial_{t}v(0,\cdot)\Vert^2.
$$
\end{lemma}

The proof of Lemma \ref{lem_competitor} is broken down into several steps, where we construct competitors for increasingly complex boundary conditions $(u,v)$:
\begin{itemize}
\item In subsection \ref{subsec_singlemode}, we treat the case where $v=0$, and $u$ is a single one-dimensional buckling mode (see subsection \ref{subsec_buckling_basis} for an introduction of this basis).
\item In subsection \ref{subsec_highermode}, we suppose $v=0$, for a general $u$.
\item In subsection \ref{subsec_partialt}, we remove the $v=0$ hypothesis and prove the full estimate of \eqref{eq_epiest}.
\end{itemize}

Lemma \ref{lem_bihext} is obtained by a direct comparison with the biharmonic extension in the disk model, in subsection \ref{subsec_biharmonic}.

Finally, in subsection \ref{subsec_conclusion_epiperimetry} we deduce the epiperimetric inequality from Lemmas \ref{lem_competitor} and \ref{lem_bihext}.

\subsection{Buckling eigenbasis}\label{subsec_buckling_basis}

We will denote by $\Vert\cdot\Vert$ (resp. $\langle\cdot,\cdot\rangle$) the $L^2(\Sp)$ norm (resp. scalar product). Let $(b_{n})_{n\geq 1}$ be an eigenbasis associated to the buckling operator in the segment $[0,\pi]$: in other words it is an orthonormal basis of $H^1_0([0,\pi])$ for the scalar product $\langle \partial_\theta b_n,\partial_\theta b_m\rangle=\delta_{n,m}$, such that $b_n\in H^2_0([0,\pi])$ and $\partial_\theta^4 b_n=-\mu_n \partial_\theta^2 b_n$ where $\mu_n$ is the $n$-th buckling eigenvalue. The eigenvalues $(\mu_n)_{n\geq 1}$ can be defined as the $n$-th positive root of the function
\begin{equation}\label{eq_numerical_eigenvalues}
\mu\mapsto \sin\left(\frac{\pi \sqrt{\mu}}{2}\right)\left(\pi\sqrt{\mu} \cos\left(\frac{\pi \sqrt{\mu}}{2}\right)-2\sin\left(\frac{\pi \sqrt{\mu}}{2}\right) \right).
\end{equation}
In particular $\mu_n=(n+1)^2$ for odd $n$, and $n^2<\mu_n<(n+2)^2$ for even $n$. The first few values are
$$\mu_1=4,\ \mu_2=\left(\frac{2t_1}{\pi}\right)^2\approx 8.183,\ \mu_3=16,\ \mu_4\approx 24.187,\ \hdots$$
and $b_{1}(\theta)=\sqrt{\frac{2}{\pi}}\sin(\theta)_+^2$. For any $\om>0$, we denote
\begin{equation}\label{eq_exprb_nmu_n}
b_{n,\om}(\theta)=\sqrt{\frac{\om}{\pi}}b_n\left(\frac{\pi}{\om}\theta\right),\ \mu_{n,\om}=\frac{\pi^2}{\om^2}\mu_n
\end{equation}
Such that $b_{n,\om}$ is a $\langle\partial_\theta\cdot,\partial_\theta\cdot\rangle$-orthonormal basis of $H^1_0([0,\om])$, such that $b_{n,\om}\in H^2_0([0,\om])$ and $\partial_\theta^4 b_{n,\om}=-\mu_{n,\om}\partial_{\theta}^2b_{n,\om}$.
A central point is that $(b_{n,\om})$ is \textbf{not} orthogonal for the $L^2([0,\om])$ scalar product ($\langle b_1,b_3\rangle\neq 0$ for instance, since both do not change sign). We will use the following estimate.

\begin{lemma}\label{lem_estbn}
For any $\om\in (0,2\pi]$, for any $n\geq 1$, we have
$$ \Vert b_{n,\om}\Vert^2\lesssim \frac{1}{\mu_{n,\om}}.$$
\end{lemma}

\begin{proof}
It is enough to prove this for $\om=\pi$, by change of variable.

When $n$ is odd, we have $\mu_{n}=(n+1)^2$ and
$$b_{n}(\theta)=\sqrt{\frac{2}{\pi}}\frac{1-\cos((n+1)\theta)}{n+1}$$
And so $\Vert b_n\Vert^2\lesssim \frac{1}{(n+1)^2}=\frac{1}{\mu_n}$.

When $n$ is even, the computation is simpler on the interval $\left[-\frac{\pi}{2},+\frac{\pi}{2}\right]$. We know that $\mu_{n}\in \left( n^2,(n+2)^2\right)$ and $\theta\mapsto b_n\left(\theta+\frac{\pi}{2}\right)$ is a scalar multiple of the function
$$f_n(\theta)=\sin\left(\sqrt{\mu_n}\theta\right)-\beta_n \theta$$
where $\beta_n\in\R$ is chosen such that $f_n\left(\pm\frac{\pi}{2}\right)=f_n'\left(\pm\frac{\pi}{2}\right)=0$, meaning
\begin{equation}\label{eq_defbeta_n}
\sin\left(\frac{\pi\sqrt{\mu_n}}{2}\right)=\beta_n\frac{\pi}{2},\qquad \sqrt{\mu_n}\cos\left(\frac{\pi\sqrt{\mu_n}}{2}\right)=\beta_n.
\end{equation}
The first condition gives $|\beta_n|\lesssim 1$, so $|f_n|\lesssim 1$. Then, using these two estimates,
$$\Vert b_n\Vert^2=\frac{\Vert f_n\Vert^2}{\Vert\partial_\theta f_n\Vert^2}\lesssim \frac{1}{\int_{-\pi/2}^{+\pi/2}\left(\sqrt{\mu_n}\sin(\sqrt{\mu_n}\theta)-\beta_n\right)^2d\theta}\lesssim \frac{1}{\mu_n}.$$
\end{proof}

From the expression \eqref{eq_exprb_nmu_n}, we see that for $\om\in (0,2\pi]$, $\mu_{n,\om}=4$ is only verified in three cases:
\begin{itemize}
\item[(I) ]$\om=\pi$, $n=1$. Then we have $\mu_{1,\pi}=4$, $\mu_{2,\pi}\approx 8.187$.
\item[(II) ]$\om=t_1$, $n=2$. Then we have $\mu_{1,t_1}\approx 1.955$, $\mu_{2,t_1}=4$, $\mu_{3,t_1}\approx 7.821$.
\item[(III) ]$\om=2\pi$, $n=3$. Then we have $\mu_{2,2\pi}\approx 2.046$, $\mu_{3,2\pi}=4$, $\mu_{4,2\pi}\approx 6.047$.
\end{itemize}
For $\Theta\in \{\pi,t_1,2\pi\}$, we will denote
$$i(\Theta)=\begin{cases}1&\text{ if }\Theta=\pi,\\ 2&\text{ if }\Theta=t_1,\\ 3&\text{ if }\Theta=2\pi,\end{cases}$$
and the numerical values above give us the following estimate: for any $\om\in (0,2\pi]$ such that $|\om-\Theta|\leq 0.1\Theta$, $n\in\N^*$, we have
\begin{equation}\label{eq_gapestimate}
n<i(\Theta)\Rightarrow \mu_{n,\om}\leq 3,\qquad n>i(\Theta)\Rightarrow \mu_{n,\om}\geq 5.
\end{equation}
Similarly, if $\om\leq \frac{\pi}{2}$, then $\mu_{n,\om}\geq 5$ for any $n\in\N^*$.
\subsection{Single mode case}\label{subsec_singlemode}

We start by proving Lemma \ref{lem_competitor} under the simplest boundary conditions:  $v=0$, and $u$ is a single buckling mode on a segment with eigenvalue close to $4$, meaning the first buckling mode (for angle close to $\pi$), or the second buckling mode (for angles close to $t_1$), or the third buckling mode (for angles close to $2\pi$). The proof in each case is exactly the same, so we prove all three at the same time.

\begin{lemma}\label{lem_lowermode_improvement}
Let $\Theta\in\{\pi,t_1,2\pi\}$. There exist $c,\eps>0$ such that for any $a\in\R^*$, $\om\in [\Theta-c,\Theta+c]\cap (0,2\pi)$, if one denotes 
$$u(\theta)=ab_{i(\Theta),\om}(\theta),$$
then there exists $U\in H^2_\lin(\Cy_0)$ such that $U(0,\theta)=u(\theta)$, $\partial_{t}U(0,\theta)=0$ and
\begin{equation}
\G(U,0)\leq \G(u,0)-\eps\left(\G(u,0)-\frac{\Theta}{2}\right)_+.
\end{equation}
\end{lemma}

\begin{proof}
We denote in this proof $$\Gn(U,0)=\G(U,0)-\frac{\Theta}{2}$$ 
so that our goal is to find some competitor $U$ such that $\Gn(U,0)\leq \Gn(u,0)-\eps\Gn(u,0)_+$. When $\Gn(u,0)\leq 0$, this is obvious by taking $U(t,\theta)=u(\theta)$ for every $t\geq 0$: we therefore  suppose in what follows that $\Gn(u,0)>0$. We let $\eps>0$ to be fixed sufficiently small later, and
\[\chi_\eps(t)=1-9\eps(1-(1+t)e^{-t}),\]
which  is positive in $\R_+$ (for small enough $\eps$), decreasing, with $\chi_\eps(0)=1$, $\chi_\eps'(0)=0$, $\chi_{\eps}(t)\underset{t\to+\infty}{\longrightarrow}1-9\eps$ and
\begin{equation}\label{eq_equilibrium}
\int_{0}^{+\infty}e^{-2t}\left(\chi_\eps(t)-(1-\eps)\right)dt=0 . 
\end{equation}
We define $\delta_0\in\R$ by $$\om=(1+\delta_0)\Theta,$$
and for any $\delta$ we write
$$\phi_{\delta}(\theta)=b_{i(\Theta),\Theta}\left(\frac{\theta}{1+\delta}\right)$$

so that $u(\theta)=a\sqrt{1+\delta_0}\phi_{\delta_{0}}(\theta)$. We define our competitor $U=U(t,\theta)$ by
\[U(t,\theta)=a\sqrt{1+\delta_0}\phi_{\delta_0\chi_\eps(t)}(\theta).\]
First observe that $U(0,\theta)=u(\theta)$, $\partial_t U(0,\theta)=0$ (since $\chi_\eps(0)=1$, $\chi_\eps'(0)=0$) and $U\in H^2_\lin(\Cy_0)$ (since $\chi_\eps\to 1-9\eps$, and $\chi_{\eps}'$, $\chi_{\eps}''\to 0$ at $+\infty$). Then
\begin{align*}
\Gn(U,0)=&\int_{0}^{+\infty}e^{-2t}\left\{\Vert \partial_{t,t}U\Vert^2+2\Vert\partial_{t,\theta}U\Vert^2\right\}dt\\
&+\int_{0}^{+\infty}e^{-2t}\left\{\Vert \partial_{\theta,\theta}U\Vert^2-4\Vert\partial_{\theta}U\Vert^2+\Theta\delta_0\chi_\eps\right\}dt,
\end{align*}
so
\begin{equation}\label{eq_proof_lowermode}
\begin{split}
\Gn(U,0)-(1-\eps)\Gn(u,0)=&a^2(1+\delta_0)\int_{0}^{+\infty}e^{-2t}\left\{\Vert \partial_{t,t}\phi_{\delta_0\chi_\eps}\Vert^2+2\Vert\partial_{t,\theta}\phi_{\delta_0\chi_\eps}\Vert^2\right\}dt\\
&+(1-a^2(1+\delta_0))\Theta\delta_0\int_{0}^{+\infty}e^{-2t}\left\{ \chi_\eps -(1-\eps)\right\}dt\\
&+a^2(1+\delta_0)\int_{0}^{+\infty}e^{-2t}\left\{\Vert \partial_{\theta,\theta}\phi_{\delta_0\chi_\eps}\Vert^2-4\Vert\partial_{\theta}\phi_{\delta_0\chi_\eps}\Vert^2+\Theta\delta_0\chi_\eps\right\}dt\\
&-a^2(1+\delta_0)(1-\eps)\int_{0}^{+\infty}e^{-2t}\left\{\Vert \partial_{\theta,\theta}\phi_{\delta_0}\Vert^2-4\Vert\partial_{\theta}\phi_{\delta_0}\Vert^2+\Theta\delta_0\right\}dt.
\end{split}
\end{equation}
Since $|\chi_\eps'|+|\chi_\eps''|\lesssim \eps$, the first line of \eqref{eq_proof_lowermode} is bounded by 
$$A_1(1+\delta_0)a^2\eps^2\delta_0^2$$ for some universal constant $A_1>0$. The second line is zero by our normalization condition \eqref{eq_equilibrium}. Now we estimate the sum of the third and fourth lines, which is equal to $a^2(1+\delta_0)F(\delta_0)$, where
\begin{align*}
F_\eps(\delta):&=\int_{0}^{+\infty}e^{-2t}\left(f(\delta\chi_\eps)-(1-\eps)f(\delta)\right)dt,\\
f(\delta):&=\Vert \partial_{\theta,\theta}\phi_{\delta}\Vert^2-4\Vert\partial_{\theta}\phi_{\delta}\Vert^2+\Theta\delta.
\end{align*}
We first compute $f(\delta)$ explicitly:
\begin{align*}
f(\delta)&=\Theta\delta+\int_{0}^{\Theta(1+\delta)}\left\{\frac{1}{(1+\delta)^4}\partial_{\theta,\theta}b_{i(\Theta),\Theta}\left(\frac{\theta}{1+\delta}\right)^2-\frac{4}{(1+\delta)^2}\partial_{\theta}b_{i(\Theta),\Theta}\left(\frac{\theta}{1+\delta}\right)^2\right\}d\theta\\
&=\Theta\delta+4\left(\frac{1}{(1+\delta)^3}-\frac{1}{1+\delta}\right).
\end{align*}
In particular $f(0)=0$, $f''(0)=40$, and there exists $C>0$ such that for any $\delta\in [-1/2,1/2]$, we have
\[|f(\delta)-f'(0)\delta-20\delta^2|\leq C|\delta|^3.\]
Then
\begin{align*}
F_\eps(0)&=\int_{0}^{+\infty}e^{-2t}\left(f(0)-(1-\eps)f(0)\right)dt=0\qquad\text{ since }f(0)=0\\
F_\eps'(0)&=\int_{0}^{+\infty}e^{-2t}\left(\chi_\eps(t)f'(0)-(1-\eps)f'(0)\right)dt=0\qquad\text{ by \eqref{eq_equilibrium}}\\
F_\eps''(0)&=\int_{0}^{+\infty}e^{-2t}\left(\chi_\eps(t)^2f''(0)-(1-\eps)f''(0)\right)dt\\
&=f''(0)\int_{0}^{+\infty}e^{-2t}\left((1-9(1-(1+t)e^{-t})\eps)\chi_\eps(t)-(1-\eps)\right)dt\\
&=40\int_{0}^{+\infty}e^{-2t}\left(-9(1-(1+t)e^{-t})\eps)\chi_\eps(t)\right)dt\text{ by \eqref{eq_equilibrium}}\\
&=-20\left(1-\frac{45}{16}\eps\right)\eps.
    \end{align*}
Moreover, for any $\delta\in [-1/2,1/2]$ we can estimate uniformly
\begin{align*}
\left|F_\eps(\delta)+10\left(1-\frac{45}{16}\eps\right)\eps\delta^2\right|=&\left|F_\eps(\delta)-F_\eps(0)-F_\eps'(0)\delta-\frac{1}{2}F_\eps''(0)\delta^2\right|\\
\leq& \int_{0}^{+\infty}e^{-2t}\left|f(\delta\chi_\eps)-f(0)-f'(0)\delta\chi_\eps-\frac{1}{2}f''(0)(\delta\chi_\eps)^2\right|dt\\
&+ (1-\eps)\int_{0}^{+\infty}e^{-2t}\left|f(\delta)-f(0)-f'(0)\delta-\frac{1}{2}f''(0)\delta^2\right|dt\\
\leq& A_2 |\delta|^3,
\end{align*}
for some $A_2>0$ that is independent of $\eps$: in particular, for $\eps\leq\frac{8}{45}$ we obtain
$$F_\eps(\delta)\leq -5\eps\delta^2+A_2|\delta|^3.$$
Going back to \eqref{eq_proof_lowermode}, we may now bound
\begin{align*}
\Gn(U,0)-(1-\eps)\Gn(u,0)=&a^2(1+\delta_0)\int_{0}^{+\infty}e^{-2t}\left\{\Vert \partial_{t,t}\phi_{\delta_0\chi_\eps}\Vert^2+2\Vert\partial_{t,\theta}\phi_{\delta_0\chi_\eps}\Vert^2\right\}dt\\
&+a^2(1+\delta_0)F(\delta_0)\\
\leq& (1+\delta_0)a^2\left(A_1\eps^2\delta_0^2-5\eps\delta_0^2+A_2|\delta_0|^3\right).
\end{align*}
Thus we may fix a small enough $\eps\leq\min\left(\frac{8}{45},\frac{1}{A_1}\right)$ such that for any $\delta_0$ verifying $|\delta_0|\leq \frac{\eps}{A_2}$, the right-hand side above is negative so
\begin{align*}
\Gn(U,0)\leq (1-\eps)\Gn(u,0).
\end{align*}
\end{proof}

In the case $\Theta=2\pi$, there is an additional possibility we have to consider here: the support of $u$ is the union of two disjoint segments of length close to $\pi$, and $u$ is a first buckling eigenmode on each connected component of the support.

\begin{lemma}\label{lem_lowermode_improvement_double}
There exist $c,\eps>0$ such that for any $a_1,a_2\in\R^*$, $\om_1,\om_2\in [\pi-c,\pi+c]$ with $\om_1+\om_2\leq 2\pi$, for any $\beta\in [0,2\pi-\om_1-\om_2]$,
if one denotes
$$u(\theta)=\begin{cases} a_1b_{1,\om_1}(\theta)& \text{ if }\theta\in [0,\om_1]\\ a_2b_{1,\om_2}(\theta-\om_1-\beta)& \text{ if }\theta\in [\om_1+\beta,\om_1+\beta+\om_2]\end{cases},$$
then there exists $U\in H^2_\lin(\Cy_0)$ such that $U(0,\theta)=u(\theta)$, $\partial_{t}U(0,\theta)=0$ and
\begin{equation}
\G(U,0)\leq \G(u,0)-\eps\left(\G(u,0)-\pi\right)_+
\end{equation}
\end{lemma}
\begin{proof}
Again the case $\G(u,0)\leq\pi$ is direct by choosing $U(t,\theta)=u(\theta)$, so we assume $\G(u,0)>\pi$.

Without loss of generality (up to switching the roles of the two sectors) we suppose that $\beta\geq \frac{2\pi-\om_1-\om_2}{2}$. Let $U_1,U_2$ be the competitors built in Lemma \ref{lem_lowermode_improvement} for $a_1b_{1,\om_1}$, $a_2b_{1,\om_2}$ respectively. We let $\eps$ be the constant from Lemma \ref{lem_lowermode_improvement}, so that
$\Spt(U_i(t,\cdot))=[0,\pi+(\om_i-\pi)\chi_{\eps}(t)]$
where we remind $\chi_\eps(t)=1-9\eps(1-(1+t)e^{-t})$ takes values between $1-9\eps$ and $1$. We define
$$
U(t,\theta)=U_1(t,\theta)+U_2(t,\om_1+\om_2+\beta-\theta).
$$
We claim the support of the two terms on the right-hand side are disjoint: indeed for any $t\geq 0$, we have
\begin{align*}
\Spt(U_1(t,\cdot))&=[0,\pi+(\om_1-\pi)\chi_{\eps}(t)],\\
\Spt(U_2(t,\om_1+\om_2+\beta-\cdot))&=[\om_1+\om_2+\beta-\pi-(\om_2-\pi)\chi_{\eps}(t),\om_1+\om_2+\beta]
\end{align*}
and the condition
$$\pi+(\om_1-\pi)\chi_{\eps}(t)\leq \om_1+\om_2+\beta-\pi-(\om_2-\pi)\chi_{\eps}(t)$$
reduces to
$$(2\pi-\om_1-\om_2)\chi_\eps(t)\geq 2\pi-\om_1-\om_2-\beta.$$
By our hypothesis on $\beta$, the right-hand side is smaller than $\frac{2\pi-\om_1-\om_2}{2}$, while the left-hand side is larger than $(1-9\eps)(2\pi-\om_1-\om_2)$: since $\eps$ may be chosen smaller than $\frac{1}{18}$, this condition is always verified. Thus, we have by additivity
\begin{align*}
\G(U,0)&=\G(U_1,0)+\G(U_2,0)\\
&\leq \G(a_1b_{1,\om_1},0)+\G(a_2b_{1,\om_2},0)-\eps\left(\G(a_1b_{1,\om_1},0)-\frac{\pi}{2}\right)_+-\eps\left(\G(a_2b_{1,\om_2},0)-\frac{\pi}{2}\right)_+\\
&\leq \G(u,0)-\eps\left(\G(u,0)-\pi\right)_+\text{ by the general inequality }x_++y_+\geq (x+y)_+,\ \forall x,y\in\R.
\end{align*}
This concludes the proof.
\end{proof}

\subsection{Removal of higher modes}\label{subsec_highermode}

We now prove Lemma \ref{lem_competitor} in the case where $v=0$, and $u$ is a general function such that $\Spt(u)$ verifies the hypothesis of Lemma \ref{lem_competitor}. We reduce the problem to the single mode case to obtain the estimate \eqref{eq_epiest}.

We first prove some useful computations. Let $u\in H^2_0([0,\om])$, written in the eigenbasis $(b_{n,\om})$ as
\[u(\theta)=\sum_{n\geq 1}u_n b_{n,\om}(\theta)\]
where $(u_n)_{n\geq 1}$ is uniquely defined, and $\sum_{n\geq 1}\mu_{n,\om}|u_n|^2=\Vert\partial_{\theta,\theta} u\Vert^2_{L^2([0,\om])}<\infty$. For any uniformly bounded sequence of functions $f_n\in\mathcal{C}^2(\R_+,\R)$  with $f_n(0)=1,f_n'(0)=0$, we let
\[u^{(f_n)_{n\geq 1}}(t,\theta)=\sum_{n\geq 1}u_nf_n(t)b_{n,\om}(\theta).\]
In particular, $u^{(f_n)}\in H^2_\lin(\Cy_0)$ verifies $u^{(f_n)_{n\geq 1}}(0,\theta)=u(\theta)$, $\partial_{t}|_{t=0}u^{f,N}(t,\theta)=0$.

\begin{lemma}\label{lemma_decomposition}
There exists some universal constant $C_0>0$ such that the following holds: for any $\om\in (0,2\pi]$, $u\in H^2_0([0,\om])$, $(f_n)_{n\geq 1}$ a sequence of uniformly bounded functions in $\mathcal{C}^2(\R_+,\R)$ such that $f_n(0)=1$, $f_n'(0)=0$, and $u^{(f_n)_{n\geq 1}}$ defined as above, we have
\begin{equation}\label{est_decomposition}
\begin{split}
\G(u^{(f_n)_{n\geq 1}},0)&\leq\om\int_{0}^{+\infty}e^{-2t}\chi_{\exists n:f_n(t)\neq 0}dt+\sum_{n\geq 1}Q_{\mu_{n,\om}-4}(f_n)u_n^2,
\end{split}
\end{equation}
where
\begin{equation}\label{eq_Q}
Q_\lambda(f):=\int_{0}^{+\infty}e^{-2t}\left\{C_0f''(t)^2+2f'(t)^2+\lambda f(t)^2\right\}dt.
\end{equation}
Moreover, when $f_n\equiv 1$ for every $n$ and $|\{u\neq 0\}|=\om$, then \eqref{est_decomposition} is an equality, and in this case we have
$$\G(u^{(1)_{n\geq 1}},0)=\G(u,0)=\frac{\om}{2}+\sum_{n\geq 1}\frac{\mu_{n,\om}-4}{2}u_n^2.$$
\end{lemma}
\begin{proof}
We compute directly
\begin{align*}
\G(u^{(f_n)_{n\geq 1}},0)=&\int_{0}^{+\infty}e^{-2t}\left(\int_{\Sp}\chi_{u^{(f_n)_{n\geq 1}}\neq 0}\right)dt+\int_{0}^{+\infty}e^{-2t}\left\{\left\Vert\sum_{n\geq 1}f_n''(t)u_n b_{n,\om}\right\Vert^2+2\left\Vert\sum_{n\geq 1}f_n'(t)u_n \partial_\theta b_{n,\om}\right\Vert^2\right\}dt\\
&+\int_{0}^{+\infty}e^{-2t}\left(\left\Vert\sum_{n\geq 1}f_n(t)u_n\partial_{\theta,\theta} b_{n,\om}\right\Vert^2-4\left\Vert\sum_{n\geq 1}f_n(t)u_n\partial_{\theta} b_{n,\om}\right\Vert^2\right)dt\\
\leq&\,\om\int_{0}^{+\infty}e^{-2t}\chi_{\exists n:f_n(t)\neq 0}dt+\sum_{n,m\geq 1}\int_{0}^{+\infty}e^{-2t}f_n''(t)f_m''(t)dt\langle b_{n,\om},b_{m,\om}\rangle u_n u_m\\
&+2\sum_{n\geq 1}\int_{0}^{+\infty}e^{-2t}f_n'(t)^2dt\ u_n^2+\sum_{n\geq 1}(\mu_{n,\om}-4)\int_{0}^{+\infty}e^{-2t}f_n(t)^2dt\ u_n^2
\end{align*}
Here we recall the orthogonality properties $\langle\partial_{\theta} b_{n,\om},\partial_{\theta} b_{m,\theta}\rangle =\delta_{n,m}$, $\langle\partial_{\theta,\theta} b_{n,\om},\partial_{\theta,\theta} b_{m,\theta}\rangle =\mu_{n,\om}\delta_{n,m}$, but $\langle b_{n,\om},b_{m,\om}\rangle$ is generally not zero; by Cauchy-Schwarz inequality, we have
\begin{align*}
\sum_{n,m\geq 1}\left(\int_{0}^{+\infty}e^{-2t}f_n''(t)f_m''(t)dt\right)\langle b_{n,\om},b_{m,\om}\rangle u_n u_m&\leq \left(\sum_{n\geq 1}\left(\int_{0}^{+\infty}e^{-2t}f_n''(t)^2dt\right)^{\frac{1}{2}}\Vert b_{n,\om}\Vert |u_n|\right)^2\\
&\leq \left(\sum_{n\geq 1}\Vert b_{n,\om}\Vert^2\right)\left(\sum_{n\geq 1}\int_{0}^{+\infty}e^{-2t}f_n''(t)^2dt\ u_n^2\right).
\end{align*}
Denote $C_0=\sup_{\om\in (0,2\pi]}\left(\sum_{n\geq 1}\Vert b_{n,\om}\Vert^2\right)$, $C_0$ is finite since by Lemma \ref{lem_estbn}, $\Vert b_{n,\om}\Vert^2\lesssim \mu_{n,\om}^{-1}=\frac{\om^2}{\pi^2}\mu_{n}^{-1}$ which is summable (we recall that $\mu_n^{-1}=\mathcal{O}\left(\frac{1}{n^2}\right)$). Then we identify

\begin{align*}
\G(u^{(f_n)_{n\geq 1}},0)\leq\om \int_{0}^{+\infty}e^{-2t}\chi_{\exists n:f_n(t)\neq 0}dt+\sum_{n\geq 1}Q_{\mu_{n,\om}-4}(f_n)u_n^2
\end{align*}
with an equality in the chain of inequality above when every $f_n$ is constant and the measure of the support of $u$ is $\om$. In this case $Q_\lambda(1)=\frac{\lambda}{2}$, which concludes the proof.
\end{proof}

The constant function $f\equiv 1$ is clearly not a minimizer of the integral $Q_\lambda(f)$ for $\lambda>0$, since it does not verify the associated Euler-Lagrange equation. It is however non-trivial whether there is a function with \textbf{compact support} and quantifiably lower energy $Q_\lambda$ for every $\lambda\geq 1$: this is shown in the following Lemma.
\begin{lemma}\label{lem_highermodes_deletion_positive}
There exist $\eta>0$, $f\in\mathcal{C}^2_c([0,+\infty[,[0,1])$ verifying $f(0)=1$, $f'(0)=0$, such that for any $\lambda\geq 1$ we have
\[Q_{\lambda}(f)\leq (1-\eta)Q_{\lambda}(1),\]
where $Q_\lambda$ is defined in \eqref{eq_Q}.
\end{lemma}

Note that the condition $\lambda\geq 1$ could be replaced by any condition $\lambda\geq \lambda_0$ for some $\lambda_0>0$.
\begin{proof}
It is sufficient to prove the result for $\lambda=1$, and it implies the result for any $\lambda\geq 1$. Indeed, suppose $f$ verifies the inequality for the value $\lambda=1$, since $f\leq 1$ and $f$ is not identically $1$, then for some constant $\eta'\in (0,1)$ we have
$$\int_{0}^{+\infty}e^{-2t}f(t)^2dt\leq \frac{1-\eta'}{2}$$
so for any $\lambda\geq 1$,
\begin{align*}
Q_\lambda(f)&=Q_1(f)+(\lambda-1)\int_{0}^{+\infty}e^{-2t}f(t)^2dt\\
&\leq (1-\eta)\frac{1}{2}+(1-\eta')\frac{\lambda-1}{2}\\
&\leq (1-\eta\wedge \eta')\frac{\lambda}{2}.
\end{align*}
Now we prove the result for $\lambda=1$. Let $\eps\in (0,1)$, $T>1$ to be fixed later, $\chi\in\mathcal{C}^\infty(\R,[0,1])$ equal to $1$ on $]-\infty,-1[$, $0$ on $[0,+\infty[$, and
\[f_\eps(t)=(1+\eps t)e^{-\eps t},\qquad f_{\eps,T}(t)=\chi(t-T)f_\eps(t).\]
We prove that for any sufficiently small $\eps$ we have the two estimates:
\begin{equation}\label{estimate_highermode}
(i)\ Q_1(f_\eps)\leq \left(1-\frac{1}{4}\eps^2\right)Q_1(1),\qquad  (ii)\ Q_1(f_{\eps,T})\leq \left(1+Ce^{-2T}\right)Q_1(f_\eps),
\end{equation}
for some universal constants $c,C>0$ (that are independant of $T,\eps$). The result is then obtained by fixing a sufficiently small $\eps$ such that the estimates hold, and then choosing $T$ large enough depending on $\eps$ such that $Q_1(f_{\eps,T})\leq \left(1-\frac{1}{8}\eps^2\right)Q_1(1)$.
\begin{itemize}
\item[(i)]We compute $f_\eps'(t)=-\eps^2 te^{-\eps t}$,  $f_\eps ''(t)=-\eps^2(1-\eps t)e^{-\eps t}$. So for some constant $C>0$ we may bound the first two terms of $Q_1(f_\eps)$ (see \eqref{eq_Q}) by $C\eps^4$.
Then the last term of \eqref{eq_Q} is estimated by
\begin{align*}
\int_{0}^{+\infty}&e^{-2t} f_\eps(t)^2dt=\int_{0}^{+\infty}e^{-2(1+\eps)t}(1+\eps t)^2dt= \frac{1}{2}\left(1-\frac{\eps^2}{2}+\Oo(\eps^3)\right).
\end{align*}
This proves (i) for sufficiently small $\eps$ (we remind that $Q_1(1)=\frac{1}{2}$).
\item[(ii)]We compute directly, using the (rough) estimate $f_\eps,f_\eps',f_\eps''\lesssim 1$ and $Q(f_\eps)\gtrsim 1$ (for any $\eps\in (0,1)$):
\begin{align*}
Q_1(f_{\eps,T})-Q_1(f_\eps)&\leq\int_{T-1}^{T}e^{-2t}\left(C_0(f_{\eps,T}'')^2+2(f_{\eps,T}')^2+(f_{\eps,T})^2\right)dt\\
&\leq Ce^{-2T}Q_1(f_\eps),
\end{align*}
for some constant $C>0$. So we get the estimate (ii).
\end{itemize}
This concludes the proof.
\end{proof}

We now prove an analogue of Lemma \ref{lem_highermodes_deletion_positive} for \textbf{negative} values of $\lambda$.

\begin{lemma}\label{lem_highermodes_deletion_negative}
There exists $f\in\mathcal{C}^2_c([0,+\infty[,\R)$ verifying $f(0)=1$, $f'(0)=0$, such that for any $\lambda\leq -1$ we have
\[Q_{\lambda}(f)\leq Q_{\lambda}(1),\]
where $Q_\lambda$ is defined in \eqref{eq_Q}.
\end{lemma}
Note that in this case $Q_\lambda(1)=\frac{\lambda}{2}< 0$ since $\lambda<0$, in particular we have $Q_\lambda(1)\leq (1-\eta)Q_\lambda(1)$ for any $\eta>0$.

\begin{proof}
Define, as in the previous proof, $\chi\in\mathcal{C}^\infty(\R,[0,1])$ equal to $1$ on $]-\infty,-1]$, $0$ on $[0,+\infty[$. We let
$$f_\eps(t)=\eps t+e^{-\eps t},$$
for some $\eps\in (0,1)$ to be fixed. We have $f_\eps(0)=1$, $f_\eps'(0)=0$, and moreover $f_\eps'(t)=\eps\left(1-e^{-\eps t}\right)$, $f_\eps''(t)=\eps^2e^{-\eps t}$. Thus
\begin{align*}
\int_{0}^{+\infty}e^{-2t}f_\eps''(t)^2dt&\lesssim \eps^4,\qquad \int_{0}^{+\infty}e^{-2t}f_\eps'(t)^2dt\lesssim \eps^3\\
\int_{0}^{+\infty}e^{-2t}f_\eps(t)^2dt&=\frac{1}{2}+\frac{\eps^2}{4}+\mathcal{O}\left(\eps^3\right)\geq \frac{1}{2}\left(1+\frac{\eps^2}{3}\right)\text{ for small }\eps.
\end{align*}
We now fix $\eps\in (0,1)$ small enough such that the last inequality holds. Let
\[f_{\eps,T}(t)=\chi(t-T)f_\eps(t)\]
for some $T>1$ to be fixed. With the rough estimate $|f_\eps(t)|+|f_\eps'(t)|+|f_\eps''(t)|\lesssim 1+t$ (for small $\eps$), we have for any $\lambda\leq -1$:
\begin{align*}
Q_\lambda(f_{\eps,T})&=Q_\lambda(f_\eps)+(Q_\lambda(f_{\eps,T})-Q_\lambda(f_\eps))\\
&\leq \frac{\lambda}{2}\left(1+\frac{\eps^2}{3}\right)+\int_{T-1}^{+\infty}e^{-2t}\left(C_0f_{\eps,T}^{''2}+2f_{\eps,T}^{'2}-\lambda f_{\eps}^2\right)dt\\
&\leq \frac{\lambda}{2}\left(1+\frac{\eps^2}{3}-CT^2e^{-2T}\right)
\end{align*}
for some universal constant $C>0$. Thus for a large enough $T$ we have $Q_\lambda(f_{\eps,T})\leq \frac{\lambda}{2}=Q_\lambda(1)$.
\end{proof}

We now use Lemma \ref{lem_highermodes_deletion_positive} and \ref{lem_highermodes_deletion_negative}, to prove the Lemma \ref{lem_competitor} in the case where $v=0$. We prove in fact a slightly stronger estimate, with an additional control on $\partial_\theta u$ that will be useful in the next section.

We first introduce some notations: consider $u\in H^2(\Sp)$ such that $\Spt(u)\neq \Sp$. Its support is a countable union of disjoint intervals $(\theta_k,\theta_k+\om_k)_{k\in\N}$ of respective lengths $$\om_0\geq\om_1\geq\om_2\geq\hdots\to 0.$$
We let $u^k(\theta)=u(\theta_k+\theta)\chi_{0<\theta<\om_k}$ such that $u^k\in H^2_0([0,\om_k])$ is decomposed as
\begin{equation}\label{eq_fulldecompositionu}
u^k=\sum_{n\geq 1}u_n^k b_{n,\om_k}
\end{equation}
for some sequence of coefficients $(u_n^k)_{n\geq 1,k\geq 0}$.

\begin{lemma}\label{lem_highermode_improvement}
Let $\Theta\in\{\pi,t_1,2\pi\}$. There exist $c,\eps,a>0$ such that for any $u\in H^2(\Sp)$ that verifies
$$\Spt(u)\neq \Sp,\ [c,\Theta-c]\subset\Spt(u)\subset[-c,\Theta+c],$$
there exists $U\in H^2_\lin(\Cy_0)$ such that $U(0,\theta)=u(\theta)$, $\partial_{t}U(0,\theta)=0$ and
\begin{equation}
\G(U,0)\leq \G(u,0)-\eps\left(\G(u,0)-\frac{\Theta}{2}\right)_+-a\left(\left\Vert\partial_{\theta,\theta} \left(u^0-\sum_{n=1}^{i(\Theta)}u^0_n b_{n,\om_0}\right)\right\Vert^2+\sum_{k\geq 1}\Vert\partial_{\theta,\theta} u^k\Vert^2\right)
\end{equation}
where $u^k$ is defined from the decomposition of $u$ in \eqref{eq_fulldecompositionu}.
\end{lemma}

\begin{proof}
Let $\om=|\Spt(u)|$. Let $(\theta_k,\om_k,u^k,u^k_n)$ be defined from the decomposition of $u$ as in \eqref{eq_fulldecompositionu}. By our hypothesis we have $|\om_0-\Theta|\leq 2c$ and $|\om_1|\leq 2c$. We suppose $c$ is sufficiently small so that $|\om_0-\Theta|<0.1\Theta$ and $\om_1\leq \frac{\pi}{2}$ (see equation \eqref{eq_gapestimate})

Let now $g$ be the function constructed from Lemma \ref{lem_highermodes_deletion_positive}, and $h$ the function constructed from Lemma \ref{lem_highermodes_deletion_negative}. Denote $T>0$ the smallest value such that $g=h=0$ on $[T,+\infty[$.
We now define, for any $n\geq 1$, $k\geq 0$, 
$$f_{n}^{k}(t)=
\begin{cases}
g(t)&\text{ if }\mu_{n,\om_k}\geq 5\text{ i.e. }k\geq 1\text{ or }n>i(\Theta),\\
h(t)&\text{ if }\mu_{n,\om_k}\leq 3\text{ i.e. }k=0 \text{ and }n<i(\Theta),\\
1&\text{ if }|\mu_{n,\om_k}-4|<1\text{ i.e. }k=0\text{ and }n=i(\Theta).
\end{cases}
$$
We define the competitor
$$\tilde{U}(t,\theta)=\sum_{k\geq 0}\sum_{n\geq 1}f_{n}^{k}(t)u_n^k b_{n,\om_k}(\theta-\theta_k).$$
By the definition of $(f_{n}^{k})$, we have for any $t\geq T$: $$\tilde{U}(t,\theta)=u_{i(\Theta)}^0 b_{i(\Theta),\om_0}(\theta-\theta_0).$$
Denote $\eta>0$ the constant from Lemma \ref{lem_highermodes_deletion_positive}, $\eps>0$ the constant from Lemma \ref{lem_lowermode_improvement}, and $\widehat{U}(t,\theta)$ the competitor built in Lemma \ref{lem_lowermode_improvement} from the initial data $u_{i(\Theta)}^0 b_{i(\Theta),\om_0}$, so that
$$\G(\widehat{U},0)\leq \G(u_{i(\Theta)}^0 b_{i(\Theta),\om_0},0)-\eps\left(\G(u_{i(\Theta)}^0 b_{i(\Theta),\om_0},0)-\frac{\Theta}{2}\right)_+.$$
Let then
$$U(t,\theta)=\begin{cases}\tilde{U}(t,\theta)&\text{ if }t< T\\ \widehat{U}(t-T,\theta+\theta_0)&\text{ if }t\geq T\end{cases}$$
For any $k\geq 1$, since every function $(f_n^k)_{n\geq 1}$ vanishes on $[T,+\infty[$, then $$\int_{0}^{+\infty}e^{-2t}\chi_{\exists n:f_{n}^k\neq 0}dt\leq \frac{1-e^{-2T}}{2}.$$
We compute:
\begin{align*}
\G(U,0)= &\  \G(\tilde{U},0)+\left(\G(U,0)-\G(\tilde{U},0)\right)\\
\leq &\  \G(u_{i(\Theta)}^0b_{i(\Theta),\om_0},0)+(1-\eta)\sum_{n>i(\Theta)}Q_{\mu_{n,\om_0}-4}(1)|u_n^0|^2+\sum_{n<i(\Theta)}Q_{\mu_{n,\om_0}-4}(1)|u_n^0|^2\\
&+\sum_{k\geq 1}\left(\frac{1-e^{-2T}}{2}\om_k+(1-\eta)\sum_{n\geq 1}Q_{\mu_{n,\om_k}-4}(1)|u_n^k|^2\right)\\
&+e^{-2T}\left(\G(\widehat{U},0)-\G(u_{i(\Theta)}^0b_{i(\Theta),\om_0},0)\right)\text{ by Lemmas }\ref{lemma_decomposition},\ref{lem_highermodes_deletion_positive},\ref{lem_highermodes_deletion_negative}
\\
\leq &\  \G(u_{i(\Theta)}^0b_{i(\Theta),\om_0},0)-e^{-2T}\eps\left(\G(u_{i(\Theta)}^0 b_{i(\Theta),\om_0},0)-\frac{\Theta}{2}\right)_+\\
&+(1-\eta)\sum_{n>i(\Theta)}Q_{\mu_{n,\om_0}-4}(1)|u_n^0|^2+\sum_{n<i(\Theta)}Q_{\mu_{n,\om_0}-4}(1)|u_n^0|^2\\
&+\sum_{k\geq 1}\left(\frac{1-e^{-2T}}{2}\om_k+(1-\eta)\sum_{n\geq 1}Q_{\mu_{n,\om_k}-4}(1)|u_n^k|^2\right)\text{ by Lemma }\ref{lem_lowermode_improvement}.
\end{align*}
Denote $\ov{\eps}=\min\left(e^{-2T}\eps,e^{-2T},\frac{\eta}{2}\right)$. We remind that when $n<i(\Theta)$, $Q_{\mu_{n,\om_0}-4}(1)<0$ so $Q_{\mu_{n,\om_0}-4}(1)<(1-\ov{\eps})Q_{\mu_{n,\om_0}-4}(1)$. We remind also that for any $x,y\in\R$, we have $x_++y_+\geq (x+y)_+$, which we will use repeatedly to gather each term. We identify
$$\G(u,0)=\G(u_{i(\Theta)}^0b_{i(\Theta),\om_0},0)+\sum_{n\neq i(\Theta)}Q_{\mu_{n,\om_0}-4}(1)|u_n^0|^2+\sum_{k\geq 1}\left(\frac{1}{2}\om_k+\sum_{n\geq 1}Q_{\mu_{n,\om_k}-4}(1)|u_n^k|^2\right).$$
Then following the estimates above we obtain
\begin{align*}
\G(U,0)\leq&\  \G(u_{i(\Theta)}^0b_{i(\Theta),\om_0},0)-\ov{\eps}\left(\G(u_{i(\Theta)}^0 b_{i(\Theta),\om_0},0)-\frac{\Theta}{2}\right)_+\\
&+\left(1-\ov{\eps}-\frac{\eta}{2}\right)\sum_{n> i(\Theta)}Q_{\mu_{n,\om_0}-4}(1)|u_n^0|^2+\left(1-\ov{\eps}\right)\sum_{n< i(\Theta)}Q_{\mu_{n,\om_0}-4}(1)|u_n^0|^2\\
&+\sum_{k\geq 1}\left(\frac{1-\ov{\eps}}{2}\om_k+\left(1-\ov{\eps}-\frac{\eta}{2}\right)\sum_{n\geq 1}Q_{\mu_{n,\om_k}-4}(1)|u_n^k|^2\right)\\
&\leq \G(u,0)-\ov{\eps}\left(\G(u,0)-\frac{\Theta}{2}\right)_+-\frac{\eta}{2}\sum_{(n,k):n>i(\Theta)\text{ or }k\geq 1}Q_{\mu_{n,\om_k}-4}(1)|u_n^k|^2.
\end{align*}
We remind for the last line that if $(n,k)$ is such that $n>i(\Theta)$ or $k\geq 1$, then $\mu_{n,\om_k}\geq 5$ so $Q_{\mu_{n,\om_k}-4}(1)=\frac{\mu_{n,\om_k}-4}{2}\geq \frac{1}{10}\mu_{n,\om_k}$. We now conclude the proof with $a=\frac{\eta}{20}$, since
\begin{align*}
\sum_{(n,k):n>i(\Theta)\text{ or }k\geq 1}\mu_{n,\om_k}|u_n^k|^2&=\sum_{n>i(\Theta)}\mu_{n,\om_0}|u_n^0|^2+\sum_{n\geq 1}\sum_{k\geq 1}\mu_{n,\om_k}|u_n^k|^2\\
&= \left\Vert\partial_{\theta,\theta} \left(u^0-\sum_{n=1}^{i(\Theta)}u^0_n b_{n,\om_0}\right)\right\Vert^2+\sum_{k\geq 1}\Vert\partial_{\theta,\theta} u^k\Vert^2.
\end{align*}
\end{proof}

Next, we treat separately the case $\Theta=2\pi$ where the support of $u$ is the union of two disjoint segments of length close to $\pi$. The proof follows the same computations, we have separated the two for clarity.

\begin{lemma}\label{lem_highermode_improvement_double}
There exist $c,\eps,a>0$ such that for any $u\in H^2(\Sp)$ such that
$$[c,\pi-c]\cup[\pi+c,2\pi-c]\subset\Spt(u),\  [\pi-c,\pi+c] \nsubseteq \Spt(u),\ [2\pi-c,2\pi+c] \nsubseteq \Spt(u),$$
there exists $U\in H^2_\lin(\Cy_0)$ such that $U(0,\theta)=u(\theta)$, $\partial_{t}U(0,\theta)=0$ and
\begin{equation}
\G(U,0)\leq \G(u,0)-\eps\left(\G(u,0)-\pi\right)_+-a\left(\left\Vert\partial_{\theta,\theta} \left(u^0-u^0_1 b_{1,\om_0}\right)\right\Vert^2+\left\Vert\partial_{\theta,\theta} \left(u^1-u^1_1 b_{1,\om_1}\right)\right\Vert^2+\sum_{k\geq 2}\Vert\partial_{\theta,\theta} u^k\Vert^2\right).
\end{equation}
where $u^k$ is defined from the decomposition of $u$ in \eqref{eq_fulldecompositionu}.
\end{lemma}
\begin{proof}
Let $\om=|\Spt(u)|$. Let $(\theta_k,\om_k,u^k,u^k_n)$ be defined from the decomposition of $u$ (see \eqref{eq_fulldecompositionu}). By our hypothesis we have $|\om_0-\pi|\leq 2c$, $|\om_1-\pi|\leq 2c$, and $|\om_2|\leq 2c$: we suppose that $c$ is sufficiently small so that $|\om_1-\pi|,|\om_2-\pi|\leq 0.1\pi$, and $\om_3\leq \frac{1}{2}\pi$.

As previously, let $g$ be the function constructed from Lemma \ref{lem_highermodes_deletion_positive}. Denote $T>0$ the smallest value such that $g=0$ on $[T,+\infty[$. Let
$$f_{n}^{k}(t)=
\begin{cases}
g(t)&\text{ if }\mu_{n,\om_k}\geq 5\text{ i.e. }k\geq 2\text{ or }n\geq 2,\\
1&\text{ if }|\mu_{n,\om_k}-4|<1\text{ i.e. }k\in\{0,1\}\text{ and }n=1.
\end{cases}
$$
The last case $\mu_{n,\om_k}\leq 3$ never happens in this context. We define the competitor
$$\tilde{U}(t,\theta)=\sum_{k\geq 1}\sum_{n\geq 1}f_{n}^{k}(t)u_n^k b_{n,\om_k}(\theta-\theta_k).$$
We remind that from the estimates \eqref{eq_gapestimate}, we have exactly $$\tilde{U}(t,\theta)=u_{1}^0 b_{1,\om_0}(\theta-\theta_0)+u_{1}^1 b_{1,\om_1}(\theta-\theta_1)$$
for all $t\geq T$.
Denote $\widehat{U}(t,\theta)$ the competitor built in Lemma \ref{lem_lowermode_improvement_double} from the initial data $$\tilde{U}(T,\cdot+\theta_0)=u^{\mathrm{double}}:=u_{1}^0 b_{1,\om_0}+u_{1}^1 b_{1,\om_1}(\cdot+\theta_0-\theta_1),$$
so that
$$\G(\widehat{U},0)\leq \G(u^{\mathrm{double}},0)-\eps\left(\G(u^{\mathrm{double}},0)-\pi\right)_+.$$
Let then
$$U(t,\theta)=\begin{cases}\tilde{U}(t,\theta)&\text{ if }t< T,\\ \widehat{U}(t-T,\theta+\theta_0)&\text{ if }t\geq T.\end{cases}$$
Following the computations from the previous lemma (denoting $\eta>0$ the constant from Lemma \ref{lem_highermodes_deletion_positive}, $\eps>0$ the constant from Lemma \ref{lem_lowermode_improvement_double}), we get
\begin{align*}
\G(U,0)=&\,\G(\tilde{U},0)+\left(\G(U,0)-\G(\tilde{U},0)\right)\\
\leq&\, \G(u^{\mathrm{double}},0)+(1-\eta)\sum_{k=0}^{1}\sum_{n>1}Q_{\mu_{n,\om_k}-4}(1)|u_n^k|^2\\
&+\sum_{k\geq 2}\left(\frac{1-e^{-2T}}{2}\om_k+(1-\eta)\sum_{n\geq 1}Q_{\mu_{n,\om_k}-4}(1)|u_n^k|^2\right)\\
&+e^{-2T}\left(\G(\widehat{U},0)-\G(u^{\mathrm{double}},0)\right)\quad\text{ by Lemmas }\ref{lemma_decomposition},\ref{lem_highermodes_deletion_positive}.
\end{align*}
Then, by lemma \ref{lem_lowermode_improvement}, we have
\begin{align*}
\G(U,0)\leq&\, \G(u^{\mathrm{double}},0)-e^{-2T}\eps\left(\G(u^{\mathrm{double}},0)-\pi\right)_+\\
&+(1-\eta)\sum_{k=0}^{1}\sum_{n>1}Q_{\mu_{n,\om_k}-4}(1)|u_n^k|^2\\
&+\sum_{k\geq 2}\left(\frac{1-e^{-2T}}{2}\om_k+(1-\eta)\sum_{n\geq 1}Q_{\mu_{n,\om_k}-4}(1)|u_n^k|^2\right).
\end{align*}
Denote $\ov{\eps}=\min\left(e^{-2T}\eps,e^{-2T},\frac{\eta}{2}\right)$, and identify
$$
\G(u,0)=\G(u^{\mathrm{double}},0)+\sum_{k=0}^{1}\sum_{n>1}Q_{\mu_{n,\om_k}-4}(1)|u_n^k|^2+\sum_{k\geq 2}\left(\frac{1}{2}\om_k+\sum_{n\geq 1}Q_{\mu_{n,\om_k}-4}(1)|u_n^k|^2\right).
$$
Then
\begin{align*}
\G(U,0)&\leq \G(u,0)-\ov{\eps}\left(\G(u,0)-\pi\right)_+-\frac{\eta}{2}\sum_{(n,k):n>1\text{ or }k\geq 2}Q_{\mu_{n,\om_k}-4}(1)|u_n^k|^2\\
&\leq\G(u,0)-\ov{\eps}\left(\G(u,0)-\pi\right)_+-\frac{\eta}{20}\sum_{(n,k):n>1\text{ or }k\geq 2}\mu_{n,\om_k}|u_n^k|^2.
\end{align*}
This ends the proof with $a=\frac{\eta}{20}$, since
\begin{align*}
\sum_{(n,k):n>1\text{ or }k\geq 2}\mu_{n,\om_k}|u_n^k|^2&=\sum_{k=0}^{1}\sum_{n\geq 2}\mu_{n,\om_k}|u_n^k|^2+\sum_{k\geq 2}\sum_{n\geq 1}\mu_{n,\om_k}|u_n^k|^2\\
&= \left\Vert\partial_{\theta,\theta} \left(u^0-u^0_1 b_{1,\om_0}\right)\right\Vert^2+\left\Vert\partial_{\theta,\theta} \left(u^1-u^1_1 b_{1,\om_1}\right)\right\Vert^2+\sum_{k\geq 2}\Vert\partial_{\theta,\theta} u^k\Vert^2.
\end{align*}

\end{proof}
\subsection{General case $v\neq 0$}\label{subsec_partialt}
In this subsection, we give the proof of Lemma \ref{lem_competitor}, by handling general boundary data $(u,v)$ where $v$ is not necessarily zero. We prove it by reducing the analysis to the case where $v=0$. 
We first prove a purely technical lemma used below in the proof of Lemma \ref{lemma_removal_derivative}.

\begin{lemma}\label{lem_defchiom}
For any $\om\in [0.5\pi,2\pi]$, there exists $\chi_\om\in \mathcal{C}^\infty(\R_+,\R)$ such that $\chi_\om(t)=1$ in a neighbourhood of $0$, $\chi_\om(t)=0$ in $[1,+\infty[$, $\sup_{\om\in [0.5\pi,2\pi]}\Vert\chi_\om\Vert_{\mathcal{C}^2}\lesssim 1$ and
$$\forall \om\in [0.5\pi,2\pi],\ \forall n\in\{1,2,3\}, \ \int_{0}^{+\infty}te^{-(2+\mu_{n,\om}^{1/3})t}\chi_\om(t)dt=0.$$
\end{lemma}
\begin{proof}
Let $g_{n,\om}(t)=te^{-(2+\mu_{n,\om}^{1/3})t}$. Let $t_i=\frac{i}{4}$ for $i=1,2,3$. For any $\om\in [0.5\pi,2\pi]$, we have
$$\det\left((g_{n,\om}(t_m))_{1\leq n,m\leq 3}\right)\neq 0,$$
since this is (up to a nonzero scalar factor) the determinant of the non-trivial Vandermonde matrix
$$\left(e^{-\frac{2+\mu_{n,\om}^{1/3}}{4}k}\right)_{1\leq k,n\leq 3}.$$

Functions $g_{n,\om}$ depend continuously on $\om$, so we may find some approximation of unity $\sigma_1,\sigma_2,\sigma_3\in\mathcal{C}^\infty_c((0,1))$ around $t_1,t_2,t_3$ respectively such that $$\inf_{\om\in[0.5\pi,2\pi]}\left|\det\left(\left(\int_{0}^{1}g_{n,\om}(t)\sigma_{m}(t)dt\right)_{1\leq n,m\leq 3}\right)\right|>0$$
Let now $h\in\mathcal{C}^{\infty}(\R_+,\R)$ that is equal to $1$ in a neighbourhood of $0$, and to $0$ on $[1,+\infty[$, then there exists continuously defined $a_1(\om),a_2(\om),a_3(\om)\in\R$ such that for any $\om\in [0.5\pi,2\pi]$, $n\in\{1,2,3\}$, we have
$$\int_{0}^{1}\left(a_1(\om)\sigma_1+a_2(\om)\sigma_2+a_3(\om)\sigma_3\right)g_{n,\om}=\int_{0}^{1}hg_{n,\om}.$$
The function $\chi_\om(t)=h(t)-a_1(\om)\sigma_{1}(t)-a_2(\om)\sigma_2(t)-a_3(\om)\sigma_3(t)$  has all the required properties.
\end{proof}

\begin{lemma}\label{lemma_removal_derivative}
Let $\om\in (0,2\pi]$, $u\in H^2(\Sp)$, $v\in H^1(\Sp)$, with
\[\{v\neq 0\}\subset\Spt(u)=(0,\om),\] then there exists $U\in H^2_\lin(\Cy_0)$ such that $U(0,\cdot)=u$, $\partial_t U(0,\cdot)=v$, $U(t,\cdot)=u$ for all $t\geq 1$ and
\[\G(U,0)\leq \G(u,0)+C\left(\left\Vert \partial_{\theta,\theta} u\right\Vert\left\Vert \partial_{\theta} v\right\Vert+\left\Vert \partial_{\theta} v\right\Vert^2\right).\]
Moreover, if $|\om-\Theta|\leq 0.1\Theta$ for some $\Theta\in\{\pi,t_1,2\pi\}$, then we have the improved bound
\[\G(U,0)\leq \G(u,0)+C\left(\left\Vert  \partial_{\theta,\theta}\left(u-\sum_{n=1}^{i(\Theta)}u_nb_{n,\om}\right)\right\Vert \left\Vert \partial_{\theta}v\right\Vert+ \Vert \partial_{\theta}v\Vert^2\right)\]
where $(u_n)_{n\geq 1}$ are the coefficients in the decomposition $u=\sum_{n\geq 1}u_nb_{n,\om}$.
\end{lemma}

\begin{proof}
For any $n\in\N^*$, we define
$$ f_{n,\om}(t)=te^{-\mu_{n,\om}^{1/3}t}.$$
Note that $f_{n,\om}(0)=0$ and $f_{n,\om}'(0)=1$. Let now $\chi_\om$ be defined from Lemma \ref{lem_defchiom}, such that
\[\int_0^{+\infty}e^{-2t}f_{n,\om}(t)\chi_\om(t)dt=0\]
for $n=1,2,3$ when $\om\geq 0.5\pi$. When $\om<0.5\pi$ we just set $\chi_\om:=\chi_{0.5\pi}$.
We define
\[U(t,\theta)=\sum_{n\geq 1}\left(u_n+\chi_\om(t)f_{n,\om}(t)v_n\right)b_{n,\om}(\theta).\]
Then $U$ verifies $U\in H^2_\lin(\Cy_0)$, $U(0,\cdot)=u$, $\partial_t U(0,\cdot)=v$, $U(t,\cdot)=u$ for $t\geq 1$, and
\begin{align*}
\G(U,0)\leq&\, \frac{\om}{2}+\int_{0}^{+\infty}e^{-2t}\left\{\left\Vert\sum_{n\geq 1}(\chi_\om f_{n,\om})''(t)v_n b_{n,\om}\right\Vert^2+2\left\Vert\sum_{n\geq 1}(\chi_\om f_{n,\om})'(t) v_n \partial_{\theta}b_{n,\om}\right\Vert^2\right\}dt\\
&+\int_{0}^{+\infty}e^{-2t}\left\{\left\Vert\sum_{n\geq 1}(u_n+(\chi_\om f_{n,\om})(t)v_n)\partial_{\theta,\theta} b_{n,\om}\right\Vert^2-4\left\Vert\sum_{n\geq 1}(u_n+(\chi_\om f_{n,\om})(t)v_n)\partial_{\theta} b_{n,\om}\right\Vert^2\right\}dt\\
=&\,\frac{1}{2}\left(\om+\Vert \partial_{\theta,\theta}u\Vert^2-4\Vert \partial_{\theta}u\Vert^2\right)\\
&+\sum_{n\geq 1}(\mu_{n,\om}-4)\left(\int_{0}^{+\infty}2e^{-2t}\chi_\om(t)f_{n,\om}(t)dt\right)u_n v_n\\
&+\sum_{n\geq 1}\left(\int_{0}^{+\infty}e^{-2t}\left\{2|(\chi_\om f_{n,\om})'(t)|^2+(\mu_{n,\om}-4)|\chi_\om f_{n,\om}(t)|^2\right\}dt\right)v_n^2\\
&+\sum_{n,m\geq 1}\left(\int_{0}^{+\infty}e^{-2t}(\chi_\om f_{n,\om})''(t)(\chi_\om f_{m,\om})''(t)dt\right)\langle b_{n,\om},b_{m,\om}\rangle v_n v_m.
\end{align*}
We handle each of these four lines separately.
\begin{itemize}[label=\textbullet]
\item  The first line is exactly $\G(u,0)$.
\item When $\om\geq 0.5\pi$ and $n=1,2,3$, this term vanishes by the orthogonality condition we imposed on $\chi_\om$ in Lemma \ref{lem_defchiom}. In all other cases we have
\begin{align*}
\left|(\mu_{n,\om}-4)\int_{0}^{+\infty}2e^{-2t}\chi_\om(t)f_{n,\om}(t)dt\right|&\lesssim \mu_{n,\om}\int_{0}^{+\infty}e^{-(2+\mu_{n,\om}^{1/3})t}tdt\lesssim \mu_{n,\om}^{1/3}.
\end{align*}
\item Using the uniform bound on $\Vert\chi_\om\Vert_{\mathcal{C}^2}$, we have
\begin{align*}
\int_{0}^{+\infty}e^{-2t}\left\{2|(\chi_\om f_{n,\om})'(t)|^2+(\mu_{n,\om}-4)|\chi_\om f_{n,\om}(t)|^2\right\}dt&\lesssim \int_{0}^{+\infty}e^{-2t}\left\{f_{n,\om}'(t)^2+\mu_{n,\om}f_{n,\om}(t)^2\right\}dt\\
&\lesssim \frac{1}{\mu_{n,\om}^{1/3}}+1\lesssim 1
\end{align*}
by the definition of $f_{n,\om}$.
\item Similarly, we have
\begin{align*}
\int_{0}^{+\infty}e^{-2t}\left|(\chi_\om f_{n,\om})''(t)\right|^2dt&\lesssim\int_{0}^{+\infty}e^{-2t}\left\{f_{n,\om}''(t)^2+f_{n,\om}'(t)^2+f_{n,\om}(t)^2\right\}dt\\
&\lesssim \mu_{n,\om}^{1/3}+\mu_{n,\om}^{-1/3}+\mu_{n,\om}^{-1}\lesssim \mu_{n,\om}^{1/3}
\end{align*}
so by Cauchy-Schwarz inequality and using the bound $\Vert b_{n,\om}\Vert\lesssim \mu_{n,\om}^{-1/2}$ from Lemma \ref{lem_estbn}, we have
\begin{align*}
\left(\int_{0}^{+\infty}e^{-2t}(\chi_\om f_{n,\om})''(t)(\chi_\om f_{m,\om})''(t)dt\right)\langle b_{n,\om},b_{m,\om}\rangle\lesssim \frac{1}{\mu_{n,\om}^{1/3}\mu_{m,\om}^{1/3}}.
\end{align*}
\end{itemize}
We gather the three estimates. First for a general $\om\in (0,2\pi]$:
\begin{align*}
\G(U,0)-\G(u,0)&\lesssim \sum_{n\geq 1}\mu_{n,\om}^{1/3}|u_n v_n|+\sum_{n\geq 1}v_n^2+\sum_{n,m\geq 1}\frac{|v_n v_m|}{\mu_{n,\om}^{1/3}\mu_{m,\om}^{1/3}}\\
&\lesssim \sum_{n\geq 1}\mu_{n,\om}^{1/2}|u_n v_n|+\sum_{n\geq 1}v_n^2+\left(\sum_{n\geq 1}\frac{|v_n|}{\mu_{n,\om}^{1/3}}\right)^2\\
&\lesssim  \left\Vert \partial_{\theta,\theta} u\right\Vert\Vert \partial_{\theta} v\Vert+\left(1+\sum_{n\geq 1}\frac{1}{\mu_{n,\om}^{2/3}}\right)\sum_{n\geq 1}|v_n|^2\\
&\lesssim  \left\Vert \partial_{\theta,\theta} u\right\Vert\left\Vert \partial_{\theta} v\right\Vert+\left\Vert \partial_{\theta}v\right\Vert^2.
\end{align*}
We used the fact that $\mu_{n,\om}\gtrsim \frac{n^2}{\om^2}$, so $\mu_{n,\om}^{-2/3}$ is summable with uniformly bounded sum in $\om$.

Then, when $|\om-\Theta|\leq 0.1\Theta$ for some $\Theta\in\{\pi,t_1,2\pi\}$, in particular $\om\geq 0.5\pi$ so using Lemma \ref{lem_defchiom} we have the stronger estimate:
\begin{align*}
\G(U,0)-\G(u,0)&\lesssim \sum_{n>i(\Theta)}\mu_{n,\om}^{1/3}|u_n v_n|+\sum_{n\geq 1}v_n^2+\sum_{n,m\geq 1}\frac{|v_n v_m|}{\mu_{n,\om}^{1/3}\mu_{m,\om}^{1/3}}\\
&\lesssim \sum_{n>i(\Theta)}\mu_{n,\om}^{1/2}|u_n v_n|+\sum_{n\geq 1}v_n^2+\left(\sum_{n\geq 1}\frac{|v_n|}{\mu_{n,\om}^{1/3}}\right)^2\\
&\lesssim  \left\Vert \partial_{\theta,\theta}\left(u-\sum_{n\leq i(\Theta)}u_nb_{n,\om}\right)\right\Vert\Vert \partial_{\theta} v\Vert+\left(1+\sum_{n\geq 1}\frac{1}{\mu_{n,\om}^{2/3}}\right)\sum_{n\geq 1}|v_n|^2\\
&\lesssim  \left\Vert \partial_{\theta,\theta}\left(u-\sum_{n\leq i(\Theta)}u_nb_{n,\om}\right)\right\Vert\left\Vert \partial_{\theta} v\right\Vert+\left\Vert \partial_{\theta}v\right\Vert^2.
\end{align*}
\end{proof}

We can now prove Lemma \ref{lem_competitor} as a combination of Lemma \ref{lemma_removal_derivative} on each interval of the support of $u$, and Lemma \ref{lem_highermode_improvement}.

\begin{proof}[Proof of Lemma \ref{lem_competitor}]

As in the previous section, we first introduce the decomposition of $u$ on the connected component of its support: we write $\Spt(u)$ as a countable union of disjoint intervals $[\theta_k,\theta_k+\om_k]$ of lengths $\om_0\geq\om_1\geq\om_2\geq\hdots\to 0$. We remind that either $$|\om_0-\Theta|\leq 2c,\ \om_1\leq 2c,$$
or 
$$\Theta=2\pi,\ |\om_0-\pi|\leq 2c,\ |\om_1-\pi|\leq 2c,\ \om_2\leq 2c.$$
We call (a) the first case, (b) the second case.
We let $u^k(\theta)=u(\theta_k+\theta)\chi_{0<\theta<\om_k}$ such that $u^k\in H^2_0([0,\om_k])$. Likewise, we let $v^k(\theta)=v(\theta_k+\theta)\chi_{0<\theta<\om_k}\in H^1_0([0,\om_k])$. 

We now construct the competitor $U$ as follows: we let $\ov{U}^k$ be the competitor defined in Lemma \ref{lemma_removal_derivative} from the initial condition $u^k,v^k$, $\ov{U}(t,\theta)=\sum_{k\geq 0}\ov{U}^k(t,\theta-\theta_k)$, and $\widehat{U}$ the competitor defined in Lemma \ref{lem_highermode_improvement} (in the case (a)) or Lemma \ref{lem_highermode_improvement_double} (in the case (b)) from the initial condition $u$. Then we let
$$U(t,\theta)=\begin{cases}\ov{U}(t,\theta) & \text{ if }t<1,\\ \widehat{U}(t-1,\theta)&\text{ if }t\geq 1.\end{cases}$$
which verifies $U(0,\theta)=u(\theta)$, $\partial_t U(t,\theta)=v(\theta)$.
We denote $c,\eps,a>0$ the constants from Lemma \ref{lem_highermode_improvement}, \ref{lem_highermode_improvement_double} (note that we may always take the minimum between each) and $C$ the constant from Lemma \ref{lemma_removal_derivative}.
\begin{itemize}
\item[(a)]We first suppose we are in the case (a) i.e. $|\om_0-\Theta|\leq 2c$, $\om_1\leq 2c$. Then by Lemma \ref{lem_highermode_improvement} we have
\begin{align*}
\G(U,0)=&\,\G(\ov{U},0)+\left(\G(U,0)-\G(\ov{U},0)\right)=\G(\ov{U},0)+e^{-2}\left(\G(\widehat{U},0)-\G(u,0)\right)\\
\leq&\, \G(u,0)+C\left(\Vert \partial_\theta v\Vert^2+\Vert \partial_\theta v^0\Vert\ \left\Vert\partial_{\theta,\theta}\left(u^0-\sum_{n\leq i(\Theta)}u_n^0b_{n,\om_0}\right)\right\Vert+\sum_{k\geq 1}\Vert \partial_\theta v^k\Vert\ \Vert\partial_{\theta,\theta}u^k\Vert\right)\\
&-e^{-2}\eps\left(\G(u,0)-\frac{\Theta}{2}\right)_+-e^{-2}a\left(\left\Vert\partial_{\theta,\theta}\left(u^0-\sum_{n\leq i(\Theta)}u_n^0b_{n,\om_0}\right)\right\Vert^2+\sum_{k\geq 1}\Vert\partial_{\theta,\theta} u^k\Vert^2\right).
\end{align*}
Now, we bound
\begin{align*}
C\Vert \partial_\theta v^0\Vert\ \left\Vert\partial_{\theta,\theta}\left(u^0-\sum_{n\leq i(\Theta)}u_n^0b_{n,\om_0}\right)\right\Vert&\leq \frac{C^2e}{4a}\Vert\partial_\theta v^0\Vert^2+e^{-2}a\left\Vert\partial_{\theta,\theta}\left(u^0-\sum_{n\leq i(\Theta)}u_n^0b_{n,\om_0}\right)\right\Vert^2,\\
\text{ and }\qquad C\Vert \partial_\theta v^k\Vert\ \Vert\partial_{\theta,\theta} u^k\Vert&\leq \frac{C^2e}{4a}\Vert\partial_\theta v^k\Vert^2+e^{-2}a\left\Vert \partial_{\theta,\theta} u^k\right\Vert^2\text{ for }k\geq 1,
\end{align*}
so
\begin{align*}
\G(U,0)&\leq \G(u,0)-e^{-2}\eps\left(\G(u,0)-\frac{\Theta}{2}\right)_+ +\left(C+\frac{C^2e^2}{4a}\right)\Vert\partial_\theta v\Vert^2,
\end{align*}
which prove the result in case (a).
\item[(b) ]This time we have
\begin{align*}
\G(U,0)=&\,\G(\ov{U},0)+\left(\G(U,0)-\G(\ov{U},0)\right)=\G(\ov{U},0)+e^{-2}\left(\G(\widehat{U},0)-\G(u,0)\right)\\
\leq&\, \G(u,0)+C\left(\Vert \partial_\theta v\Vert^2+\sum_{k=0}^{1}\Vert \partial_\theta v^k\Vert\ \left\Vert\partial_{\theta,\theta}\left(u^k-u_1^kb_{1,\om_k}\right)\right\Vert+\sum_{k\geq 2}\Vert \partial_\theta v^k\Vert\ \Vert\partial_{\theta,\theta} u^k\Vert\right)\\
&-e^{-2}\eps\left(\G(u,0)-\frac{\Theta}{2}\right)_+-e^{-2}a\left(\sum_{k=0}^{1}\left\Vert\partial_{\theta,\theta}\left(u^k-u_1^kb_{1,\om_k}\right)\right\Vert^2+\sum_{k\geq 2}\Vert\partial_{\theta,\theta} u^k\Vert^2\right).
\end{align*}
Now, we bound
\begin{align*}
C\Vert \partial_\theta v^k\Vert\ \left\Vert\partial_{\theta,\theta}\left(u^k-\sum_{n\leq i(\Theta)}u_n^kb_{n,\om_0}\right)\right\Vert&\leq \frac{C^2e}{4a}\Vert\partial_\theta v^k\Vert^2+e^{-2}a\left\Vert\partial_{\theta,\theta}\left(u^k-\sum_{n\leq i(\Theta)}u_n^kb_{n,\om_0}\right)\right\Vert^2\\
&\hspace{6.5cm}\text{ for }k=0,1\\
\text{ and }\qquad C\Vert \partial_\theta v^k\Vert\ \Vert\partial_{\theta,\theta} u^k\Vert&\leq \frac{C^2e}{4a}\Vert\partial_\theta v^k\Vert^2+e^{-2}a\left\Vert \partial_{\theta,\theta} u^k\right\Vert^2\qquad\text{ for }k\geq 2,
\end{align*}
so
\begin{align*}
\G(U,0)&\leq e^{-2}\eps\pi+(1-e^{-2}\eps)\G(\ov{U},0)+\left(C+\frac{C^2e^2}{4a}\right)\Vert\partial_\theta v\Vert^2,
\end{align*}
which prove the result in case (b).
\end{itemize}
\end{proof}

\subsection{Estimates on biharmonic extensions}\label{subsec_biharmonic}

We now prove Lemma \ref{lem_bihext}, which is an analog of Lemma \ref{lem_competitor} in the case $\Spt(u)=\Sp$. Our method in this case is a more elementary construction of competitors: we consider the competitor corresponding to the biharmonic extension in the disk coordinates. We follow closely the energy difference between a biharmonic function in $\D_1$ and its $2$-homogeneous extension.

In what follows, for any $f\in L^1(\Sp)$, we denote
$$c_n[f]=\frac{1}{2\pi}\int_{\Sp}e^{-in\theta}f(e^{i\theta})d\theta,$$
its Fourier coefficients: by the convention we chose we have
$$\Vert f\Vert_{L^2(\partial\D_1)}^2=2\pi\sum_{n\in\Z}|c_n[f]|^2.$$
\begin{proof}[Proof of Lemma \ref{lem_bihext}]
We remind that $v$ being a minimizer of $\G(\cdot,0)$ means that
$$u(re^{i\theta}):=r^2v(-\log(r),\theta)$$
belongs to $\M(\D_1)$. Consider $\tilde{u}$ the biharmonic extension of the boundary data $(u|_{\partial\D_1},\partial_r u|_{\partial\D_1})$ in $\D_1$, which can be decomposed as
$$\tilde{u}\left(re^{i\theta}\right)=\sum_{n\in\Z}\left(a_n r^{|n|+2}+b_nr^{|n|}\right)e^{in\theta},$$
for some coefficients $a_n,b_n$ (we recall this is called Goursat's decomposition, also discussed at \eqref{eq_Goursat}). As we already noticed in the proof of Lemma \ref{lem_computationWbih}, the energy $\int_{\D_1}|\Delta \tilde{u}|^2$ only depends on the $(a_n)_n$ coefficients, and more precisely $\Delta \tilde{u}=\sum_{n\in\Z}4(|n|+1)a_n r^{|n|}e^{in\theta}$, so
$$
\int_{\D_1}|\Delta \tilde{u}|^2=\int_{0}^{1}2\pi r\left|\sum_{n\in\Z}4(|n|+1)r^{|n|}e^{in\theta}\right|^2dr=16\pi\sum_{n\in\Z}(|n|+1)|a_n|^2.
$$
We may identify
\begin{align*}
c_n[u]&=c_n[\tilde{u}]=a_n+b_n,\\
c_n[\partial_r u]&=c_n[\partial_r\tilde{u}]=(|n|+2)a_n+|n|b_n,
\end{align*}
so $a_n=\frac{c_n[\partial_r u]-|n|c_n[u]}{2}$ and
$$E(\tilde{u};\D_1)=\int_{\D_1}\left(|\Delta \tilde{u}|^2+\chi_{\tilde{u}\neq 0}\right)=\pi+2\pi\sum_{n\in\Z}2(|n|+1)\Big|c_n[\partial_r u]-|n|c_n[u]\Big|^2.$$
We identify $u(e^{i\theta})=v(0,\theta)$, $\partial_r u(e^{i\theta})=2v(0,\theta)-\partial_t v(0,\theta)$, so $c_n[u]=c_n[v(0,\cdot)]$, $c_n[\partial_r u]=2c_n[v]-c_n[\partial_t v(0,\cdot)]$, meaning 
$$E(\tilde{u};\D_1)=\pi+\pi\sum_{n\in\Z}4(|n|+1)\Big|c_n[\partial_t v(0,\cdot)]+(|n|-2)c_n[v(0,\cdot)]\Big|^2.$$
Since $u\in\M(\D_1)$, we have $E(u;\D_1)\leq E(\tilde{u};\D_1)$, so by Lemma \ref{lem_change_variable_diskexp}, we get
\begin{align*}
\G(v,0)\leq&\,  \pi+\pi\sum_{n\in\Z}4(|n|+1)\Big|c_n[\partial_t v(0,\cdot)]+(|n|-2)c_n[v(0,\cdot)]\Big|^2\\
&-2\int_{\partial \Cy_\tau}\left(2\left(\partial_tv\right)^2-4v \partial_tv+4v^2-\left(\partial_\theta v\right)^2+\partial_{t,\theta}v\partial_\theta v\right)(0,\cdot)d\theta.
\end{align*}
We express everything in terms of Fourier coefficients of $v(0,\cdot)$, $\partial_t v(0,\cdot)$. To simplify notations we write $x_n=c_n[v(0,\cdot)]$, $y_n=c_n[\partial_t v(0,\cdot)]$:
\begin{align*}
\G(v,0)-\pi \leq&\, \pi\sum_{n\in\Z}4(|n|+1)\Big|y_n+(|n|-2)x_n\Big|^2\\
&+\pi\sum_{n\in\Z}-8|y_n|^2-4(n^2-4)\Re(\ov{x_n}y_n)+4(n^2-4)|x_n|^2,
\end{align*}
so
\begin{equation}\label{eq_Gvmpi}
\G(v,0)-\pi\leq\pi\sum_{n\in\Z}4n^2(|n|-2)|x_n|^2+4|n|(|n|-2)\Re(\ov{x_n}y_n)+4(|n|-1)|y_n|^2.
\end{equation}
We can also express in terms of Fourier coefficients:
\begin{equation}\label{eq_Gv0mpi}
\G(v(0,\cdot),0)-\pi=\frac{1}{2}\Vert\partial_{\theta,\theta}v(0,\cdot)\Vert^2-2\Vert\partial_{\theta}v(0,\cdot)\Vert^2=\pi\sum_{n\in\Z}n^2(n^2-4)|x_n|^2.
\end{equation}
Here we used the fact that the support of $v(0,\cdot)$ has measure $2\pi$. 
Now we compare the $n$-th term of \eqref{eq_Gvmpi} and \eqref{eq_Gv0mpi}:
\begin{itemize}[label=\textbullet]
\item if $|n|\geq 2$, then we bound $4|n|(|n|-2)\Re(\ov{x_n}y_n)\leq \frac{1}{2}n^2(|n|-2)|x_n|^2+8(|n|-2)|y_n|^2$, so
\begin{align*}
4n^2(|n|-2)|x_n|^2+4|n|(|n|-2)\Re(\ov{x_n}y_n)&+4(|n|-1)|y_n|^2\\
&\leq \frac{9}{2}n^2(|n|-2)|x_n|^2+(12|n|-20)|y_n|^2\\
&\leq 0.9 n^2(n^2-4)|x_n|^2+(12|n|-4)|y_n|^2.
\end{align*}
Here the factor ``$0.9$'' is sharp for $n=3$ in the last inequality.
\item For $n=0$, then the $0$-th terms of \eqref{eq_Gvmpi} and \eqref{eq_Gv0mpi} are respectively $-4|y_0|^2$, $0$.
\item For $n=\pm 1$, we write
\begin{align*}
-4|x_{\pm 1}|^2-4\Re(\ov{x_{\pm 1}}y_{\pm 1})\leq -3|x_{\pm 1}|^2+4|y_{\pm 1}|^2\leq (1^2(1^2-4))|x_{\pm 1}|^2+(12-4)|y_{\pm 1}|^2.
\end{align*}
\end{itemize}
Summing every contributions, we obtain 
\begin{align*}
\G(v,0)-\G(v(0,\cdot),0)&\leq -0.1\pi\sum_{n\in\Z\setminus\{-1,1\}} n^2(n^2-4)|x_n|^2 +\pi\sum_{n\in\Z}(12|n|-4)|y_n|^2\\
&\leq -0.1\left(\pi\sum_{n\in\Z} n^2(n^2-4)|x_n|^2 \right)_++2\pi\sum_{n\in\Z}(6|n|-2)|y_n|^2\\
&\leq -0.1\left(\G(v(0,\cdot),0)-\pi\right)_+ +6\Vert\partial_{t,\theta}v(0,\cdot)\Vert^2-2\Vert\partial_{t}v(0,\cdot)\Vert^2,
\end{align*}
which implies the result.
\end{proof}

\subsection{Proof of the epiperimetric inequality}\label{subsec_conclusion_epiperimetry}

As a consequence of Lemma \ref{lem_competitor} and  \ref{lem_bihext}, we obtain the following corollary.
\begin{corollary}\label{cor_epiest}
Let $\Theta\in\{\pi,t_1,2\pi\}$. There exist constants $c,\eps,C>0$ such that, if $u\in H^2_\lin(\Cy_0)$ is a minimizer of $\G(\cdot,0)$ that verifies the support condition
$$
\begin{cases}
(c,\Theta-c)\subset \Spt(u(0,\cdot))\subset (-c,\Theta+c)&\text{ when }\Theta\in\{\pi,t_1\}\\
(c,\pi-c)\cup(\pi+c,2\pi-c)\subset \Spt(u(0,\cdot))&\text{ when }\Theta=2\pi,\\
\end{cases}
$$
then
\begin{equation}\label{eq_epiest_improved}
\G(u,0)\leq \G(u(0,\cdot),0)- \eps\left(\G(u(0,\cdot),0)-\frac{\Theta}{2}\right)_++C\Vert\partial_{t,\theta} u(0,\cdot)\Vert_{L^2(\Sp)}^2-2\Vert \partial_t u(0,\cdot)\Vert_{L^2(\Sp)}^2.
\end{equation}
\end{corollary}
\begin{proof}
Denote $c,\eps,C$ the constants from Lemma \ref{lem_competitor}. If $u$ verifies the support condition and $\Spt(u)\neq \Sp$, then by Lemma \ref{lem_competitor} applied to $u$ we get
\begin{align*}
\G(u,0)&\leq \G(u(0,\cdot),0)- \eps\left(\G(u(0,\cdot),0)-\frac{\Theta}{2}\right)_++C\Vert\partial_{t,\theta} u(0,\cdot)\Vert_{L^2(\Sp)}^2\\
&\leq \G(u(0,\cdot),0)- \eps\left(\G(u(0,\cdot),0)-\frac{\Theta}{2}\right)_++\left(C+8\right)\Vert\partial_{t,\theta} u(0,\cdot)\Vert_{L^2(\Sp)}^2-2\Vert \partial_t u(0,\cdot)\Vert_{L^2(\Sp)}^2
\end{align*}
where we used the Poincaré inequality $\Vert f\Vert_{L^2(\Sp)}^2\leq 4 \Vert \partial_\theta f\Vert_{L^2(\Sp)}^2$, which is valid since $\Spt(u(0,\cdot))\neq\Sp$.

Suppose now that $\Spt(u)=\Sp$ (so in particular we are in the case $\Theta=2\pi$), then the result is directly implied by Lemma \ref{lem_bihext} (for any $\eps\leq 0.1$, $C\geq 6$).
\end{proof}

We now have all the tools to prove the epiperimetric inequality. We will make use of the following interpolation inequality: for any smooth bounded set $D\subset\R^2$, there exists $c(D)>0$ such that for any $u\in W^{2,4}(D) (\subset\mathcal{C}^1(D))$:
\begin{equation}\label{eq_interp_C1H1}
\Vert u\Vert_{\mathcal{C}^{1}(D)}\leq c(D)\Vert u\Vert_{H^1(D)}^{\frac{1}{2}}\Vert u\Vert_{W^{2,4}(D)}^{\frac{1}{2}}.
\end{equation}
This is a consequence of the following classical Gagliardo-Nirenberg inequality applied to $f=\nabla u$:
$$\Vert f\Vert_{L^\infty(D)}\leq C(D)\Vert f\Vert_{L^2(D)}^\frac{1}{2}\Vert f\Vert_{W^{1,4}(D)}^\frac{1}{2}.$$

\begin{proof}[Proof of Theorem \ref{th_epiperimetry_exp}]
Let $c_0>0$ be some quantity that we fix (arbitrarily small) later such that
$$\Vert u-b\Vert_{H^1(\Cy_0\setminus\Cy_1)}\leq c_0\Vert b\Vert_{H^1(\Sp)}.$$
By the Cacciopoli inequality of Lemma \ref{lem_Cacciopoli} and the BMO estimate of Lemma \ref{lem_BMO}, we have 
$$\Vert u\Vert_{W^{2,4}(\Cy_{\frac{1}{4}}\setminus\Cy_{\frac{3}{4}})}\lesssim \Vert u\Vert_{H^1(\Cy_0\setminus\Cy_1)}\lesssim\Vert b\Vert_{H^1(\Sp)}.$$
By the interpolation inequality \eqref{eq_interp_C1H1} we then have
\begin{equation}\label{eq_C1est}
\Vert u-b\Vert_{\mathcal{C}^{1}(\Cy_{\frac{1}{4}}\setminus\Cy_{\frac{3}{4}})}\lesssim \Vert u-b\Vert_{H^{1}(\Cy_{\frac{1}{4}}\setminus\Cy_{\frac{3}{4}})}^\frac{1}{2}\Vert u-b\Vert_{W^{2,4}(\Cy_{\frac{1}{4}}\setminus\Cy_{\frac{3}{4}})}^\frac{1}{2}\lesssim c_0^\frac{1}{2}\Vert b\Vert_{H^{1}(\Sp)}.
\end{equation}
We claim that for a sufficiently small $c_0$, $u(t,\cdot)$ verifies the support hypothesis of Corollary \ref{cor_epiest} for almost every $t$, up to a rotation in $\theta$. We denote by $c$ the constant from Corollary \ref{cor_epiest}.

First note that $\{\partial_t u\neq 0\}\subset\Spt(u)$, so $\{\partial_t u(t,\cdot)\neq 0\}\subset\Spt(u(t,\cdot))$ for almost every $t$. Then we differentiate whether $\Theta\in\{\pi,t_1\}$ or $\Theta=2\pi$.
\begin{itemize}
\item $\Theta\in\{\pi,t_1\}$. Suppose without loss of generality that $b=b^\mathrm{I}$ or $b=b^\mathrm{II}$. Then 
$$\inf_{\theta\in [c,\Theta-c]}\left( |b(\theta)|+|\partial_\theta b(\theta)|\right)\gtrsim 1.$$
So by the $\mathcal{C}^{1}$ estimate of $u-b$ in \eqref{eq_C1est}, we have for a small enough $c_0$:
$$\forall t\in\left[\frac{1}{4},\frac{3}{4}\right],\ \inf_{\theta\in [c,\Theta-c]}\left( |u(t,\theta)|+|\partial_\theta u(t,\theta)|\right)\gtrsim 1.$$
As a consequence, $(c,\Theta-c)\subset\Spt(u(t,\cdot))$ for any $t\in \left[\frac{1}{4},\frac{3}{4}\right]$.

On the other hand, for any disk $D=\D_{p,c}$ with $p\in \left[\frac{1}{4},\frac{3}{4}\right]\times \left(\Sp\setminus \left[-c,\Theta+c\right]\right)$, we have
$$\Vert u\Vert_{H^1(D)}= \Vert u-b\Vert_{H^1(\Cy_0\setminus\Cy_1)}\lesssim c_0.$$
For a small enough $c_0$ we may apply Lemma \ref{lem_nondegeneracy}, so $u(p)=0$. This implies that for every $t\in \left[\frac{1}{4},\frac{3}{4}\right]$, we have $\Spt(u(t,\cdot))\subset (-c,\Theta+c)$.
\item $\Theta=2\pi$. Up to a rotation (and change of sign) there exists $(\alpha,\beta)\in\R_+^2$, such that 
$$b(\theta)=\alpha -\beta \cos(2\theta).$$
In particular
\begin{align*}
|b(\theta)|^2+\frac{1}{4}|b'(\theta)|^2&=\alpha^2+\beta^2-2\alpha\beta\cos(2\theta)\\
&\geq \frac{1-\cos(2\theta)}{2}(\alpha+\beta)^2\\
&\gtrsim (1-\cos(2\theta))\Vert b\Vert_{H^1(\Sp)}^2,
\end{align*}
so for any $\theta\in (c,\pi-c)\cup(\pi+c,2\pi-c)$, we have $|b(\theta)|+|b'(\theta)|\gtrsim \Vert b\Vert_{H^1(\Sp)}$. By the $\mathcal{C}^1$ estimate \eqref{eq_C1est} we obtain for a small enough $c_0$:
$$\forall t\in \left[\frac{1}{4},\frac{3}{4}\right],\ \inf_{\theta\in [c,\pi-c]\cup[\pi+c,2\pi-c]}\left(|u(t,\theta)|+|u'(t,\theta)|\right)\gtrsim \Vert b\Vert_{H^1(\Sp)},$$
so the support hypothesis of Corollary \ref{cor_epiest} is verified.
\end{itemize}
Let $\phi\in\mathcal{C}^\infty_c\left(\left]\frac{1}{4},\frac{3}{4}\right[,[0,1]\right)$, such that
$$|\phi'|\lesssim 1,\ \phi|_{\left[\frac{1}{3},\frac{2}{3}\right]}\equiv 1.$$
Then
\begin{align*}
\W(u,0)&-\W(u,1)=\int_{0}^{1}-\W'(u,t)dt\geq \int_{0}^{1}-\phi(t)^2\W'(u,t)dt\\
=&\,\int_{0}^{1}\left(-\phi(t)^2\G'(u,t)-2\phi(t)^2\partial_t\langle \partial_{t}u,\partial_{t}u-\partial_{t,t}u\rangle\right)dt\\
=&\,\int_{0}^{1}\phi(t)^2\left(\Vert \partial_{t,t}u\Vert^2+2\Vert \partial_{t,\theta}u\Vert^2+\Vert\partial_{\theta,\theta}u\Vert^2-4\Vert \partial_{\theta}u\Vert^2+|\{u(t,\cdot)\neq 0\}|-2\G(u,t)\right)dt\\
&-4\int_{0}^{1}\phi(t)\left(\phi(t)+\phi'(t)\right)\langle \partial_{t}u,\partial_{t,t}u\rangle dt.
\end{align*}
Since $\G(u(t,\cdot),t)=\frac{1}{2}\left(\Vert\partial_{\theta,\theta}u\Vert^2-4\Vert \partial_{\theta}u\Vert^2+|\{u(t,\cdot)\neq 0\}|\right)$, we may rewrite this as
\begin{equation}\label{eq_aux_1}
\begin{split}
\W(u,0)-\W(u,1)\geq&\int_{0}^{1}\phi(t)^2\left(\Vert \partial_{t,t}u\Vert^2+2\Vert \partial_{t,\theta}u\Vert^2+2\G(u(t,\cdot),t)-2\G(u,t)\right)dt\\
&-4\int_{0}^{1}\phi(t)\left(\phi(t)+\phi'(t)\right)\langle \partial_{t}u,\partial_{t,t}u\rangle dt.
\end{split}
\end{equation}
Now we apply Corollary \ref{cor_epiest} to $u(t+\cdot,\cdot)$ (which verifies the hypothesis for almost every $t\in\left[\frac{1}{4},\frac{3}{4}\right]$, from the previous discussion). The estimate of Corollary \ref{cor_epiest} may be rewritten as
\begin{align*}
\G(u(t,\cdot),t)-\eps\left(\G(u(t,\cdot)-\frac{\Theta}{2}\right)_+-\G(u,t)\geq -C\Vert\partial_{t,\theta} u(0,\cdot)\Vert_{L^2(\Sp)}^2+2\Vert \partial_t u(0,\cdot)\Vert_{L^2(\Sp)}^2,
\end{align*}
which implies (supposing without loss of generality that $\eps$ is less that $\frac{1}{2}$, so that $-\frac{1}{1-\eps}\geq -2$):
\begin{align*}
\G(u(t,\cdot),t)-\G(u,t)&\geq \frac{\eps}{1-\eps}\left(\G(u,t)-\frac{\Theta}{2}\right) -2C\Vert\partial_{t,\theta} u(t,\cdot)\Vert_{L^2(\Sp)}^2+\frac{2}{1-\eps}\Vert \partial_t u(t,\cdot)\Vert_{L^2(\Sp)}^2\\
&\geq\eps\left(\W(u,t)-\frac{\Theta}{2}\right)+2\Vert\partial_t u(t,\cdot)\Vert^2-4\eps\left|\langle\partial_t u,\partial_{t,t}u\rangle\right| -2C\Vert\partial_{t,\theta} u(t,\cdot)\Vert_{L^2(\Sp)}^2.
\end{align*}
We inject this estimate in \eqref{eq_aux_1}:
\begin{equation}\label{eq_aux_2}
\begin{split}
\W(u,0)-\W(u,1)\geq&\,2\eps\left(\W(u,1)-\frac{\Theta}{2}\right)\int_{0}^{1}\phi(t)^2dt+\int_{0}^{1}4\phi(t)^2\Vert\partial_t u\Vert^2dt\\
&+\int_{0}^{1}\phi(t)^2\left(\Vert \partial_{t,t}u\Vert^2+2\Vert \partial_{t,\theta}u\Vert^2-8\eps\left|\langle\partial_{t,t}u,\partial_t u\rangle\right| -4C\Vert\partial_{t,\theta}u\Vert^2\right)dt\\
&-4\int_{0}^{1}\phi(t)\left(\phi(t)+\phi'(t)\right)\langle \partial_{t}u,\partial_{t,t}u\rangle dt
\end{split}
\end{equation}
We have $-8\eps\left|\langle\partial_{t,t}u,\partial_t u\rangle\right|\geq -\Vert\partial_{t,t}u\Vert^2-\Vert \partial_t u\Vert^2,$
for a sufficiently small $\eps>0$ (which we can assume without loss of generality). Similarly, we bound
$$-4\phi(t)\left(\phi(t)+\phi'(t)\right)\langle \partial_{t}u,\partial_{t,t}u\rangle \geq -\phi(t)^2\Vert\partial_t u\Vert^2-8(\phi(t)^2+\phi'(t)^2)\Vert\partial_{t,t}u\Vert^2.$$
We obtain (for some universal constant $C'>0$):
\begin{equation}\label{eq_aux_3}
\begin{split}
\W(u,0)-\W(u,1)\geq&\,2\eps\left(\W(u,1)-\frac{\Theta}{2}\right)\int_{0}^{1}\phi(t)^2dt+\int_{0}^{1}2\phi(t)^2\Vert\partial_t u\Vert^2dt\\
&-C'\int_{0}^{1}\left(\phi(t)^2+\phi'(t)^2\right)\left(\Vert\partial_{t,\theta}u\Vert^2+\Vert\partial_{t,t}u\Vert^2\right)dt.
\end{split}
\end{equation}
We remind that $\phi|_{\left[\frac{1}{3},\frac{2}{3}\right]}\equiv 1$ and $|\phi'|\lesssim 1$, so
\begin{equation}\label{eq_aux_4}
\begin{split}
\W(u,0)-\W(u,1)\geq&\,\frac{2\eps}{3}\left(\W(u,1)-\frac{\Theta}{2}\right)+2\Vert \partial_t u\Vert_{L^2(\Cy_{\frac{1}{3}}\setminus\Cy_{\frac{2}{3}})}^2\\
&-C''\int_{0}^{1}\left(\Vert\partial_{t,\theta}u\Vert^2+\Vert\partial_{t,t}u\Vert^2\right)dt,
\end{split}
\end{equation}
for some constant $C''>0$. Finally, $f\mapsto \sqrt{\Vert f\Vert_{L^2(\Cy_{\frac{1}{3}}\setminus\Cy_{\frac{2}{3}})}^2+\Vert \nabla f\Vert_{L^2(\Cy_0\setminus\Cy_1)}^2}$ is a norm for $H^1(\Cy_0\setminus\Cy_1)$, so
$$\Vert f\Vert_{L^2(\Cy_{\frac{1}{3}}\setminus\Cy_{\frac{2}{3}})}^2+\Vert \nabla f\Vert_{L^2(\Cy_0\setminus\Cy_1)}^2\gtrsim \Vert f\Vert_{L^2(\Cy_0\setminus\Cy_1)}^2+\Vert \nabla f\Vert_{L^2(\Cy_0\setminus\Cy_1)}^2,$$
and there is some universal constant $d\in (0,1)$, $D>0$ such that 
\begin{equation}\label{eq_aux_5}
\W(u,0)-\W(u,1)\geq d\left(\left(\W(u,1)-\frac{\Theta}{2}\right) +\Vert \partial_t u\Vert_{H^1(\Cy_{0}\setminus\Cy_{1})}^2\right)-4D\int_{0}^{1}\left(\Vert\partial_{t,\theta}u\Vert^2+\Vert\partial_{t,t}u\Vert^2\right)dt.
\end{equation}
Using the expression of $\W'(u,t)$ (see Theorem \ref{Th_Monotonicity_exp}), we have
\begin{equation}\label{eq_proofepi_monotonicity}
4\int_{0}^{1}\left(\Vert \partial_{t,t}u\Vert^2+\Vert \partial_{t,\theta}u\Vert^2\right)dt=\W(u,0)-\W(u,1),
\end{equation}
so we may replace the last term of \eqref{eq_aux_5} with $D(\W(u,0)-\W(u,1))$. We obtain
$$\W(u,0)-\W(u,1)\geq \frac{d}{D+1}\left(\left(\W(u,1)-\frac{\Theta}{2}\right) +\Vert \partial_t u\Vert_{H^1(\Cy_{0}\setminus\Cy_{1})}^2\right),$$
which can be rearranged into
$$\W(u,1)\leq\W(u,0)-\frac{d}{D+1-d}\left(\left(\W(u,0)-\frac{\Theta}{2}\right) +\Vert \partial_t u\Vert_{H^1(\Cy_{0}\setminus\Cy_{1})}^2\right),$$
which is the result.
\end{proof}
We also need the following lemma to control the variation of the blow-up sequence $(u_r)$ by the variation of $\W$.
\begin{lemma}\label{lem_control_variation_u}
There exist $c,C>0$ such that, for any $u\in H^2_\lin(\Cy_0)$ a local minimizer of $\G(\cdot\ ;0)$, assume that there exists $b\in\mathscr{B}_{\mathrm{hom}}^{\Theta}$ if $\Theta\in\{\pi,t_1\}$ (resp. $b\in\mathrm{Span}(1,\cos(2\theta),\sin(2\theta))$ if $\Theta=2\pi$) such that
$$\Vert u-b\Vert_{H^1(\Cy_0\setminus\Cy_1)}\leq c\Vert b\Vert_{H^1(\Sp)},\ \W(u,+\infty)\geq \frac{\Theta}{2},$$
then
$$\Vert u-u(1+\cdot,\cdot)\Vert_{H^1(\Cy_0\setminus\Cy_1)}\leq C\sqrt{\W(u,0)-\W(u,+\infty)}.$$
\end{lemma}
\begin{proof}
For any $\tau\in (0,1)$, we write
\begin{align*}
\int_{\Sp}\left|\nabla u(\tau+1,\theta)-\nabla u(\tau,\theta)\right|^2d\theta&=\int_{\Sp}\left|\int_{\tau}^{\tau+1}\partial_t \nabla u(s,\theta)ds\right|^2d\theta\\
&\leq \int_{\Sp}\int_{\tau}^{\tau+1}|\partial_t\nabla u(s,\theta)|^2dsd\theta\\
&\leq \frac{\W(u,0)-\W(u,+\infty)}{4}\text{ by Theorem }\ref{Th_Monotonicity_exp}.
\end{align*}
Likewise,
\begin{align*}
\int_{\Sp}\left| u(\tau+1,\theta)- u(\tau,\theta)\right|^2d\theta&=\int_{\Sp}\left|\int_{\tau}^{\tau+1}\partial_t  u(s,\theta)ds\right|^2d\theta\leq \int_{\Sp}\int_{\tau}^{\tau+1}|\partial_t u(s,\theta)|^2dsd\theta\\
&= \int_{\Sp}\int_{\tau}^{\tau+1}\left|\partial_t u(\tau,\theta)+\int_{\tau}^{s}\partial_{t,t}u(s',\theta)ds'\right|^2dsd\theta\\
&\leq 2\int_{\Sp}|\partial_t u(\tau,\theta)|^2d\theta+2\int_{\Sp}\int_{\tau}^{\tau+1}|\partial_{t,t}u(s,\theta)|^2dsd\theta\\
&\leq 2\int_{\Sp}|\partial_t u(\tau,\theta)|^2d\theta+\frac{\W(u,0)-\W(u,+\infty)}{2}
\end{align*}
so
\begin{align*}
\int_{\Cy_0\setminus\Cy_1}\int_{\Sp}\left| u(\tau+1,\theta)- u(\tau,\theta)\right|^2d\theta d\tau&\leq 2\int_{\Cy_0\setminus\Cy_1}|\partial_t u(\tau,\theta)|^2d\theta d\tau+\frac{\W(u,0)-\W(u,+\infty)}{2}.
\end{align*}
When $\frac{\Vert u-b\Vert_{H^1(\Cy_0\setminus\Cy_1)}}{\Vert b\Vert_{H^1(\Sp)}}$ is small enough, then by the epiperimetry inequality of Theorem \ref{th_epiperimetry_exp} we have $\int_{\Cy_0\setminus\Cy_1}|\partial_t u|^2\lesssim \W(u,0)-\W(u,1)(\leq \W(u,0)-\W(u,+\infty))$, which concludes the proof.
\end{proof}

We now prove the main result of this section.

\begin{proof}[Proof of Theorem \ref{th_epiperimetry}]
We let $v(t,\theta)=e^{2t}u\left(e^{-t+i\theta}\right)$, so that $v$ is a minimizer of $\G(\cdot,0)$. We let also $b\in\mathscr{B}_{\mathrm{hom}}^\Theta$ be such that
$$\ov{u}(re^{i\theta})=r^2b(\theta).$$
We remind that $\W(v,t)=W(u,e^{-t})$ (by Theorem \ref{Th_Monotonicity_disk}). Let $c_1>0$ to be fixed small enough later. The bound on  $\Vert u-\ov{u}\Vert_{H^1(\D_{1}\setminus\D_{e^{-1}})}$ implies
$$\Vert v-b\Vert_{H^1(\Cy_0\setminus\Cy_1)}\lesssim c_1\Vert b\Vert_{H^1(\Sp)}.$$
Then for a small enough $c_1$, Theorem \ref{th_epiperimetry_exp} applies and we have
$$W(u,e^{-1})\leq \eta \frac{\Theta}{2}+(1-\eta)W(u,1).$$
Similarly, for a small enough $c_1$, we may apply Lemma \ref{lem_control_variation_u}, which implies
$$\Vert u-u_{e^{-1}}\Vert_{H^1(\D_1\setminus\D_{e^{-1}})}\lesssim \Vert v-v(1+\cdot,\cdot)\Vert_{H^1(\Cy_0\setminus\Cy_1)}\leq C_1\sqrt{W(u,1)-\frac{\Theta}{2}}.$$
Since $W(u,0)\geq \frac{\Theta}{2}$, this concludes the proof.
\end{proof}

\section{Description of flat, angular and isolated boundary points.}\label{sec_reg}

In this section, we prove the main Theorems \ref{mr_blowup}, \ref{mr_typeIII_IV}, \ref{mr_typeII},  \ref{mr_typeI} (in this order), and conclude with the proof of Theorem \ref{mr_total}. 

The proof of Theorem \ref{mr_blowup} and \ref{mr_typeIII_IV} is based on the epiperimetric inequality and variation control given by Theorem \ref{th_epiperimetry}: we will see in subsection \ref{subsec_uniquenessblowup} this implies an explicit polynomial rate of convergence of $u_r$ to a (unique) limit, given in Proposition \ref{prop_convpolyblow_up}.

We then prove in subsection \ref{subsec_C1alphaest}  $\eps$-regularity results near flat boundary points (Proposition \ref{prop_C1a}) and angular boundary points (Proposition \ref{prop_II_C1a}); assuming $u$ is sufficiently close to a flat or angular blow-up, we prove that the support of $u$ is a $\mathcal{C}^{1,\alpha}$ perturbation of the support of the associated homogeneous solution. This will imply Theorem \ref{mr_typeII}.

In subsection \ref{subsec_higherreg}, we make a conformal change of variable to transform an overdetermined Stokes equation on a $\mathcal{C}^{1,\alpha}$ free boundary into an overdetermined elliptic equation on a disk: from this we obtain the higher regularity of the boundary and Theorem \ref{mr_typeI}.

Finally, we prove the main Theorem \ref{mr_total} in subsection \ref{subsec_proof_mr_total} by gathering all the previous results.

\subsection{Uniqueness and speed of convergence of blow-ups}\label{subsec_uniquenessblowup}

\begin{proposition}\label{prop_convpolyblow_up}
Let $\Theta\in\{\pi,t_1,2\pi\}$. There exist constants $c_2,C_2>0$, $\gamma\in (0,1)$ such that the following holds: let $\ov{u}\in\M_{\mathrm{hom}}^{\Theta}$ if $\Theta\in\{\pi,t_1\}$ (resp. $\ov{u}\in\mathrm{Span}(x^2+y^2,x^2-y^2,xy)$ if $\Theta=2\pi$), $u\in\M(\D_1)$ such that
$$W(u,0)\geq\frac{\Theta}{2},\qquad \Vert u-\ov{u}\Vert_{H^1(\D_1)}\leq c_2\Vert \ov{u}\Vert_{H^1(\D_1)}.$$
Then there exists $\widehat{u}\in \M_{\mathrm{hom}}^{\Theta}$ such that for any $r\in (0,1]$:
$$\Vert u_r-\widehat{u}\Vert_{H^1(\D_1)}\leq C_2\min\left(r,\frac{\Vert u-\ov{u}\Vert_{H^1(\D_1)}}{\Vert \ov{u}\Vert_{H^1(\D_1)}}\right)^\gamma\Vert \ov{u}\Vert_{H^1(\D_1)},$$
and 
$$\Vert\ov{u}-\widehat{u}\Vert_{H^1(\D_1)}\leq C_2 \Vert u-\ov{u}\Vert_{H^1(\D_1)}^\gamma\Vert \ov{u}\Vert_{H^1(\D_1)}^{1-\gamma}.$$
\end{proposition}
Note that the hypothesis $W(u,0)\geq\frac{\Theta}{2}$ is implied by the existence of a blow-up of opening $\Theta$ (i.e. type I if $\Theta=\pi$, II if $\Theta=t_1$, III, IV if $\Theta=2\pi$).

In the proof, we will use the following scaling property: for any $r\in (0,1)$, for any $u\in H^1(rD)$ for some domain $D\subset\R^2$,
\begin{equation}\label{eq_scaling_property}
r^{-2}\Vert u\Vert_{H^1(rD)}\leq \Vert u_r\Vert_{H^1(D)}\leq r^{-3}\Vert u\Vert_{H^1(rD)}.
\end{equation}
Indeed, this is a consequence of the change of variable
$\Vert u_r\Vert_{H^1(D)}=\left(\int_{rD}r^{-4}u^2+r^{-6}|\nabla u|^2\right)^{\frac{1}{2}}$.
\begin{proof}
Let $c_2>0$ be some constant that will be fixed arbitrarily small later. Let $c_1,C_1,\eta$ be the constants from Theorem \ref{th_epiperimetry}. By Lemma \ref{lem_W_cacciopoli}, we have
$$W\left(u,\frac{1}{2}\right)\leq C\Vert\ov{u}\Vert_{H^1(\D_1)}^2,$$
for some universal constant $C>0$. We define the auxiliary constants 
$$\alpha=\left(\frac{c_1\left(1-\sqrt{1-\eta}\right)}{2C_1}\right)^2,\quad K=\left\lceil\frac{\log(\alpha/C)}{\log(1-\eta)}\right\rceil_+,\quad \ov{r}=\frac{1}{2}e^{-K},$$
i.e. $K\in\N$ is the smallest integer such that $(1-\eta)^{K} C\leq \alpha$. We suppose $c_2$ verifies
\begin{equation}\label{eq_smallness_c2}
c_2\leq\frac{\ov{r}^3c_1}{2}.
\end{equation}
For any $k\in\{0,1,2,\hdots,K-1\}$, we have (by the scaling formula \eqref{eq_scaling_property}):
\begin{align*}
\Vert u_{\frac{1}{2}e^{-k}}-\ov{u}\Vert_{H^1(\D_1\setminus\D_{e^{-1}})}&=\Vert \left(u-\ov{u}\right)_{\frac{1}{2}e^{-k}}\Vert_{H^1(\D_1\setminus\D_{e^{-1}})}&\text{ since }\ov{u}\text{ is homogeneous}\\
&\leq 8e^{3k}\Vert u-\ov{u}\Vert_{H^1(\D_1)}&\text{ by the scaling \eqref{eq_scaling_property}}\\
&\leq \ov{r}^{-3}c_2\Vert \ov{u}\Vert_{H^1(\D_1)}&\text{ by definition of }\ov{r}\\
&\leq \frac{c_1}{2}\Vert \ov{u}\Vert_{H^1(\D_1)}&\text{ by assumption \eqref{eq_smallness_c2}}.
\end{align*}
As a consequence, Theorem \ref{th_epiperimetry} applies to $u_{\frac{1}{2}e^{-k}}$, and we get 
$$\forall k=0,1,2,\hdots,K-1,\ W\left(u,\frac{1}{2}e^{-k-1}\right)-\frac{\Theta}{2}\leq (1-\eta)\left(W\left(u,\frac{1}{2}e^{-k}\right)-\frac{\Theta}{2}\right).$$
By induction on $k=0,1,\hdots,K-1$, we obtain
\begin{align*}
W\left(u,\ov{r}\right)-\frac{\Theta}{2}&=W\left(u,\frac{1}{2}e^{-K}\right)-\frac{\Theta}{2}\leq (1-\eta)^K\left(W\left(u,\frac{1}{2}\right)-\frac{\Theta}{2}\right)\\
&\leq C(1-\eta)^K\Vert\ov{u}\Vert_{H^1(\D_1)}^2\text{ by Lemma \ref{lem_W_cacciopoli}}\\
&\leq \alpha\Vert\ov{u}\Vert_{H^1(\D_1)}^2\text{ by definition of }\alpha.
\end{align*}
We claim now that for any $k\in\N$, we have
\begin{equation}\label{eq_inductive}
W\left(u,e^{-k}\ov{r}\right)\leq \frac{\Theta}{2}+(1-\eta)^k\alpha\Vert\ov{u}\Vert_{H^1(\D_1)}^2,\quad \frac{\Vert u_{e^{-k}\ov{r}}-\ov{u}\Vert_{H^1(\D_1\setminus\D_{e^{-1}})}}{\Vert\ov{u}\Vert_{H^1(\D_1)}}\leq\frac{c_1}{2}+C_1\sqrt{\alpha}\sum_{j=0}^{k-1}(1-\eta)^\frac{j}{2}
\end{equation}
We prove this by induction on $k$: this is verified at $k=0$, since we verified $W(u,\ov{r})\leq \frac{\Theta}{2}+\alpha\Vert \ov{u}\Vert_{H^1(\D_1)}^2$ and
$$\frac{\Vert u_{\ov{r}}-\ov{u}\Vert_{H^1(\D_1\setminus\D_{e^{-1}})}}{\Vert\ov{u}\Vert_{H^1(\D_1)}}\leq \ov{r}^{-3}c_2\leq\frac{1}{2}c_1.$$
Suppose now that \eqref{eq_inductive} is verified for some $k\in\N$. By our definition of $\alpha$, we have
$$\frac{c_1}{2}+C_1\sqrt{\alpha}\sum_{j=0}^{k-1}(1-\eta)^{\frac{j}{2}}\leq \frac{c_1}{2}+\frac{C_1\sqrt{\alpha}}{1-\sqrt{1-\eta}}=c_1,$$
so Theorem \ref{th_epiperimetry} applies to $u_{e^{-k}\ov{r}}$. This gives the two conclusions of \eqref{eq_inductive} for $k+1$:
\begin{align*}
W\left(u,e^{-k-1}\ov{r}\right)-\frac{\Theta}{2}&\leq (1- \eta)\left(W\left(u,e^{-k}\ov{r}\right)-\frac{\Theta}{2}\right)\\
&\leq  (1-\eta)^{k+1}\alpha\Vert\ov{u}\Vert_{H^1(\D_1)}^2
\end{align*}
and
\begin{align*}
\frac{\Vert u_{e^{-k-1}\ov{r}}-\ov{u}\Vert_{H^1(\D_1\setminus\D_{e^{-1}})}}{\Vert \ov{u}\Vert_{H^1(\D_1)}}&\leq\frac{ \Vert u_{e^{-k}\ov{r}}-\ov{u}\Vert_{H^1(\D_1\setminus\D_{e^{-1}})}+ \Vert u_{e^{-k-1}\ov{r}}-u_{e^{-k}\ov{r}}\Vert_{H^1(\D_1\setminus\D_{e^{-1}})}}{\Vert \ov{u}\Vert_{H^1(\D_1)}}\\
&\leq \frac{c_1}{2}+C_1\sqrt{\alpha}\sum_{j=0}^{k-1}(1-\eta)^\frac{j}{2}+ C_1\sqrt{W\left(u,e^{-k}\ov{r}\right)-\frac{\Theta}{2}}&\text{ by }\eqref{eq_inductive}\\
&\leq \frac{c_1}{2}+C_1\sqrt{\alpha}\sum_{j=0}^{k}(1-\eta)^\frac{j}{2}&\text{ by }\eqref{eq_inductive}.
\end{align*}
As a consequence, we have for every $k\in\N$:
$$W\left(u,e^{-k}\ov{r}\right)\leq \frac{\Theta}{2}+(1-\eta)^k\alpha\Vert \ov{u}\Vert_{H^1(\D_1)}^2,\ \Vert u_{e^{-k}\ov{r}}-\ov{u}\Vert_{H^1(\D_1\setminus\D_{e^{-1}})}\leq c_1\Vert \ov{u}\Vert_{H^1(\D_1)},$$
so using again Theorem \ref{th_epiperimetry}:
\begin{align*}
\Vert u_{e^{-k}\ov{r}}-u_{e^{-(k-1)}\ov{r}}\Vert_{H^1(\D_1\setminus\D_{e^{-1}})}&\leq C_1\sqrt{\alpha}(1-\eta)^{\frac{k}{2}}\Vert \ov{u}\Vert_{H^1(\D_1)}
\end{align*}
and
\begin{align*}
\Vert u_{e^{-k}\ov{r}}-u_{e^{-(k-1)}\ov{r}}\Vert_{H^1(\D_1)}&=\sqrt{\sum_{p\in\N}\Vert u_{e^{-k}\ov{r}}-u_{e^{-(k-1)}\ov{r}}\Vert_{H^1(\D_{e^{-p}}\setminus\D_{e^{-p-1}})}^2}\\
&\leq\sqrt{\sum_{p\in\N}e^{-4p}\Vert u_{e^{-(k+p)}\ov{r}}-u_{e^{-(k+p-1)}\ov{r}}\Vert_{H^1(\D_{1}\setminus\D_{e^{-1}})}^2}&\text{ by \eqref{eq_scaling_property}}\\
&\leq \sqrt{\sum_{p\in\N}e^{-4p}C_1^2\alpha (1-\eta)^{k+p}}\Vert \ov{u}\Vert_{H^1(\D_1)}\\
&\lesssim (1-\eta)^{\frac{k}{2}}\Vert \ov{u}\Vert_{H^1(\D_1)},
\end{align*}
The quantity $\Vert u_{e^{-k}\ov{r}}-u_{e^{-(k-1)}\ov{r}}\Vert_{H^1(\D_1)}$
is summable in $k$: as a consequence there exists $\widehat{u}\in\M_{\mathrm{hom}}$ a limit of $u_{e^{-k}\ov{r}}$ such that for any $k\geq 0$:
$$\Vert u_{e^{-k}\ov{r}}-\widehat{u}\Vert_{H^1(\D_1)}\leq \sum_{p\geq k}\Vert u_{e^{p}\ov{r}}-u_{e^{-p-1}\ov{r}}\Vert_{H^1(\D_1)}\lesssim \sum_{p\geq k}(1-\eta)^{\frac{p}{2}}\Vert \ov{u}\Vert_{H^1(\D_1)}\lesssim (1-\eta)^\frac{k}{2}\Vert \ov{u}\Vert_{H^1(\D_1)}.$$
Denoting $\nu=-\log\sqrt{1-\eta}$, we obtain, for any $r\in [0,1]$:
\begin{equation}\label{eq_est_u_r_hat_u}
\Vert u_r-\widehat{u}\Vert_{H^1(\D_1)}\lesssim r^{\nu} \Vert \ov{u}\Vert_{H^1(\D_1)}.
\end{equation}
Consider now any $r\in (0,1]$, and let $s\in (0,1)$ to be fixed. Then
\begin{align*}
\Vert u_r-\widehat{u}\Vert_{H^1(\D_1)}&\leq \Vert u_r-\ov{u}\Vert_{H^1(\D_1)}+\Vert \ov{u}-u_s\Vert_{H^1(\D_1)}+\Vert u_s-\widehat{u}\Vert_{H^1(\D_1)}\\
&\lesssim \left(\frac{1}{r^3}+\frac{1}{s^3}\right)\Vert u-\ov{u}\Vert_{H^1(\D_1)}+s^\nu\Vert \ov{u}\Vert_{H^1(\D_1)}.
\end{align*}
We optimize this quantity in $s$ by fixing $s=\left(\frac{\Vert u-\ov{u}\Vert_{H^1(\D_1)}}{\Vert \ov{u}\Vert_{H^1(\D_1)}}\right)^\frac{1}{3+\nu}$ (which is smaller than $1$ since $c_1\leq 1$), so
\begin{align*}
\Vert u_r-\widehat{u}\Vert_{H^1(\D_1)}&\lesssim \frac{\Vert u-\ov{u}\Vert_{H^1(\D_1)}}{r^3}+\Vert u-\ov{u}\Vert_{H^1(\D_1)}^{\frac{\nu}{3+\nu}}\Vert \ov{u}\Vert_{H^1(\D_1)}^\frac{3}{3+\nu}.
\end{align*}
We gather this estimate and \eqref{eq_est_u_r_hat_u}:
\begin{align*}
\Vert u_r-\widehat{u}\Vert_{H^1(\D_1)}&\lesssim\min\left(r^\nu\Vert \ov{u}\Vert_{H^1(\D_1)},\frac{\Vert u-\ov{u}\Vert_{H^1(\D_1)}}{r^3}+\Vert u-\ov{u}\Vert_{H^1(\D_1)}^{\frac{\nu}{3+\nu}}\Vert \ov{u}\Vert_{H^1(\D_1)}^\frac{3}{3+\nu}\right)\\
&\lesssim \min\left(r^\nu\Vert \ov{u}\Vert_{H^1(\D_1)},\Vert \ov{u}\Vert_{H^1(\D_1)}^{\frac{3}{3+\nu}}\Vert u-\ov{u}\Vert_{H^1(\D_1)}^{\frac{\nu}{3+\nu}}\right)\\
&\lesssim \min\left(r,\frac{\Vert u-\ov{u}\Vert_{H^1(\D_1)}}{\Vert \ov{u}\Vert_{H^1(\D_1)}}\right)^\gamma \Vert \ov{u}\Vert_{H^1(\D_1)},
\end{align*}
with $\gamma=\frac{\nu}{3+\nu}$. Evaluating this at $r=1$ we get
\begin{align*}
\Vert \ov{u}-\widehat{u}\Vert_{H^1(\D_1)}&\leq \Vert \ov{u}-u_r\Vert_{H^1(\D_1)}+\Vert u_r-\widehat{u}\Vert_{H^1(\D_1)}\\
&\lesssim \Vert u-\ov{u}\Vert_{H^1(\D_1)}+\Vert u-\ov{u}\Vert_{H^1(\D_1)}^\gamma\Vert \ov{u}\Vert_{H^1(\D_1)}^{1-\gamma}\\
&\lesssim \left(c_2+c_2^\gamma\right)\Vert \ov{u}\Vert_{H^1(\D_1)},
\end{align*}
so for a sufficiently small constant $c_2$ we obtain that $\widehat{u}\in\M_{\mathrm{hom}}^\Theta$ (by Lemma \ref{lem_separation_blowup}: the $H^1(\D_1)$ distance between $\M_{\mathrm{hom}}^\pi,\M_{\mathrm{hom}}^{t_1},\M_{\mathrm{hom}}^{2\pi}$ is positive). This concludes the proof.
\end{proof}

As a corollary, we obtain Theorem \ref{mr_typeIII_IV} and \ref{mr_blowup}.

\begin{proof}[Proof of Theorem \ref{mr_typeIII_IV}]
This is a direct consequence of the previous Proposition \ref{prop_convpolyblow_up} for $\Theta=2\pi$.
\end{proof}

\begin{proof}[Proof of Theorem \ref{mr_blowup}]
Let $u\in\M(\D_1)$ such that $0\in\Spt(u)$. By Proposition \ref{prop_convpolyblow_up}, it is sufficient to prove the following claims:
\begin{itemize}
\item[(i) ]In the case $\liminf_{r\to 0}\frac{|\Spt(u)\cap\D_{p,r}|}{\pi r^2}<1$, then there exists a sequence $r_n\to 0$ and $\ov{u}\in\B^{\pi}_{\mathrm{hom}}\cup \B^{t_1}_{\mathrm{hom}}$ such that $u_{r_n}$ converges to $\ov{u}$. Indeed this implies $W(u,0)\geq \frac{\Theta}{2}$ (with $\Theta=2|\Spt(\ov{u})\cap\D_1|$) and for some sufficiently large $n$ we have
$$\frac{\Vert u_{r_n}-\ov{u}\Vert_{H^1(\D_1)}}{\Vert \ov{u}\Vert_{H^1(\D_1)}}\leq c_2,$$
where $c_2$ is the constant from Proposition \ref{prop_convpolyblow_up}, and this proposition implies the result.
\item[(ii) ]In the case $W(u,0)\geq \pi$, then there exists $\ov{u}\in\mathrm{Span}(x^2+y^2,x^2-y^2,xy)$ and $r\in (0,1)$ such that $$\Vert u_r-\ov{u}\Vert_{H^1(\D_1)}\leq, c_2\Vert \ov{u}\Vert_{H^1(\D_1)}.$$
Indeed, by Proposition \ref{prop_convpolyblow_up}, this implies the result.
\end{itemize}
The first claim is an immediate consequence of Proposition \ref{prop_existence_blowup}.\\
Assume now we are in the second case. If some sequence $u_{r_n}$ is bounded in $H^1(\D_1)$ as $r_n\to 0$, then by Lemma \ref{lem_conv_blowup}, $u_{r_n}$ converges to some $2$-homogeneous limit $v$. Since $W(v,1)\geq \pi$, necessarily $v$ is of type III or IV, and we obtain the result.

Assume that $\Vert u_{r}\Vert_{H^1(\D_1)}\underset{r\to 0}{\longrightarrow}+\infty$ instead. By Proposition \ref{prop_renormalized_blowup}, there exists some extracted sequence $r_n\to 0$ and some nonzero $2$-homogeneous biharmonic function $v$ such that
$$\frac{u_{r_n}}{\Vert u_{r_n}\Vert_{H^1(\D_1)}}\underset{n\to +\infty}{\longrightarrow} v\text{ in }H^2_\lin(\R^2).$$
In particular, for some large enough $n$ we have
$$\frac{\left\Vert u_{r_n}-\Vert u_{r_n}\Vert_{H^1(\D_1)}v\right\Vert_{H^1(\D_1)}}{\Vert u_{r_n}\Vert_{H^1(\D_1)}}=\left\Vert \frac{u_{r_n}}{\Vert u_{r_n}\Vert_{H^1(\D_1)}}-v\right\Vert_{H^1(\D_1)}\leq c_2,$$
which concludes the proof.
\end{proof}

\subsection{$\mathcal{C}^{1,\alpha}$ regularity of the boundary}\label{subsec_C1alphaest}

We are now ready to prove a weaker form (with only $\mathcal{C}^{1,\gamma}$ bound) of Theorem \ref{mr_typeI}. We will use the following interpolation inequalities: for any smooth bounded open set $D\subset\R$ or $\R^2$, $0<\alpha<\beta<1$, there exists $C=C(D,\alpha,\beta)\geq 0$ such that for any $f\in\mathcal{C}^{1,\beta}(D,\R)$:
\begin{align}
\Vert f\Vert_{\mathcal{C}^{1}(D)}&\leq C \Vert f\Vert_{L^\infty(D)}^{\frac{\beta}{1+\beta}}\Vert f\Vert_{\mathcal{C}^{1,\beta}(D)}^{\frac{1}{1+\beta}}\label{eq_HolderInterp_C1Linfty},\\
\Vert f\Vert_{\mathcal{C}^{1,\alpha}(D)}&\leq C \Vert f\Vert_{\mathcal{C}^1(D)}^{\frac{\beta-\alpha}{\beta}}\Vert f\Vert_{\mathcal{C}^{1,\beta}(D)}^{\frac{\alpha}{\beta}}\label{eq_HolderInterp_C1aC1},\\
\Vert f\Vert_{\mathcal{C}^{1,\alpha}(D)}&\leq C \Vert f\Vert_{L^\infty(D)}^{\frac{\beta-\alpha}{\beta+1}}\Vert f\Vert_{\mathcal{C}^{1,\beta}(D)}^{\frac{\alpha+1}{\beta+1}}\label{eq_HolderInterp_C1aLinfty}.
\end{align}

\begin{proposition}\label{prop_C1a}
There exist $c_3>0$, $\alpha,\kappa\in (0,1)$, with the following property: for any $u\in\M(\D_1)$ such that
\[\Vert u-u^{\mathrm{I}}\Vert_{H^1(\D_1)}\leq c_3,\]
there exists a function $h\in\mathcal{C}^{1,\alpha}\left(\left[-\frac{1}{2},\frac{1}{2}\right],\R\right)$ such that $\Vert h\Vert_{L^{\infty}\left(\left[-\frac{1}{2},\frac{1}{2}\right]\right)}\leq \sqrt[6]{\Vert u-u^{\mathrm{I}}\Vert_{H^1(\D_1)}}$ and
$$\Vert h\Vert_{\mathcal{C}^{1,\alpha}\left(\left[-\frac{1}{2},\frac{1}{2}\right]\right)}\lesssim \Vert u-u^\mathrm{I}\Vert_{H^1(\D_1)}^\kappa$$
and
$$\left\{(x,y)\in\D_{\frac{3}{4}}:|x|\leq \frac{1}{2},\ u(x,y)\neq 0\right\}=\left\{(x,y)\in\D_{\frac{3}{4}}:|x|\leq \frac{1}{2},\ y>h(x)\right\}.$$
Moreover, $\Delta u\in\mathcal{C}_\loc^{1,\alpha}\left(\D_{\frac{1}{2}}\cap\ov{\Spt(u)}\right)$ and $\Delta u=1$ in $\D_{\frac{1}{2}}\cap\partial\Spt(u)$.
\end{proposition}

\begin{proof}
In this proof, when $u_{p,r}$ has a unique limit as $r\to 0$, we write its limit $u_{p,0}$. We let $\eps>0$ be such that $\Vert u-u^\mathrm{I}\Vert_{H^1(\D_1)}\leq \eps$, and $\eps$ will be supposed arbitrarily small. We define the cone
$$\mathcal{C}=\{(x,y)\in\R^2:y<-|x|/2\}.$$
The angular opening of the cone $\mathcal{C}$ is approximately $0.7\pi$, which is strictly larger than $2\pi-t_1\approx 0.57\pi$. We let
$$G=\left\{p\in\partial\Spt(u)\cap\left[-\frac{1}{2},\frac{1}{2}\right]^2:\ (p+\mathcal{C})\cap\D_{\frac{3}{4}}\subset\{u=0\}\right\}.$$
Note that the square $\left[-\frac{1}{2},\frac{1}{2}\right]^2$ is fully contained in the open disk $\D_{\frac{3}{4}}$.

We now prove a sequence of claims on $G$ and $\partial\Spt(u)$, to prove that $G$ is exactly $\partial\Spt(u)\cap\left[-\frac{1}{2},\frac{1}{2}\right]^2$ and that it is a $\mathcal{C}^{1,\alpha}$ graph (over the $x$ coordinate) for some $\alpha\in (0,1)$ to be defined.

\begin{itemize}
\item[(\textbf{Claim a})] For a sufficiently small $\eps$, which we will suppose in the rest of the proof, we have $$\D_{\frac{3}{4}}\cap\partial\Spt(u)\subset\{(x,y)\in\R^2:|y|\leq \sqrt[6]{\eps}\}.$$
Indeed, let $r=\sqrt[6]{\eps}$, $p=(x,y)\in \D_{\frac{3}{4}}$ such that $y\leq -r$, then $u^\mathrm{I}_{p,r}=0$ in $\D_1$ so
$$\Vert u_{p,r}\Vert_{H^1(\D_1)}=\Vert (u-u^\mathrm{I})_{p,r}\Vert_{H^1(\D_1)}\leq\frac{1}{r^3}\Vert u-u^\mathrm{I}\Vert_{H^1(\D_1)}\leq \frac{\eps}{r^3}=\sqrt{\eps}$$
so for a small enough $\eps$ by the non-degeneracy of Lemma \ref{lem_nondegeneracy} we have $u=0$ in $\D_{p,r/2}$: $p$ cannot be in the boundary of the support.

Assume now instead that $y\geq r$. Lemma \ref{lem_Cacciopoli} gives a bound $\Vert u\Vert_{H^2(\D_{\frac{3}{4}})}\lesssim 1$, so
\begin{align*}
\Vert u-u^\mathrm{I}\Vert_{\mathcal{C}^0(\D_{\frac{3}{4}})}&\lesssim \Vert u-u^\mathrm{I}\Vert_{W^{1,4}(\D_{\frac{3}{4}})}\lesssim \Vert u-u^\mathrm{I}\Vert_{H^{2}(\D_{\frac{3}{4}})}^\frac{1}{2} \Vert u-u^\mathrm{I}\Vert_{H^{1}(\D_{\frac{3}{4}})}^\frac{1}{2}\\ 
&\lesssim \Vert u-u^\mathrm{I}\Vert_{H^{1}(\D_{\frac{3}{4}})}^\frac{1}{2}\lesssim \sqrt{\eps}
\end{align*}
by Gagliardo-Nirenberg interpolation inequality. Thus,
$u(x,y)\geq \frac{1}{2}y^2 -C\sqrt{\eps}$
for some universal constant $C$. Since $y\geq \sqrt[6]{\eps}$, then $u(x,y)>0$ for a small enough $\eps$. This concludes claim (a).
\item[(\textbf{Claim b})] For any $p\in \D_{\frac{3}{4}}\cap\partial\Spt(u)$, we claim $\Vert u_{p,\frac{1}{8}}-u^\mathrm{I}\Vert_{H^1(\D_1)}\lesssim \sqrt[12]{\eps}$. Indeed, using Lemma \ref{lem_Cacciopoli}, \ref{lem_C1log}, we have a uniform bound
$$\Vert u\Vert_{\mathcal{C}^{1,\frac{1}{2}}(\D_{7/8})}\lesssim 1.$$
Let $p=(x,y)\in G$, by the claim (a) we have $|y|\leq \sqrt[6]{\eps}$, so
\begin{align*}
\Vert u_{p,\frac{1}{8}}-u^\mathrm{I}\Vert_{H^1(\D_1)}&\leq \Vert u_{(x,y),\frac{1}{8}}-u_{(x,0),\frac{1}{8}}\Vert_{H^1(\D_1)}+\Vert u_{(x,0),\frac{1}{8}}-u^\mathrm{I}\Vert_{H^1(\D_1)}\\
&\lesssim \sqrt{|y|}+\Vert \left(u-u^\mathrm{I}\right)_{(x,0),\frac{1}{8}}\Vert_{H^1(\D_1)}\text{ by the }\mathcal{C}^{1,\frac{1}{2}}(\D_{\frac{7}{8}})\text{ bound}\\
&\leq \sqrt[12]{\eps}+\eps\lesssim \sqrt[12]{\eps}
\end{align*}
We obtain claim (b).

\item[(\textbf{Claim c})] $G$ is a graph in the variable $x\in \left[-\frac{1}{2},\frac{1}{2}\right]$: by definition two distinct points of $G$ cannot share the same $x$ coordinate, so it is sufficient to prove that for any $x_0\in \left[-\frac{1}{2},\frac{1}{2}\right]$ there exists $y_0\in\R$ such that $(x_0,y_0)\in G$.

To prove this, consider the largest possible value of $y_0$ such that $((x_0,y_0)+\mathcal{C})\cap\D_{\frac{3}{4}}\subset\{u=0\}$. Since $y_0$ is maximal, there must exist some contact point $$q=(x,y)\in \partial\left[((x_0,y_0)+\mathcal{C})\cap\D_{\frac{3}{4}}\right]\cap \partial\Spt(u).$$
By the claim (a), necessarily $|y|\leq \sqrt[6]{\eps}$.

Suppose $q$ is \textbf{not} $(x_0,y_0)$. Then $\Spt(u)$ has Lebesgue density at most $\frac{1}{2}$ at $q$, so there exists a unique blow-up $u_{q,0}$ that is necessarily of the form
$$u_{q,0}=s_1 u^\mathrm{I} \circ \mathrm{rot}_{s_2\theta_0} $$
where $\theta_0=\arctan(1/2)$, $s_1,s_2\in\{-1,+1\}$. However by the claim (b) we have 
$$\Vert u_{q,\frac{1}{8}}-u^\mathrm{I}\Vert_{H^1(\D_1)}\lesssim \sqrt[12]{\eps},$$

so Proposition \ref{prop_convpolyblow_up} applied to $u_{q,\frac{1}{8}}$ gives
$$\Vert s_1 u^\mathrm{I} \circ \mathrm{rot}_{s_2\theta_0}-u^\mathrm{I}\Vert_{H^1(\D_1)}\lesssim \eps^\frac{\gamma}{12}.$$
Here $\gamma$ is the constant from Proposition \ref{prop_convpolyblow_up}.
This is a contradiction when $\eps$ is small enough: as a consequence $G$ is a graph, and it is automatically the graph of a $1$-Lipschitz function (since the slope of $\mathcal{C}$ is smaller than $1$), denoted 
$$g:\left[-\frac{1}{2},\frac{1}{2}\right]\to \left[-\sqrt[6]{\eps},\sqrt[6]{\eps}\right].$$

\item[(\textbf{Claim d})] For every $p\in G$, $u$ admits at $p$ a unique blow-up of type I, denoted $u_{p,0}$, such that for every $r\in \left(0,\frac{1}{8}\right)$ we have
$$\Vert u_{p,r}-u_{p,0}\Vert_{H^1(\D_1)}\lesssim \min\left(\eps^\frac{1}{12},r\right)^\gamma.$$
Indeed, the existence of a blow-up of type I is a consequence of Proposition \ref{prop_existence_blowup}, since the density of the support of $u$ at points of $G$ is strictly less than $\frac{t_1}{2\pi}$. The uniqueness and rate of convergence follows from applying Proposition \ref{prop_convpolyblow_up} to $u_{p,\frac{1}{8}}$, which verifies $\Vert u_{p,\frac{1}{8}}-u^\mathrm{I}\Vert_{H^1(\D_1)}\lesssim \sqrt[12]{\eps}$ by the claim (b).

\item[(\textbf{Claim e})] $\Vert g\Vert_{\mathcal{C}^{1,\beta}(\left[-\frac{1}{2},\frac{1}{2}\right])}\lesssim 1$, where $\beta=\frac{\gamma}{2\gamma+1}\in (0,1)$.

To prove this, consider $p,q$ two points in $G$, $d:=|p-q|$, and $r=d^\frac{1}{2\gamma+1}$. We suppose $d$ is small enough so that $r\leq \frac{1}{16}$ and $\frac{d}{r}\leq \frac{1}{4}$. We apply claim (d) above to $u$ at $p$ and $q$:
\begin{align*}
\Vert u_{p,2r}-u_{p,0}\Vert_{H^1(\D_1)}&\lesssim r^\gamma=d^\frac{\gamma}{2\gamma+1}\\
\Vert u_{q,r}-u_{q,0}\Vert_{H^1(\D_1)}&\lesssim r^\gamma=d^\frac{\gamma}{2\gamma+1}.
\end{align*}
By the Cacciopoli-type and higher regularity estimates from Lemma \ref{lem_Cacciopoli}, \ref{lem_C1log}, we have
$$\Vert u_{p,2r}\Vert_{\mathcal{C}^{1,\frac{1}{2}}(\D_{\frac{3}{4}})}\lesssim 1 $$
Then,
$$\Vert u_{p,r}-u_{q,r}\Vert_{H^1(\D_1)}=\left\Vert u_{p,r}-u_{p,r}\left(\cdot+\frac{q-p}{r}\right)\right\Vert_{H^1(\D_1)}\lesssim \Vert u_{p,r}\Vert_{\mathcal{C}^{1,\frac{1}{2}}(\D_{\frac{3}{2}})}\sqrt{\frac{|q-p|}{r}}\lesssim d^\frac{\gamma}{2\gamma+1}.$$

This implies $\Vert u_{q,0}-u_{p,0}\Vert_{H^1(\D_1)}\lesssim d^\frac{\gamma}{2\gamma+1}$, so  $g$ (the function defining the graph $G$) is differentiable at every point $x\in\left[-\frac{1}{2},\frac{1}{2}\right]$, with $|g'(x)-g'(y)|\lesssim (|x-y|+|g(x)-g(y)|)^{\beta}\lesssim |x-y|^\beta$.

\item[(\textbf{Claim f})] $\left\{(x,y)\in\D_{\frac{3}{4}}:|x|\leq \frac{1}{2}\right\}\cap\partial\Spt(u)\subset G$. Indeed, suppose that there exists some point 
$$q\in\left\{(x,y)\in\D_{\frac{3}{4}}:|x|\leq \frac{1}{2}\right\}\cap\partial\Spt(u)\setminus G.$$
Let $p$ be a projection of $q$ on $G$ ($p$ may not be unique). Since the $y$-coordinate of $p$ and $q$ are bounded by $\sqrt[6]{\eps}$, we have necessarily $|p-q|\lesssim \sqrt[6]{\eps}$. Let $r=4|p-q|$, then by application of claim (d) to $u$ at the point $p$ with radius $r$, we have
$$\Vert u_{p,r}-u^\mathrm{I}\Vert_{H^1(\D_1)} \lesssim \eps^\frac{\gamma}{12}.$$

Let $\left(\tilde{u},\tilde{q}\right)=\left(u_{p,r},\frac{q-p}{r}\right)$. Let $\tilde{G}$ be the graph associated to $\tilde{u}$ (defined in the same way that $G$ is defined for $u$), note that $0\in \tilde{G}$ (since $p\in G$).

All the above claims apply to $\tilde{u}$:  by claims (a,e) applied to $\tilde{u}$, we obtain that $\tilde{G}$ is the graph of a function $\tilde{g}\in\mathcal{C}^{1,\alpha}\left(\left[-\frac{1}{2},\frac{1}{2}\right]\right)$ such that $\Vert\tilde{g}\Vert_{L^\infty\left(\left[-\frac{1}{2},\frac{1}{2}\right]\right)}\lesssim \eps^\frac{\gamma}{72}$. By definition of $r$, we have $|\tilde{q}|=\frac{1}{4}$. By claim (a), we have $|\tilde{q}_y|\lesssim \eps^\frac{\gamma}{36}$. At the same time, the origin is a projection of $\tilde{q}$ on $\frac{G-p}{r}$: this is a contradiction for a small enough $\eps$, which proves the claim (f).
\end{itemize}
Thus, we have proved that $$\left\{(x,y)\in\D_{\frac{3}{4}}:|x|\leq \frac{1}{2}\right\}\cap\Spt(u)=\left\{(x,y)\in\D_{\frac{3}{4}}:|x|\leq \frac{1}{2},\ y>g(x)\right\}$$
for some function $g:\left[-\frac{1}{2},\frac{1}{2}\right]\to \R$ that verifies
$$\Vert g\Vert_{\mathcal{C}^{1,\beta}\left(\left[-\frac{1}{2},\frac{1}{2}\right]\right)}\lesssim 1,\ \Vert g\Vert_{L^{\infty}\left(\left[-\frac{1}{2},\frac{1}{2}\right]\right)}\leq \sqrt[6]{\eps}.$$
We define $\alpha=\frac{\beta}{2}$, $\kappa=\frac{\alpha}{12\alpha+4}$. By the interpolation properties \eqref{eq_HolderInterp_C1aLinfty}, this implies
$$\Vert h\Vert_{\mathcal{C}^{1,\alpha}\left(\left[-\frac{1}{2},\frac{1}{2}\right]\right)}\lesssim \eps^{\kappa}.$$
Now we prove the $\mathcal{C}^{1,\alpha}_\loc(\D_{\frac{1}{2}})$ regularity of $\Delta u$ on its support. Let $q\in \Spt(u)\cap\D_{\frac{1}{2}}$, and let $p$ be a projection of $q$ on $G$. We suppose $r:=2|p-q|$ is sufficiently small such that $r\leq \frac{1}{8}$. Then we may apply claim (d) to $u$ at the point $p$:
$$\Vert u_{p,r}-u_{p,0}\Vert_{H^1(\D_1)}\lesssim \eps^\frac{\gamma}{12}.$$
Moreover, since $p$ is a projection of $q$ on $G$, then for a small enough $\eps$ we have $\D_{\tilde{q},\frac{1}{2}}\subset\Spt(u_{p,0})$, where $\tilde{q}=\frac{q-p}{r}$. Let now $\chi\in\mathcal{C}^\infty_c(\D_{\frac{1}{2}},\R)$ a radial function with integral $1$ with $|\nabla^k \chi|\lesssim 1$ for $k=0,1$. Since $\Delta u$ is harmonic on its support, then
\begin{equation}\label{eq_mean_value}
\begin{split}
|\Delta u(q)-1|&=\left|\Delta u_{p,r}\left(\tilde{q}\right)-\Delta u_{p,0}\left(\tilde{q}\right)\right|=\left|\int_{\D_{\tilde{q},\frac{1}{2}}}\left(\Delta u_{p,r}-\Delta u_{p,0}\right)\chi(\cdot-\tilde{q})\right|\\
&=\left|\int_{\D_{\tilde{q},\frac{1}{2}}}\nabla(u_{p,r}-u_{p,0})\cdot\nabla\chi(\cdot-\tilde{q})\right|\lesssim \eps^\frac{\gamma}{12}
\end{split}
\end{equation}
Thus $\Delta u-1$ is harmonic on its support and continuous up to the boundary $G$, with a homogeneous Dirichlet condition. Since $G$ has a regularity $\mathcal{C}^{1,\alpha}$, then by classical boundary elliptic regularity we have $\Delta u\in\mathcal{C}^{1,\alpha}_\loc\left(\D_{\frac{1}{2}}\cap\ov{\Spt(u)}\right)$.
\end{proof}

\begin{proposition}\label{prop_II_C1a}
There exist $c_4,C_4>0$, $\alpha\in (0,1)$, with the following property: for any $u\in\M(\D_1)$ such that
\[\Vert u-u^{\mathrm{II}}\Vert_{H^1(\D_1)}\leq c_4,\]
there exist two functions $\theta^-,\theta^+\in\mathcal{C}^{1,\alpha}\left(\left[\frac{1}{4},\frac{3}{4}\right]\right)$ such that
\begin{align*}
\Vert \theta^-\Vert_{\mathcal{C}^{1,\alpha}\left(\left[\frac{1}{4},\frac{3}{4}\right]\right)}+\Vert \theta^+\Vert_{\mathcal{C}^{1,\alpha}\left(\left[\frac{1}{4},\frac{3}{4}\right]\right)}&\lesssim 1,\\
\Vert \theta^-\Vert_{L^{\infty}\left(\left[\frac{1}{4},\frac{3}{4}\right]\right)}+\Vert \theta^+\Vert_{L^{\infty}\left(\left[\frac{1}{4},\frac{3}{4}\right]\right)}&\lesssim \sqrt[6]{\Vert u-u^{\mathrm{II}}\Vert_{H^1(\D_1)}},
\end{align*}
and
$$\left\{re^{i\theta}\in\Spt(u)\cap\D_{\frac{3}{4}}\setminus \D_{\frac{1}{4}}\right\}=\left\{re^{i\theta}\in\D_{\frac{3}{4}}\setminus \D_{\frac{1}{4}}:\theta^-(r)\leq \theta\leq t_1+\theta^+(r)\right\}.$$
\end{proposition}

\begin{proof}
We denote
$$\eps=\Vert u-u^{\mathrm{II}}\Vert_{H^1(\D_1)},\quad S_0=\left[\frac{1}{4}e_1,\frac{3}{4}e_1\right],\quad S_{t_1}=\mathrm{rot}_{t_1}(S_0),$$
where $\mathrm{rot}_{t_1}$ is the rotation of angle $t_1$. We let $\eps>0$ that will be fixed sufficiently small later. We let also for any $\delta>0$:
$$S_0^\delta=\{p\in\D_{\frac{3}{4}}\setminus\D_{\frac{1}{4}}:\mathrm{dist}(p,S_0)\leq \delta,\ S_{t_1}^\delta=\mathrm{rot}_{t_1}(S_0^\delta).$$
Let $c_2,C_2,\gamma$ be the constants from Proposition \ref{prop_convpolyblow_up}, and $c_3,C_3,\alpha,\kappa$ be the constants from Proposition \ref{prop_C1a}.
Let $r\in\left[\frac{1}{4},\frac{3}{4}\right]$, $\theta\in \left[-\frac{\pi}{2},\frac{\pi}{2}\right]$, then
$$u^\mathrm{II}(re^{i\theta})-u^{\mathrm{I}}(re^{i\theta})=-\frac{r^2}{2t_1}\left(\theta-\frac{\sin(2\theta)}{2}\right).$$
Observe the last factor is of the form $\mathcal{O}_{\theta\to 0}(|\theta|^3)$, so for any $\rho\in\left[0,\frac{1}{4}\right]$ we get
$$\Vert u^{\mathrm{II}}_{re_1,\rho}-u^{\mathrm{I}}\Vert_{H^1(\D_1)}\lesssim \rho^3.$$
Thus,
\begin{equation}\label{eq_comparison_flat_angle}
\Vert u_{re_1,\rho}-u^{\mathrm{I}}\Vert_{H^1(\D_1)}\leq \Vert u_{re_1,\rho}-u^{\mathrm{II}}_{re_1,\rho}\Vert_{H^1(\D_1)}+\Vert u^{\mathrm{II}}_{re_1,\rho}-u^{\mathrm{I}}\Vert_{H^1(\D_1)} \lesssim \frac{\eps}{\rho^3}+\rho^3.
\end{equation}
We may fix a sufficiently small $\ov{\rho}\gtrsim 1$ such that, when $\eps$ is sufficiently small, then
$$\forall r\in\left[\frac{1}{4},\frac{3}{4}\right],\ \Vert u_{re_1,\ov{\rho}}-u^{\mathrm{I}}\Vert_{H^1(\D_1)}\leq c_3.$$
We apply Proposition \ref{prop_C1a} to every $r\in\left[\frac{1}{4},\frac{3}{4}\right]$. As a consequence, we get that
$$S_0^{\ov{\rho}/2}\cap\partial\Spt(u)=\left\{te^{i\theta^-(t)},\ t\in\left[\frac{1}{4},\frac{3}{4}\right]\right\}$$
with $\Vert\theta^-\Vert_{\mathcal{C}^{1,\alpha}\left(\left[\frac{1}{4},\frac{3}{4}\right]\right)}\lesssim 1$. Moreover, using again \eqref{eq_comparison_flat_angle} with $\rho:=\sqrt[6]{\eps}$ and Proposition \ref{prop_C1a}, we have
$$\Vert\theta^-\Vert_{L^\infty\left(\left[\frac{1}{4},\frac{3}{4}\right]\right)}\lesssim \sqrt[6]{\eps}.$$
We define similarly $\theta^+$ from the rotation $u\circ\mathrm{rot}_{t_1}$, and we obtain
$$S_{t_1}^{\ov{\rho}/2}\cap\partial\Spt(u)=\left\{te^{i(t_1+\theta^+(t))},\ t\in\left[\frac{1}{4},\frac{3}{4}\right]\right\}$$
with
$$\Vert\theta^+\Vert_{\mathcal{C}^{1,\alpha}\left(\left[\frac{1}{4},\frac{3}{4}\right]\right)}\lesssim 1,\ \Vert\theta^+\Vert_{L^\infty\left(\left[\frac{1}{4},\frac{3}{4}\right]\right)}\lesssim \sqrt[6]{\eps}.$$
By the interpolation inequality \eqref{eq_interp_C1H1}, and Lemmas \ref{lem_Cacciopoli}, \ref{lem_BMO}:
\begin{align*}
\Vert u-u^{\mathrm{II}}\Vert_{\mathcal{C}^1(\D_{\frac{3}{4}})}&\lesssim \Vert u-u^{\mathrm{II}}\Vert_{H^1(\D_{\frac{3}{4}})}^\frac{1}{2}\Vert u-u^{\mathrm{II}}\Vert_{W^{2,4}(\D_{\frac{3}{4}})}^\frac{1}{2}\\
&\lesssim \Vert u-u^{\mathrm{II}}\Vert_{H^1(\D_{1})}^\frac{1}{2}\left(\Vert u\Vert_{H^1(\D_1)}+\Vert u^{\mathrm{II}}\Vert_{H^{1}(\D_{1})}\right)^\frac{1}{2}\\
&\lesssim \eps^\frac{1}{2}.
\end{align*}
In particular, for a sufficiently small $\eps$, we obtain $$\partial\Spt(u)\cap \left(\D_{\frac{3}{4}}\setminus\D_{\frac{1}{4}}\right)\subset S_0^{\ov{\rho}/2}\cup S_{t_1}^{\ov{\rho}/2},$$
and this proves the result.
\end{proof}
We are now ready to prove Theorem \ref{mr_typeII}.

\begin{proof}[Proof of Theorem \ref{mr_typeII}]
Let $\eps=\mathrm{dist}_{H^1(\D_1)}\left(u,\M_{\mathrm{hom}}^{t_1}\right)$, which we will suppose to be arbitrarily small. 

We let $c_2,C_2>0$, $\gamma\in (0,1)$ be the constants from Proposition \ref{prop_convpolyblow_up}, and $c_4,C_4>0$, $\alpha\in (0,1)$ be the constants from Proposition \ref{prop_II_C1a}. By Proposition \ref{prop_convpolyblow_up} there exists $\widehat{u}\in \mathscr{M}_{\mathrm{hom}}^{t_1}$ such that 
$$\Vert u_r-\widehat{u}\Vert_{H^1(\D_1)}\lesssim \min(r,\eps)^\gamma$$
for any $r\in (0,1)$. Up to a rotation and a change of sign, we may assume without loss of generality that $\widehat{u}=u^{\mathrm{II}}$. By Proposition \ref{prop_II_C1a} applied to $u_r$ for every $r\in (0,1)$, there exists two functions $\theta^-,\theta^+\in\mathcal{C}^{1,\alpha}\left(\left(0,\frac{3}{4}\right],\R\right)$ such that
\begin{equation}\label{eq_esttheta}
\begin{split}
\Vert \theta^-(r\cdot)\Vert_{\mathcal{C}^{1,\alpha}\left(\left[\frac{1}{4},\frac{3}{4}\right]\right)}&+\Vert \theta^+(r\cdot)\Vert_{\mathcal{C}^{1,\alpha}\left(\left[\frac{1}{4},\frac{3}{4}\right]\right)}\lesssim 1,\\
\Vert \theta^-(r\cdot)\Vert_{L^{\infty}\left(\left[\frac{1}{4},\frac{3}{4}\right]\right)}&+\Vert \theta^+(r\cdot)\Vert_{L^{\infty}\left(\left[\frac{1}{4},\frac{3}{4}\right]\right)}\lesssim \min\left(r,\eps\right)^\frac{\gamma}{6},
\end{split}
\end{equation}
for every $r\in (0,1)$ and
$$\left\{re^{i\theta}\in\Spt(u)\cap\D_{\frac{3}{4}}\setminus\{0\}\right\}=\left\{re^{i\theta}\in\D_{\frac{3}{4}}:r>0,\ \theta^-(r)\leq \theta\leq t_1+\theta^+(r)\right\}.$$
Functions $\theta_\pm$ extends to $r=0$ with the value $0$. We now define the diffeomorphism (seeing it as a complex-valued function)
$$\Phi\left(re^{i\theta}\right)=re^{i\left(\theta+\theta^-(r)+\frac{\theta^+(r)-\theta^-(r)}{t_1}\theta\right)},$$
so that
$$\Phi\left(\D_{\frac{3}{4}}\cap\Spt(u^{\mathrm{II}})\right)=\D_{\frac{3}{4}}\cap\Spt(u).$$
We now verify that $\Phi-\mathrm{id}$ is small in a suitable Hölder space:
\begin{itemize}[label=\textbullet]
\item By \eqref{eq_esttheta}, we have $\Vert\Phi-\mathrm{id}\Vert_{L^\infty(\D_1)}\lesssim \eps^\frac{\gamma}{6}$. 
\item For any $x\in\D_{\frac{1}{2}}\setminus\{0\}$, then by \eqref{eq_esttheta} applied to $r:=2|x|$ we have
$$|D\Phi(x)-I_2|\lesssim|x|^{\frac{\gamma }{6}\frac{\alpha}{\alpha+1}}$$
We used here the interpolation inequality \eqref{eq_HolderInterp_C1Linfty}.
\item For any distinct $x,y\in\D_{\frac{1}{2}}$, suppose that $|x|\leq |y|\leq  3|x|$ i.e. there exists $r\in (0,1)$ such that $\frac{1}{4}r\leq |x|\leq |y|\leq \frac{3}{4}r$. We let
$$\beta:=\frac{\frac{\gamma}{6}\alpha}{\alpha+1+\frac{\gamma}{6}}.$$
Then by \eqref{eq_esttheta} applied with the radius $r$ we have
\begin{align*}
|D\Phi(x)-D\Phi(y)|&\lesssim \left(\Vert \theta_+(r\cdot)\Vert_{\mathcal{C}^{1,\beta}(\left[\frac{1}{4},\frac{3}{4}\right])}+\Vert \theta_-(r\cdot)\Vert_{\mathcal{C}^{1,\beta}(\left[\frac{1}{4},\frac{3}{4}\right])}\right)\left|\frac{x-y}{r}\right|^\beta\\
&\lesssim r^{\frac{\gamma}{6}\frac{\alpha-\beta}{\alpha+1}} \left|\frac{x-y}{r}\right|^\beta&\text{ by  \eqref{eq_HolderInterp_C1aLinfty}, \eqref{eq_esttheta}}\\
&=|x-y|^\beta&\text{ by definition of }\beta.
\end{align*}
\end{itemize}
The second and third point imply $\Vert \Phi\Vert_{\mathcal{C}^{1,\beta}(\D_{\frac{1}{2}})}\lesssim 1$. Since $\Vert \Phi-\mathrm{id}\Vert_{L^\infty(\D_{\frac{1}{2}})}\lesssim \eps^\frac{\gamma}{6}$, then we may apply interpolation inequality \eqref{eq_HolderInterp_C1aLinfty}; we let
$$\nu=\frac{\beta}{2},\ \mu=\frac{\beta}{\beta+1}\frac{\gamma}{12},$$
then
$$\Vert\Phi-\mathrm{id}\Vert_{\mathcal{C}^{1,\nu}(\D_{\frac{1}{2}})}\lesssim \eps^\mu.$$
\end{proof}

\subsection{Higher regularity of the boundary}\label{subsec_higherreg}

In the classical Alt-Caffarelli problem, higher order regularity of the boundary is usually obtained by some form of partial hodograph transform (i.e. seeing the state function $u$ as a coordinate, and finding a nonlinear, overdetermined elliptic equation verified in this coordinate system), based on \cite[Th. 2]{KN77}. This was generalized for elliptic systems in \cite{KNS78}, and for higher order systems in \cite{KNS79}: in particular our boundary condition appears verbatim in \cite[Th. 4.2]{KNS79} with $n=m=2$, $F(p,u,\nabla u,\hdots,\nabla^4 u)=\Delta^2u$, $g(x,M)=\mathrm{Tr}(M)-1$. However, in \cite[Th. 4.2]{KNS79} $u$ is assumed to be of class $\mathcal{C}^4$ up to the free boundary, so we cannot apply it directly. A similar issue for the Alt-Caffarelli problem (regarding the results from \cite{KNS78}) was pointed out in the appendix of \cite{KL18}.

For this reason, we give an independent proof of higher regularity based on a conformal transform of the boundary. We find that the Stokes system associated to the biharmonic equation is conformally transported to an overdetermined elliptic equation, where the additional boundary condition (from the optimality condition) allows a bootstrap procedure. A similar method was used for a second order problem in \cite{L11}.

The key argument is that after a conformal change of variable to flatten the boundary, the optimality conditions imply that the pressure (which is harmonic) verifies a homogeneous Neumann condition, so it is automatically smooth up to the boundary. From there, we treat the system as a complex-valued linear elliptic equation.

\begin{lemma}\label{lem_higherreg}
Let $\Om\subset\D_1$ an open set such that $\Gamma:=\D_1\cap\partial\Om$ is a $\mathcal{C}^{1,\alpha}$ curve (for some $\alpha\in (0,1)$) that connects two points of $\partial\D_1$. Let $u\in H^2(\Om)$ such that $\Delta u\in\mathcal{C}^0(\Om\cup\Gamma)$ and
$$
\begin{cases}
\Delta^2u = 0 &(\Om)\\
u=|\nabla u|=0 & (\Gamma)\\
\Delta u = c & (\Gamma)
\end{cases}
$$
for some nonzero constant $c$. Then $\D_1\cap\partial\Om$ is an analytic curve.
\end{lemma}
\begin{proof}
We identify $\R^2$ with $\C$ (so vector fields are complex-valued functions) and write the Wirtinger derivative
$$\partial_{z}=\frac{1}{2}\left(\partial_x-i\partial_y\right),\ \partial_{\ov{z}}=\frac{1}{2}\left(\partial_x+i\partial_y\right).$$
Since $\D_1\cap\partial\Om$ is a $\mathcal{C}^{1,\alpha}$ curve and $\Delta u-c$ is a harmonic function that verify a (homogeneous) dirichlet condition on $\D_1\cap\partial\Om$, then $\Delta u\in \mathcal{C}^{1,\alpha}_{\loc}(\Om\cup\Gamma)$.
Let $v:\Om\to\C$ and $p:\Om\to \R$ two functions defined by
$$v=i\partial_{\ov{z}} u,\ \partial_{\ov{z}}p=i\partial_{\ov{z}}\partial_{z}\partial_{\ov{z}}u$$
i.e. $p$ is a harmonic conjugate (uniquely defined up to an additive constant) of the harmonic function $\frac{1}{4}\Delta u$, that belongs to $\mathcal{C}^{1,\alpha}_\loc(\Om\cup \Gamma)$. $(v,p)$ verifies the Stokes equation with Dirichlet (or no-slip) boundary condition

$$
\begin{cases}
\partial_z\partial_{\ov{z}}v=\partial_{\ov{z}}p&(\Om)\\
v= 0 &(\Gamma).
\end{cases}
$$
Moreover, the condition $\Delta u|_{\Gamma}=c$ gives an additional boundary condition on $p$ and $v$: since $p$ is the harmonic conjugate of $\Delta u$, which is constant in $\Gamma$, then $p$ verifies a homogeneous Neumann boundary condition in $\Gamma$, and since $\partial_{z}v=\frac{i}{4}\Delta u$ we have:
$$
\partial_{z}v=\frac{1}{4}ic \text{ in }\Gamma.
$$

Let $\Phi:\D_1\to \Om$ be a bijective conformal map (that is uniquely defined up to some conformal automorphism of the disk): by the Carathéodory Theorem \cite{Car13} (see \cite[Th. 2.6]{P92} for a modern reference), $\Phi$ extends as a homeomorphism $\ov{\D_1}\to \ov{\Om}$, and we let $\Gamma'$ be the reciprocal image of $\Gamma$ through $\Phi$. By the Kellogg-Warschawski heorem \cite{W32} (see also \cite[Th. 3.6]{P92}, or \cite[Prop. 3.4]{P92} for a local version), we have $\Phi\in\mathcal{C}^{1,\alpha}_\loc(\D_1\cup\Gamma')$.

Let 
$$\ovh{p}=p\circ\Phi,\ \ovh{v}=v\circ\Phi$$
Then $\ovh{p}$ is harmonic and verifies a homogeneous Neumann boundary condition on $\Gamma'$ (since this property is transported through the conformal mapping). Thus $\ovh{p}$ is analytic in $\D_1\cup\Gamma'$. 

Then we compute the equation verified by $(\ovh{v},\ovh{p})$ in $\D_1$:
$$\partial_{z}\ovh{v}=\Phi'(\partial_{z} v)\circ\Phi,\ \partial_{\ov{z}}\ovh{p}=\ov{\Phi'}(\partial_{\ov{z}}p)\circ\Phi,$$
so
$$\partial_{z}\partial_{\ov{z}}\ovh{v}=|\Phi'|^2(\partial_{z}\partial_{\ov{z}} v)\circ\Phi=|\Phi'|^2(\partial_{\ov{z}} p)\circ\Phi=\Phi'\partial_{\ov{z}}\ovh{p}.$$
We obtain that $\ovh{v}$ verifies the overdetermined elliptic equation
$$
\begin{cases}
\partial_{\ov{z}}\partial_z \ovh{v}=\Phi'\partial_{\ov{z}}\ovh{p}& (\D_1)\\
\ovh{v}= 0 & (\Gamma')\\
\partial_{z}\ovh{v}=\frac{1}{4}ic\Phi' & (\Gamma').
\end{cases}
$$
Since $\partial_{\ov{z}}\ovh{p}$ is analytic up to the boundary $\Gamma'$, and $\Phi'\in\mathcal{C}^{0,\alpha}_\loc(\D_1\cup\Gamma')$, by standard Schauder theory (using the first two equations only) we have $\ovh{v}\in\mathcal{C}^{2,\alpha}_\loc(\D_1\cup\Gamma')$. Using the third line, we see $\Phi'\in\mathcal{C}^{1,\alpha}_\loc(\Gamma')$, so $\Phi\in\mathcal{C}^{2,\alpha}_\loc(\D_1\cup\Gamma')$.

Iterating this, for any $k\in\N_{>0}$, if $\Phi\in\mathcal{C}^{k,\alpha}_\loc(\D_1\cup\Gamma')$, then $\widehat{v}\in \mathcal{C}^{k+1,\alpha}_\loc(\D_1\cup\Gamma')$,  so $\Phi'|_{\Gamma'}\in \mathcal{C}^{k,\alpha}_\loc(\Gamma')$, which implies $\Phi\in \mathcal{C}^{k+1,\alpha}_\loc(\D_1\cup\Gamma')$.\\
Thus, we obtain $\Phi\in\mathcal{C}^{\infty}(\D_1\cup \Gamma')$. From here on, the analyticity of $\Phi$ up to the boundary $\Gamma'$ may be obtained through tracking constants in this bootstrapping procedure, or simply by application of \cite[Th. 4.2]{KNS79} with $n=m=2$, $F(p,u,\nabla u,\hdots,\nabla^4 u)=\Delta^2u$, $g(x,M)=\mathrm{Tr}(M)-c$ to the function $u$.
\end{proof}

From this we obtain the main Theorem \ref{mr_typeI}.

\begin{proof}[Proof of Theorem \ref{mr_typeI}]
This is a direct combination of Proposition \ref{prop_C1a}, and Lemma \ref{lem_higherreg}.
\end{proof}
\subsection{Proof of the main result}\label{subsec_proof_mr_total}
We may now conclude with the proof of Theorem \ref{mr_total}.

\begin{proof}[Proof of Theorem \ref{mr_total}]
Let $u\in\M(D)$ for some open set $D\subset\R^2$. We define

\begin{itemize}[label=\textbullet]
\item $\mathcal{R}_u$ the points of $D\cap\partial\Spt(u)$, with density $\frac{1}{2}$ in $\Spt(u)$.
\item $\mathcal{A}_u$ the points of $D\cap\partial\Spt(u)$, with density $\frac{t_1}{2\pi}$ in $\Spt(u)$.
\item $\mathcal{N}_u$ the points $p\in D\cap\partial\Spt(u)$ such that for some sufficiently small $r>0$, we have $|\D_{p,r}\setminus\Spt(u)|=0$.
\item $\mathcal{J}_u$ is the set of points $p\in D\cap\partial\Spt(u)\setminus (\mathcal{R}_u\cup\mathcal{A}_u\cup\mathcal{N}_u)$ such that $\liminf_{r\to 0}\Vert u_{p,r}\Vert_{H^1(\D_1)}<\infty$.
\item $\mathcal{E}_u:=D\cap\partial\Spt(u)\setminus (\mathcal{R}_u\cup\mathcal{A}_u\cup\mathcal{N}_u\cup\mathcal{J}_u)$
\end{itemize}
These sets are a partition of $D\cap\partial\Spt(u)$ by construction. Now, we prove the claim of Theorem \ref{mr_total} case-by-case.
\begin{itemize}[label=\textbullet]
\item Let $p\in\mathcal{R}_u$: by Theorem \ref{mr_blowup}, $u_{p,r}$ converges in $H^1(\D_1)$ as $r\to 0$ to some homogeneous solution of type I. By Theorem \ref{mr_typeI} applied to $u_{p,2r}$ for a sufficiently small $r$, we obtain that $\Spt(u)\cap \D_{p,r}$ is, after rotation, the epigraph of some analytic function. In particular, every point of $\partial\Spt(u)\cap\D_{p,r}$ is also regular.
\item Let $p\in\mathcal{A}_u$. By the same reasoning, by Theorem \ref{mr_blowup} and \ref{mr_typeII}, there exists a sufficiently small $r>0$ such that $\Spt(u)\cap\D_{p,r}$ is the image of the angular cone $\Spt(u^{\mathrm{II}})\cap\D_1$ through a $\mathcal{C}^{1,\gamma}$ diffeomorphism (sending $0$ to $p$, $\partial\D_1$ to $\partial\D_{p,r}$). As a consequence, every point of $\partial\Spt(u)\cap\D_{p,r}\setminus\{p\}$ belongs to $\mathcal{R}_u$, $\mathcal{A}_u$ is a discrete set where two connected components of $\mathcal{R}_u$ join with an angle $t_1$ and $\Delta u|_{\mathcal{R}_u}$ takes opposite signs on each.
\item Let $p\in\mathcal{N}_u$. Then $u$ is biharmonic in a neighbourhood of $p$. Since biharmonic functions are analytic, this implies that for a sufficiently small $r>0$, $\partial\Spt(u)\cap\D_{p,r}$ is an intersection of zeros of analytic function. Since $u(p)=|\nabla u(p)|=0$, the blow-up $u_{p,0}$ is well-defined and non-zero by the nondegeneracy Lemma \ref{lem_nondegeneracy}, and it is a blow-up either of type III (in which case $|\Delta u(p)|\geq 1$) or of type IV (in which case $\partial\Spt(u)\cap\D_{p,r}=\{p\}$ for a sufficiently small $r$).
\item Let $p\in\mathcal{J}_u$. By Lemma \ref{lem_conv_blowup} and the fact that $p\notin \mathcal{R}_u\cup\mathcal{A}_u$, $u$ admits a blow-up of type III or IV, denoted $\widehat{u}$. By Theorem \ref{mr_typeIII_IV}, we have $\Vert u_{p,r}-\widehat{u}\Vert_{H^1(\D_1)}\to 0$ as $r$ tends to $0$. If $\widehat{u}$ was of type IV, then $p$ would be an isolated point of $\partial\Spt(u)$ and in particular it would be in the nodal set $\mathcal{N}_u$: thus $\widehat{u}$ is necessarily of type III.
\item Let $p\in \mathcal{E}_u$. Since $p$ belongs to none of the sets above we know that $p$ neither has density $\frac{1}{2}$ or $\frac{t_1}{2\pi}$ in $\Spt(u)$ - meaning by Proposition \ref{prop_existence_blowup} that it has density $1$ in $\Spt(u)$ - nor is it isolated in $\partial\Spt(u)$ - meaning it does not admit a blow-up of type IV -  nor does it admit a blow-up of type III since it would otherwise belong to $\mathcal{N}_u\cup\mathcal{J}_u$. Since $u$ admits no blow-up at $p$, necessarily we have
$$
\Vert u_{p,r}\Vert_{H^1(\D_{p,r})}\underset{r\to 0}{\longrightarrow}+\infty.$$
By Theorem \ref{mr_blowup}, we have necessarily $W(u(p+\cdot),0)<\pi$, since otherwise there would be a blow-up of type III or IV.
Finally, since $\Vert u_{p,r}\Vert_{H^1(\D_1)}\to +\infty$, then for any subsequence $r_n\to 0$ we may apply Proposition \ref{prop_renormalized_blowup} to the sequence $u_{p,r_n}$: there is some subsequence $(n_i)$ such that $\frac{u_{p,r_{n_i}}}{\Vert u_{p,r_{n_i}}\Vert_{H^1(u_{p,{r_{n_i}}})}}$ converges in $H^2_\loc(\R^2)$ norm to some nonzero $2$-homogeneous biharmonic function.
\end{itemize}
\end{proof}

\section{Quasi-minimizers}\label{sec_quasireg}

We explain here how the method we developed for the energy $E(\cdot,D)$ extends to a more general setting, with applications to the buckling eigenvalue minimization \eqref{expr_Lambda1}. We do not attempt to find the weakest optimality condition that would still imply the main regularity results: for the sake of clarity we instead give some mild optimality condition that contains the buckling problem.

\subsection{Regularity of quasi-minimizers}
\begin{definition}\label{def_quasimin}
Let $D\subset\R^2$ be an open set, we say $u$ is a quasiminimiser of $E(\cdot,D)$, and we write
$$u\in\QM(D),$$
if for any $v\in H^2(D)$ such that $\{u\neq v\}\Subset D$, we have
\begin{equation}\label{eq_quasiminimality}
E(u;D)\leq E(v;D)+ \Vert v-u\Vert_{L^2(D)}.
\end{equation}

More generally, for $\mu\geq 0$ we say $u$ is a $\mu$-quasiminimizer of $E(\cdot;D)$, we write $u\in\QM^{\mu}(D)$ if for any $v\in H^2(D)$ such that $\{u\neq v\}\Subset D$ we have
\begin{equation}\label{eq_weaker_quasiminimality}
E(u;D)\leq E(v;D)+ \mu\Vert v-u\Vert_{L^2(D)}.
\end{equation}
\end{definition}
This notion of quasi-minimizer scales as follows.
\begin{lemma}\label{lem_rescale_general_minimizers}
Let $D\subset\R^2$ be an open set, $\mu\geq 0$, $r>0$, and $u\in \QM^\mu(D)$. Then $u_r\in\QM^{r\mu}(D/r)$.
\end{lemma}
In other words, it is sufficient to prove the minimality condition \eqref{eq_weaker_quasiminimality} for some (possibly large) constant $\mu>0$ to know that some rescaling of $u$ is a quasiminimizer in the sense of \eqref{eq_quasiminimality}.
\begin{proof}
Consider $v\in H^2(D/r)$ such that $\{v\neq u_r\}\Subset D/r$. This means that $v=w_r$, for some $w\in H^2(D)$ such that $\{w\neq u\}\Subset D$.
By the minimality condition \eqref{eq_weaker_quasiminimality} we have
$$E(u;D)\leq E(w;D)+\mu\Vert v-u\Vert_{L^2(D)}.$$
After a change of variable, we obtain
$$E(u_r;D/r)\leq E(v;D/r)+r\mu\Vert v-u_r\Vert_{L^2(D/r)}.$$
So $u_r\in\QM^{r\mu}(D)$.
\end{proof}

We now explain the different changes to adapt the proofs to quasi-minimizers, following the sections of the rest of the paper.
\subsubsection{Cacciopoli inequality} Extending Lemma \ref{lem_Cacciopoli} to elements of $u\in \QM(\D_1)$ is a direct consequence of the quasiminimality condition \eqref{eq_quasiminimality} with $v=e^{-t\eta^4}u_r$ for any $r\in(0,1)$, $\eta\in\mathcal{C}^{\infty}_c(\D_1,[0,1])$, $t\in\R$. Indeed, this implies
$$E(u_r;\D_1)-E(e^{t\eta^4}u_r;\D_1)\leq r\Vert (e^{t\eta^4}-1)u_r\Vert_{L^2(\D_1)},$$
so for $t\to 0^+$ we get
$$\int_{\D_1}2\nabla^2 u_r:\nabla^2(\eta^4 u_r)\leq r\Vert \eta^4u_r\Vert_{L^2(\D_1)}.$$
We bound 
\begin{align*}
r\Vert \eta^4 u_r\Vert_{L^2(\D_1)}&\leq r\Vert \eta^2 u_r\Vert_{L^2(\D_1)}\quad\text{ since }\eta\leq 1\\
&\lesssim r\Vert\nabla^2(\eta^2 u_r)\Vert_{L^2(\D_1)}\quad\text{by Poincaré inequality}\\
&\lesssim r^2+\Vert \nabla(\eta^2)\otimes\nabla u_r\Vert_{L^2(\D_1)}^2+\Vert \nabla^2(\eta^2) u_r\Vert_{L^2(\D_1)}^2+\Vert \eta^2\nabla^2 u_r\Vert_{L^2(\D_1)}^2.
\end{align*}
The rest of the proof follows with the same computation as in Lemma \ref{lem_Cacciopoli}. We obtain
\begin{lemma}\label{lem_quasi_cacciopoli}
Let $u\in\QM(\D_1)$, $r\in (0,1)$, $\rho\in \left[\frac{1}{2},1\right)$, then
\begin{equation}\label{eq_quasi_cacciopoli}
\Vert u_r\Vert_{H^2(\D_\rho)}\lesssim \frac{r^2+\Vert u_r\Vert_{H^1(\D_1\setminus\D_\rho)}}{(1-\rho)^2}.
\end{equation}
\end{lemma}

\subsubsection{BMO estimate}
The BMO estimate is proved similarly to Lemma \ref{lem_BMO}: we compare $u$ with its biharmonic replacement on any disk, and deduce a $BMO$ bound on $\Delta u$.

\begin{lemma}\label{lem_quasi_BMO}
Let $u\in\QM(\D_1)$. Then we have
\begin{equation}\label{eq_quasi_C112}
[\nabla^2 u]_{\mathrm{BMO}(\D_{\frac{1}{2}})}\lesssim 1+\Vert u\Vert_{H^2(\D_1)}.
\end{equation}
\end{lemma}
\begin{proof}
For any $\D_{p,r}\Subset D$, we let $u_{\D_{p,r}}$ be the biharmonic extension of $u$ in $\D_{p,r}$. Then the quasiminimality of $u$ gives
$$\int_{\D_{p,r}}|\Delta (u-u_{\D_{p,r}})|^2\leq |\D_{p,r}|+\Vert u-u_{\D_{p,r}}\Vert_{L^2(\D_{p,r})}.$$
We bound $\Vert u-u_{\D_{p,r}}\Vert_{L^2(\D_{p,r})}\lesssim r^2\Vert\Delta( u-u_{\D_{p,r}})\Vert_{L^2(\D_{p,r})}$ by Poincaré inequality, so by Young inequality:
$$\int_{\D_{p,r}}|\Delta (u-u_{\D_{p,r}})|^2\lesssim |\D_{p,r}|+r^4.$$
We divide by the measure of $\D_{p,r}$:
$$\fint_{\D_{p,r}}|\Delta (u-u_{\D_{p,r}})|^2\lesssim 1+r^2.$$
The rest of the proof is the same as the one of Lemma \ref{lem_BMO}.
\end{proof}
\subsubsection{Nondegeneracy lemma}
The extension of Lemma \ref{lem_nondegeneracy} to the quasi-minimal context has no major changes.
\begin{lemma}\label{lem_nondeg_quasi}
There exists $\eps>0$ such that the following holds: let $u\in\QM(\D_1)$, such that $\Spt(u)\cap\D_{\frac{1}{2}}\neq\emptyset$, then
\begin{equation}\label{eq_quasi_nondeg}
\Vert u\Vert_{H^1(\D_1)}\geq \eps.
\end{equation}
\end{lemma}
\begin{proof}
It is sufficient to prove that when $\Vert u\Vert_{H^1(\D_1)}$ is sufficiently small, then $u(0)=|\nabla u(0)|=0$ (since we may apply this to $u_{p,\frac{1}{2}}$ for $p\in\D_{\frac{1}{2}}$).
Let $r\in (0,1)$, we have the first (non-linear) estimate:
\begin{align*}
\Vert u_r\Vert_{H^1(\D_{\frac{1}{2}})}^2&\lesssim |\Spt(u_r)\cap\D_{\frac{3}{4}}|^\frac{1}{2}\Vert u_r\Vert_{W^{1,4}(\D_{\frac{3}{4}})}^2&\text{ by Hölder inequality}\\
&\lesssim |\Spt(u_r)\cap\D_{\frac{3}{4}}|^\frac{1}{2}\Vert u_r\Vert_{H^2(\D_{\frac{3}{4}})}^2&\text{ by Sobolev embedding }H^2\to W^{1,4}\\
&\lesssim |\Spt(u_r)\cap\D_{1}|^\frac{1}{2}\left(r^2+\Vert u_r\Vert_{H^1(\D_1)}^2\right)&\text{ by Lemma \ref{lem_quasi_cacciopoli}}.
\end{align*}
Then we have an estimate of the area of the support, obtained by the minimality condition applied to $v=(1-\eta)u_r$, where $\eta\in\mathcal{C}^\infty_c(\D_{\frac{3}{4}})$ verifies $\eta=1$ in $\D_{\frac{1}{2}}$:
\begin{align*}
|\Spt(u_r)\cap\D_{\frac{1}{2}}|&\leq \int_{\D_{\frac{1}{2}}}(\chi_{u_r\neq 0}-\chi_{v\neq 0})\\
&\leq \int_{\D_{\frac{3}{4}}}\left(|\Delta(u_r-\eta u_r)|^2-|\Delta u_r|^2\right)+ r \Vert\eta u_r\Vert_{L^2(\D_1)}\text{ by \eqref{eq_quasiminimality}}\\
&\lesssim \int_{\D_{\frac{3}{4}}}\left(|\Delta(\eta u_r)|^2-2\Delta(\eta u_r)\Delta u_r\right)+ r \Vert\eta u_r\Vert_{L^2(\D_1)}\\
&\lesssim \Vert u_r\Vert_{H^1(\D_{1})}^2+r^2\text{ by Lemma \ref{lem_quasi_cacciopoli}}.
\end{align*}
Denoting $g(r)=\Vert u_r\Vert_{H^1(\D_1)}^2$, we obtain that for some constant $C_0>0$:
$$\forall r>0:\D_{p,r}\subset D,\text{ we have }g\left(\frac{r}{4}\right)\leq C_0\left(r^3+g(r)^\frac{3}{2}\right).$$
Let $a_0=(4C_0)^{-2}$, $r_0:=\min\left(1,C_0^{-\frac{1}{2}},\frac{3}{4}a_0\right)$, then for any $r\in (0,r_0)$, assume $g(r)\leq a_0$, then we have
$$g\left(\frac{r}{4}\right)\leq r+\frac{1}{4}g(r)(\leq a_0)$$
and by induction $g\left(4^{-k}r\right)\leq k4^{1-k}r+4^{-k}g(r)$ for any $k\in\N^*$. This tends to $0$ as $k\to +\infty$.\\
Assume $\Vert u\Vert_{H^1(\D_1)}\leq r_0^3a_0^\frac{1}{2}$, then $\Vert u_{r_0}\Vert_{H^1(\D_1)}\leq a_0^\frac{1}{2}$, so $u(0)=|\nabla u(0)|=0$. This concludes the proof.
\end{proof}

From the previous results we can deduce a weaker (but easier to work with) optimality condition.

\begin{lemma}\label{lem_quasi_weaker_minimality}
There exists $Q>0$ such that for any $u\in\QM(\D_1)$ such that $0\in\partial\Spt(u)$, for any $r\in \left(0,\frac{1}{2}\right)$, $v\in u+H^2_0(\D_r)$, we have
$$E(u;\D_r)\leq E(v;\D_r)+Q\Vert u\Vert_{H^1(\D_1)}r^{\frac{5}{2}}.$$
\end{lemma}
\begin{proof}
We assume without loss of generality that $E(v;\D_r)\leq E(u;\D_r)$ (otherwise there is nothing to prove). By the non-degeneracy lemma we have 
\begin{equation}\label{eq_lem_quasi_weaker_minimality_1}
\Vert u\Vert_{H^1(\D_1)}\gtrsim 1.
\end{equation}
By combining Lemma \ref{lem_quasi_BMO}, \ref{lem_quasi_cacciopoli}, we have for any $p\in\D_{\frac{3}{2}r}$:
$$\frac{|u(p)|}{|p|}+|\nabla u(p)|\lesssim r^\frac{1}{2}[\nabla u]_{\mathcal{C}^{0,\frac{1}{2}}(\D_{\frac{3}{4}})}\lesssim r^\frac{1}{2}\Vert u\Vert_{H^1(\D_1)}.$$
So 
\begin{equation}\label{eq_lem_quasi_weaker_minimality_2}
\Vert u_r\Vert_{H^2(\D_1)}\lesssim \Vert u_r\Vert_{H^1(\D_{\frac{3}{2}})}\lesssim r^{-\frac{1}{2}}.
\end{equation}
Then, by the quasi-minimality condition on $u$ we have 
$$E(u;\D_r)\leq E(v;\D_r)+\Vert u-v\Vert_{L^2(\D_r)}$$
and we bound the remainder term:
\begin{align*}
\Vert u-v\Vert_{L^2(\D_r)}&\lesssim r^2\Vert \Delta(u-v)\Vert_{L^2(\D_r)}\text{ by Poincaré inequality}\\
&\lesssim r^2\left(\Vert \Delta u\Vert_{L^2(\D_r)}+\Vert \Delta v\Vert_{L^2(\D_r)}\right)\\
&\lesssim r^2E(u;\D_r)^{\frac{1}{2}}\text{ since we assume }E(v;\D_r)\leq E(u;\D_r)\\
&\lesssim r^2\left(r^2+\Vert \Delta u\Vert_{L^2(\D_r)}^2\right)^\frac{1}{2}\\
&=  r^3\left(1+\Vert \Delta u_r\Vert_{L^2(\D_1)}^2\right)^\frac{1}{2}\\
&\lesssim r^{\frac{5}{2}}\left(1+\Vert u\Vert_{H^1(\D_1)}^2\right)^\frac{1}{2}\text{ by }\eqref{eq_lem_quasi_weaker_minimality_2}\\
&\lesssim r^{\frac{5}{2}}\Vert u\Vert_{H^1(\D_1)}\text{ by }\eqref{eq_lem_quasi_weaker_minimality_1}.
\end{align*}
And so
\begin{align*}
E(u;\D_r)-E(v;\D_r)&\leq\Vert u-v\Vert_{L^2(\D_r)}\lesssim r^{\frac{5}{2}}\Vert u\Vert_{H^1(\D_1)}.
\end{align*}
This concludes the proof.
\end{proof}

\subsubsection{Sequence of minimizers}

For the next lemma, we write (for $\lambda,\mu\geq 0$) $u\in\QM_\lambda^{\mu}(D)$ if for every $\D_{p,r}\Subset D$, $v\in u+H^2_0(\D_{p,r})$, we have
$$E_\lambda(u;\D_{p,r})\leq E_\lambda(v;\D_{p,r})+\mu\Vert u-v\Vert_{L^2(\D_{p,r})}.$$
\begin{lemma}\label{lem_sequence_quasimin}
Let $\lambda^{(n)}\in \R_+$ some sequence that converges in $\mathcal{C}^0(\D_1)$ to $\lambda\geq 0$, $\mu^{(n)}$ a sequence of $[0,1]$ that converge to $\mu\geq 0$. For each $n$, let $u^{(n)}\in\QM_{\lambda^{(n)}}^{\mu^{(n)}}(\D_1)$ such that
$$\limsup_{n\to +\infty}\Vert u^{(n)}\Vert_{H^1(\D_1)}<+\infty.$$
Then there is a subsequence $u^{(n_k)}$ that converges in $H^2_\loc(\D_1)$ to $u\in\QM_\lambda^{\mu}(\D_1)$, with a strong $L^1(\D_1)$ convergence of $\lambda^{(n_k)}\chi_{u^{(n_k)}\neq 0}$ to $\lambda\chi_{u\neq 0}$.
\end{lemma}
\begin{proof}
By Lemma \ref{lem_quasi_cacciopoli}, $(u^{(n)})_{n\in\N}$ is bounded in $H^2_\loc(\D_1)$. Consider $u$ a $H^2_\loc(\D_1)$-weak limit of $(u^{(n)})_{n\in\N}$ (after extraction, which we do not write here for clarity), $v\in H^2(\D_1)$ such that $\{u\neq v\}\Subset\D_1$, $\eta\in\mathcal{C}^\infty_c(\D_1)$ such that $\eta=1$ on $\Spt(u-v)$, then we apply the quasi-optimality condition to $(1-\eta) u^{(n)}+\eta v$:
\begin{equation}\label{eq_quasiminimality_sequence}
E_{\lambda^{(n)}}(u^{(n)};\D_1)\leq E_{\lambda^{(n)}}((1-\eta) u^{(n)}+\eta v;\D_1)+\mu^{(n)}\Vert \eta (v-u^{(n)})\Vert_{L^2(\D_1)}.
\end{equation}
Following the proof of Lemma \ref{lem_sequence}, with $v=u$ we obtain 
$$\Vert u\Vert_{H^2(\D_r)}=\lim_{n\to +\infty}\Vert u^{(n)}\Vert_{H^2(\D_r)}$$
for any $r\in (0,1)$, thus the convergence $u^{(n)}\to u$ is strong in $H^2_\loc(\D_1)$. Then Lemma \ref{lem_sequence_quasimin} for a general $v$ gives as $n\to +\infty$:
$$E_\lambda(u;\D_1)\leq E_\lambda(v;\D_1)+\mu\Vert v-u\Vert_{L^2(\D_1)},$$
so $u\in\QM_{\lambda}^{\mu}(\D_1)$.
\end{proof}

\subsubsection{Monotonicity}

In this section, we consider a minimizer $u\in\QM(\D_1)$ with $0\in\partial\Spt(u)$. We prove that up to adding a sufficiently large polynomial term to $W(u,r)$, we still obtain an increasing function.

\begin{theorem}\label{Th_Monotonicity_quasi}
There exists $\lambda_0>0$ such that the following holds. For any $u\in\QM(\D_1)$ such that $0\in\partial\Spt(u)$, we let
$$v(t,\theta)=e^{2t}u(e^{-t+i\theta})$$
Then the function
$$r\in \left(0,\frac{1}{2}\right)\mapsto W(u,r)+\lambda_0\Vert u\Vert_{H^1(\D_1)}\sqrt{r}$$
is non-decreasing. More precisely, for any $\tau\geq 0$, we have $W(u,e^{-\tau})=\W(v,\tau)$ and
$$\frac{d}{d\tau}\left(\W(v,\tau)+\lambda_0\Vert u\Vert_{H^1(\D_1)}e^{-\frac{\tau}{2}}\right)\leq -4\Vert\partial_t\nabla_{t,\theta}v\Vert_{L^2(\partial\Cy_\tau)}^2.$$
\end{theorem}
Since $W(u,r)+\lambda_0\Vert u\Vert_{H^1(\D_1)}\sqrt{r}$ is nondecreasing, this means that $W(u,r)$ admits a limit - that we denote $W(u,0)$  - as $r\to 0$.

\begin{proof}
By Lemma \ref{lem_quasi_BMO}, we have $v\in H^2_\lin(\Cy_0)$. By a change of variable (see Lemma \ref{lem_change_variable_diskexp}), the quasi-minimality of $u$ implies the following: for any $\tau\geq 0$, any $w\in H^2_\lin(\Cy_0)$ such that $\{v\neq w\}\subset\Cy_\tau$, we have
$$\G(v,\tau)\leq\G(w,\tau)+e^{2\tau}\sqrt{\int_{\Cy_\tau}e^{-6t}(v-w)(t,\theta)^2dtd\theta}.$$
The minimality condition simplifies to
\begin{equation}\label{eq_quasiminimality_exp}
\G(v,\tau)\leq\G(w,\tau)+ e^{-\tau}\Vert v-w\Vert_{L^2(\Cy_\tau,e^{-2t}dtd\theta)}.
\end{equation}
As in the proof of Theorem \ref{Th_Monotonicity_exp}, let $\tau>0$, consider the vector field $\phi(t,\theta)=(f(t),0)$ for some function $f\in\mathcal{C}^\infty(\R_+)$ with $f=0$ on $[0,\tau]$, $\Vert f\Vert_{W^{1,\infty}(\R_+)}<\infty$,
$\psi_\eps$ the inverse of $t\mapsto t+\eps f(t)$ and $w(t,\theta)=v(\psi_\eps(t),\theta)$. Then the optimality condition \eqref{eq_quasiminimality_exp} applied for $\eps\to 0$ gives the condition 
\begin{align*}
\int_{0}^{+\infty}e^{-2t}(2f-f')&\left\{\Vert\partial_{t,t}v\Vert^2+2\Vert\partial_{t,\theta}v\Vert^2+\Vert\partial_{\theta,\theta}v\Vert^2-4\Vert\partial_{\theta}v\Vert^2+\int_{\partial\Cy_t}\chi_{v\neq 0}\right\}dt\\
&+2\int_{0}^{+\infty}e^{-2t}f''\langle\partial_t v,\partial_{t,t}v\rangle dt\\
&\leq -4\int_{0}^{+\infty}e^{-2t}f'\left\{\Vert \partial_{t,t}v\Vert^2+\Vert \partial_{t,\theta}v\Vert^2\right\}dt+ e^{-\tau}\Vert f\partial_t v\Vert_{L^2(\Cy_\tau,e^{-2t}dtd\theta)}.
\end{align*}
By Lemma \ref{lem_quasi_BMO}, we bound $|\partial_ t v|\lesssim e^{\frac{t}{2}} \Vert u\Vert_{H^1(\D_1)}$. So taking formally $f\to \chi_{t\geq\tau}$ (or using the weak derivative characterization as in the proof of Theorem \ref{Th_Monotonicity_exp}) we obtain
$$\W'(v,\tau)+4\left(\Vert\partial_{t,t}v(\tau)\Vert^2+\Vert\partial_{t,\theta}v(\tau)\Vert^2\right)\lesssim \Vert u\Vert_{H^1(\D_1)}e^{-\frac{\tau}{2}}.$$
So for some sufficiently large $\lambda_0\gtrsim 1$, we have
$$\frac{d}{d\tau}\left(\W(v,\tau)+\lambda_0\Vert u\Vert_{H^1(\D_1)} e^{-\frac{\tau}{2}}\right)\leq -4\left(\Vert\partial_{t,t}v(\tau)\Vert^2+\Vert\partial_{t,\theta}v(\tau)\Vert^2\right).$$
The rest follows exactly the proof of Theorem \ref{Th_Monotonicity_disk} with the additional corrective term.
\end{proof}
We obtain as a consequence the following blow-up results, that is a combination of Lemma \ref{lem_conv_blowup} and Propositions \ref{prop_existence_blowup}, \ref{prop_renormalized_blowup}.
\begin{lemma}\label{lem_existence_blowup_quasi}
Let $u\in\QM(\D_1)$ such that $0\in\partial\Spt(u)$, then
\begin{itemize}[label=\textbullet]
\item Any converging subsequence of $u_{r}$ as $r\to 0$ converges to a $2$-homogeneous minimizer of $E(\cdot;\D_1)$.
\item If $\liminf_{r\to 0}\frac{|\Spt(u)\cap \D_{r}|}{|\D_{r}|}<1$, then there exists some sequence $r_n\to 0$ such that $u_{r_n}$ converges to an element of $\M_{\mathrm{hom}}^{\pi}\sqcup \M_{\mathrm{hom}}^{t_1}$.
\item If $\limsup_{r\to 0}\Vert u_{r}\Vert_{H^1(\D_1)}=+\infty$ and $0\in\partial\Spt(u)$, then there exists a sequence $r_n\to 0$ such that $\frac{u_{r_n}}{\Vert u_{r_n}\Vert_{H^1(\D_1)}}$ converges in $H^2_\loc(\R^2)$ to a non-zero $2$-homogeneous biharmonic function.
\end{itemize}
\end{lemma}
\begin{proof}
\begin{itemize}[label=\textbullet]
\item This is a direct consequence of Lemma  \ref{lem_sequence_quasimin} and Theorem \ref{Th_Monotonicity_quasi}.
\item With the same computation as for Lemma \ref{lem_boundedgrowth}, we have for any $u\in\QM(\D_1)$, $0<s<r<1$:
\begin{equation}\label{eq_quasi_growth_estimate}
N(u,s)\geq \left(\frac{s}{r}\right)^\kappa N(u,r)-\frac{1-\left(\frac{s}{r}\right)^{\kappa}}{\kappa}\left(W(u,1)+\lambda_0\Vert u\Vert_{H^1(\D_1)}\right)
\end{equation}
which implies by contradiction that for any sequence $r_n\to 0$ such that $\limsup_{n}|\Spt(u_{r_n})\cap\D_1|<\pi$, we have
$$\limsup_{n\to +\infty}\Vert u_{r_n}\Vert_{H^1(\D_1)}<\infty.$$
Indeed, otherwise $u_{r_n}/\Vert u_{r_n}\Vert_{H^1(\D_1)}$ would converge to a non-zero biharmonic function $v$, such that $\{v=0\}$ has a positive measure, which is a contradiction.

Since $u_{r_n}$ is bounded in $H^1(\D_1)$ and belongs to $\QM^{r_n}(\D_1)$ (by Lemma \ref{lem_rescale_general_minimizers}), then up to extracting some subsequence we may suppose (by Lemma \ref{lem_sequence_quasimin}) that $u_{r_n}$ converges to  some $v\in\QM^0(\D_1)=\M(\D_1)$, which is non-zero by the nondegeneracy Lemma \ref{lem_nondeg_quasi}.

Since $W(u,r)$ converges to a constant as $r\to 0$, then $W(v,r)$ is constant with respect to $r$, so $v$ is a $2$-homogeneous element of $\M(\D_1)$. Since the density of its support is less than $1$, necessarily
$$v\in\M_{\mathrm{hom}}^{\pi}\sqcup\M_{\mathrm{hom}}^{t_1}.$$
\item This follows the same proof as Lemma \ref{prop_renormalized_blowup}. Let $\lambda_n=\Vert u_{r_n}\Vert_{H^1(\D_1)}^{-1}$, the renormalized function $\lambda_n u_{r_n}$ belongs to $\M_{\lambda_n^2 }^{\lambda_n}(\D_1)$: by Lemma \ref{lem_sequence_quasimin} we may suppose (after extraction) that $\Vert u_{r_n}\Vert_{H^1(\D_1)}^{-1}u_{r_n}$ converges to $v\in\QM_0^0(\D_1)$, meaning some biharmonic function $v$. Function $v$ is non-zero by the growth estimate \eqref{eq_quasi_growth_estimate} and we conclude as in Proposition \ref{prop_renormalized_blowup}.
\end{itemize}
\end{proof}
\subsubsection{Epiperimetry}
The central cornerstone of the epiperimetry formula is the estimate of Corollary \ref{cor_epiest}: we do not need to re-prove this result, indeed we can just use the competitor of \ref{cor_epiest} in the minimality characterization of Lemma \ref{lem_quasi_weaker_minimality}.

\begin{theorem}\label{th_epi_quasi}
Let $\Theta\in\{\pi,t_1,2\pi\}$. There exist constants $c_1,C_1,\lambda_1>0$, $\eta\in (0,1)$ such that the following holds: let $u\in\QM(\D_1)$ such that $0\in\partial\Spt(u)$, $
\ov{u}\in \M_{\mathrm{hom}}^{\Theta}$ if $\Theta\in\{\pi,t_1\}$ (resp. $\ov{u}\in\mathrm{Span}(x^2+y^2,x^2-y^2,xy)$ if $\Theta=2\pi$), $r\in \left(0,\frac{1}{2}\right)$, assume
$$\left\Vert u_r-\ov{u}\right\Vert_{H^1(\D_1\setminus\D_{e^{-1}})}\leq c_1\Vert \ov{u}\Vert_{H^1(\D_1)},\quad W(u,0)\geq \frac{\Theta}{2}.$$
Then, denoting 
\begin{equation}\label{eq_def_WM}
W^{quasi}(u,r):=W(u,r)+\lambda_1\Vert u\Vert_{H^1(\D_1)}\sqrt{r},
\end{equation}
we have
$$
W^{quasi}(u,e^{-1}r)-\frac{\Theta}{2}\leq (1-\eta)\left( W^{quasi}(u,r)-\frac{\Theta}{2}\right)$$
and
$$
\Vert u_r-u_{e^{-1}r}\Vert_{H^1(\D_1\setminus\D_{e^{-1}})}\leq C_1\sqrt{W^{quasi}(u,r)-\frac{\Theta}{2}}.
$$
\end{theorem}
\begin{proof}
We let $v(t,\theta)=e^{2t}u(e^{-t+i\theta})$ and proceed as in the proof of Theorem \ref{Th_Monotonicity_exp}. We fix $\tau=-\log(r)$. When the ratio
$$\frac{\left\Vert u_r-\ov{u}\right\Vert_{H^1(\D_1\setminus\D_{e^{-1}})}}{\Vert \ov{u}\Vert_{H^1(\D_1)}}$$
is sufficiently small, we obtain (by the same argument as in the proof of Theorem \ref{th_epiperimetry_exp}) that $(v(t,\cdot),\partial_t v(t,\cdot))$ verifies the support hypothesis of \ref{cor_epiest} for almost every $t\in\left[\tau+\frac{1}{4},\tau+\frac{3}{4}\right]$. 

We let $\phi\in\mathcal{C}^\infty_c(\R,[0,1])$ such that 
$$\{\phi\neq 0\}\subset\left[\tau+\frac{1}{4},\tau+\frac{3}{4}\right],\quad |\phi'|\lesssim 1,\quad \phi=1\text{ in }\left[\tau+\frac{1}{3},\tau+\frac{2}{3}\right].$$
Then following the computations of the proof of Theorem \ref{th_epiperimetry_exp}, we have 
\begin{align*}
\W(v,\tau)-\W(v,\tau+1)\geq&\int_{\tau}^{\tau+1}\phi(t)^2\left(\Vert \partial_{t,t}v\Vert^2+2\Vert \partial_{t,\theta}v\Vert^2+2\G(v(t,\cdot),t)-2\G(v,t)\right)dt\\
&-4\int_{\tau}^{\tau+1}\phi(t)\left(\phi(t)+\phi'(t)\right)\langle \partial_{t}v,\partial_{t,t}v\rangle dt.
\end{align*}
By the minimality condition of Lemma \ref{lem_quasi_weaker_minimality}, applied to the competitor built in Corollary \ref{cor_epiest} with the initial conditions $(v(t,\cdot),\partial_t v(t,\cdot))$, we have for every $t\in\left[\tau+\frac{1}{4},\tau+\frac{3}{4}\right]$:
\begin{align*}
\G(v(t,\cdot),t)-\G(v,t)
\geq&\,\eps\left(\W(v,t)-\frac{\Theta}{2}\right)+2\Vert\partial_t v(t,\cdot)\Vert^2-4\eps\left|\langle\partial_t v,\partial_{t,t}v\rangle\right| -2C\Vert\partial_{t,\theta} v(t,\cdot)\Vert_{L^2(\Sp)}^2\\
&-Qe^{-\frac{\tau}{2}}\Vert u\Vert_{H^1(\D_1)}
\end{align*}
for some constant $C>0$. Here $\G$ is as defined in \eqref{defGronde}. Injecting this in the previous estimate, and rearranging the terms as in the proof of Theorem \ref{th_epiperimetry_exp}, we obtain for some constants $d,D,C'>0$:
\begin{align*}
\W(v,\tau)-\W(v,\tau+1)\geq&\, d\left(\left(\W(v,\tau+1)-\frac{\Theta}{2}\right) +\Vert \partial_t v\Vert_{H^1(\Cy_{0}\setminus\Cy_{1})}^2\right)-4D\int_{0}^{1}\left(\Vert\partial_{t,\theta}v\Vert^2+\Vert\partial_{t,t}v\Vert^2\right)dt\\
&-C'e^{-\frac{\tau}{2}}\Vert u\Vert_{H^1(\D_1)}.
\end{align*}
By the monotonicity formula of Theorem \ref{Th_Monotonicity_quasi} (we denote $\lambda_0$ the constant from this result), we have
$$4\int_{0}^{1}\left(\Vert\partial_{t,\theta}v\Vert^2+\Vert\partial_{t,t}v\Vert^2\right)dt\geq\W(v,\tau)-\W(v,\tau+1)-\left(1-e^{-\frac{1}{2}}\right)\lambda_0\Vert u\Vert_{H^1(\D_1)}e^{-\frac{\tau}{2}}.$$
And so we obtain that for some constant $C''$:
$$\W(v,\tau+1)\leq \W(v,\tau)-\frac{d}{D+1-d}\left(\W(u,\tau)-\frac{\Theta}{2}+\Vert\partial_tv\Vert_{H^1(\Cy_\tau\setminus\Cy_{\tau+1})}^2\right)+ C''\Vert u\Vert_{H^1(\D_1)}e^{-\frac{\tau}{2}}.$$
Going back to the disk coordinate, and handling the estimate of $\Vert u_{e^{-1}r}-u_{r}\Vert_{H^1(\D_1\setminus\D_{e^{-1}})}$ as in the proof of Lemma \ref{lem_control_variation_u}, we obtain the result.
\end{proof}
\subsubsection{Uniqueness and speed of convergence of blow-ups}
The adaptation of Proposition \ref{prop_convpolyblow_up} to the quasiminimal setting is as follows.

\begin{proposition}\label{prop_convpolyblow_up_quasi}
Let $\Theta\in\{\pi,t_1,2\pi\}$, $M\geq 1$. There exist constants $c_2,C_2>0$, $\gamma\in (0,1)$ such that the following holds: let $u\in\QM(\D_1)$ such that $0\in\partial\Spt(u)$, $
\ov{u}\in \M_{\mathrm{hom}}^{\Theta}$ if $\Theta\in\{\pi,t_1\}$ (resp. $\ov{u}\in\mathrm{Span}(x^2+y^2,x^2-y^2,xy)$ if $\Theta=2\pi$), assume
$$W(u,0)\geq \frac{\Theta}{2},\quad \Vert u-\ov{u}\Vert_{H^1(\D_1)}\leq c_2\Vert \ov{u}\Vert_{H^1(\D_1)}.$$
Then there exists $\widehat{u}\in \M_{\mathrm{hom}}^{\Theta}$ such that for any $r\in (0,1]$:
$$\Vert u_r-\widehat{u}\Vert_{H^1(\D_1)}\leq C_2\min\left(r,\frac{\Vert u-\ov{u}\Vert_{H^1(\D_1)}}{\Vert \ov{u}\Vert_{H^1(\D_1)}}\right)^\gamma\Vert \ov{u}\Vert_{H^1(\D_1)},$$
and 
$$\Vert\ov{u}-\widehat{u}\Vert_{H^1(\D_1)}\leq C_2 \Vert u-\ov{u}\Vert_{H^1(\D_1)}^\gamma\Vert \ov{u}\Vert_{H^1(\D_1)}^{1-\gamma}.$$
\end{proposition}
\begin{proof}
The proof is the same as the proof of Proposition \ref{prop_convpolyblow_up}, by replacing $W$ with $W^{quasi}$, and the constant $c_1,C_1$ from Theorem \ref{th_epiperimetry} by $c_1,C_1$ from Theorem \ref{th_epi_quasi}.
\end{proof}

The $\eps$-regularity results near flat and angular boundary point also follow.

\begin{proposition}\label{prop_eps_reg_quasi}
There exist $c_3>0$, $\alpha,\kappa\in (0,1)$, with the following property: for any $u\in\M(\D_1)$, $\ov{u}\in\mathscr{M}_{\mathrm{hom}}^{\pi}\cup\mathscr{M}_{\mathrm{hom}}^{t_1}$ such that
\[\Vert u-\ov{u}\Vert_{H^1(\D_1)}\leq c_3,\]
then 
\begin{itemize}
\item If $\ov{u}=u^{\mathrm{I}}$, then there exists a function $h\in\mathcal{C}^{1,\alpha}\left(\left[-\frac{1}{2},\frac{1}{2}\right],\R\right)$ such that $\Vert h\Vert_{L^{\infty}\left(\left[-\frac{1}{2},\frac{1}{2}\right]\right)}\leq \sqrt[6]{\Vert u-u^{\mathrm{I}}\Vert_{H^1(\D_1)}}$ and
$$\Vert h\Vert_{\mathcal{C}^{1,\alpha}\left(\left[-\frac{1}{2},\frac{1}{2}\right]\right)}\lesssim \Vert u-u^\mathrm{I}\Vert_{H^1(\D_1)}^\kappa$$
and
$$\left\{(x,y)\in\D_{\frac{3}{4}}:|x|\leq \frac{1}{2},\ u(x,y)\neq 0\right\}=\left\{(x,y)\in\D_{\frac{3}{4}}:|x|\leq \frac{1}{2},\ y>h(x)\right\}.$$
Moreover, $\Delta u\in\mathcal{C}_\loc^{0,\alpha}\left(\D_{\frac{1}{2}}\cap\ov{\Spt(u)}\right)$ and $\Delta u=1$ in $\D_{\frac{1}{2}}\cap\partial\Spt(u)$.
\item If $u=u^{\mathrm{II}}$ and $W(u,0)\geq \frac{t_1}{2}$, then there exists a diffeomorphism $\Phi\in\mathcal{C}^{1,\gamma}\left(\D_{\frac{1}{2}},\D_{\frac{1}{2}}\right)$ such that
$$\Phi(0)=0,\ D\Phi(0)=I_2,\ \Vert \Phi-\mathrm{id}\Vert_{\mathcal{C}^{1,\nu}\left(\D_{\frac{1}{2}}\right)}\leq C\Vert u-u^{\mathrm{II}}\Vert_{H^1(\D_1)}^\mu$$
and
$$\Spt(u)\cap\D_{1/2}=\Phi\left(\Spt(u^{\mathrm{II}})\cap\D_{1/2}\right).$$
\end{itemize}
\end{proposition}

In the first case, the only difference with Proposition \ref{prop_C1a} is that $\Delta u$ is only proved to be $\mathcal{C}^{0,\alpha}$ (instead of $\mathcal{C}^{1,\alpha}$) up to the boundary. This is because we do not know  that $\Delta u$ is harmonic in $\Spt(u)$ anymore.

\begin{proof}
We only detail the part that differs (in the proof of Proposition \ref{prop_C1a}) in the quasi-minimal setting: the regularity of $\Delta u$ up to the boundary. Indeed, starting over from equation \eqref{eq_mean_value}, we have
$$
|\Delta u(q)-1|=\left|\Delta (u_{p,r}-u_{p,0})\left(\tilde{q}\right)\right|.
$$
Since $u_{p,r}\in \QM(\D_1)$ and $\D_{\tilde{q},\frac{1}{2}}\subset\Spt(u_{p,r})$, in particular we have for every $\varphi\in\mathcal{C}^{\infty}_c(\D_{\tilde{q},\frac{1}{2}},\R)$, $t\in\R$,
$$\int_{\D_{\tilde{q},\frac{1}{2}}}|\Delta u_{p,r}|^2\leq \int_{\D_{\tilde{q},\frac{1}{2}}}|\Delta (u_{p,r}+t\varphi)|^2+ |t|\Vert \varphi\Vert_{L^2(\D_{\tilde{q},\frac{1}{2}})}.$$
This implies that for every such $\varphi$, we have,
$$\left|\int_{\D_{\tilde{q},\frac{1}{2}}}\Delta u_{p,r}\Delta \varphi\right|\leq\frac{1}{2}\Vert\varphi\Vert_{L^2(\D_{\tilde{q},\frac{1}{2}})}.$$
Since $\Delta^2u_{p,0}=0$ in $\D_{\tilde{q},\frac{1}{2}}$, we obtain
$$\Vert \Delta^2 (u_{p,r}-u_{p,0})\Vert_{L^2(\D_{\tilde{q},\frac{1}{2}})}\leq \frac{1}{2}.$$
We remind that $\Vert u_{p,r}-u_{p,0}\Vert_{H^1(\D_1)}\lesssim \eps^\frac{\gamma}{12}$, so $u_{p,r}-u_{p,0}$ is bounded in $H^4(\D_{\tilde{q},\frac{1}{2}})$. By Gagliardo-Nirenberg inequality (interpolating $W^{2,\infty}(\D_{\tilde{q},\frac{1}{2}})$ between $H^1(\D_{\tilde{q},\frac{1}{2}})$ and $H^4(\D_{\tilde{q},\frac{1}{2}})$) we obtain
\begin{align*}
|\Delta u(q)-1|&=\left|\Delta (u_{p,r}-u_{p,0})\left(\tilde{q}\right)\right|\\
&\lesssim \Vert u_{p,r}-u_{p,0}\Vert_{H^1(\D_{\tilde{q},\frac{1}{2}})}^{\frac{1}{3}}
\Vert u_{p,r}-u_{p,0}\Vert_{H^4(\D_{\tilde{q},\frac{1}{2}})}^{\frac{2}{3}}\\
&\lesssim r^{\frac{\gamma}{36}}.
\end{align*}
This concludes the proof of the Hölder regularity of $\Delta u$ up to the regular boundary.
\end{proof}

\subsubsection{Boundary decomposition for quasiminimizers}

\begin{proof}[Proof of Theorem \ref{mr_total_quasi}]
The proof is identical to the proof of Theorem \ref{mr_total}: we define $\mathcal{R}_u,\mathcal{A}_u,\mathcal{N}_u,\mathcal{J}_u,\mathcal{E}_u$ the same way. Proposition \ref{prop_eps_reg_quasi} implies the claim on $\mathcal{R}_u,\mathcal{A}_u$. The claim on $\mathcal{J}_u$ is a consequence of Proposition \ref{prop_convpolyblow_up_quasi}.
Finally, the description of $\mathcal{E}_u$ is a consequence of Lemma \ref{lem_existence_blowup_quasi} and Proposition \ref{prop_convpolyblow_up_quasi}.
\end{proof}

\subsection{Application to buckling eigenvalue}\label{subsec_buckling}

We remind that the definition of $\Lambda_1(\Om)$, for an open set $\Om$ with finite area, is given in \eqref{expr_Lambda1}. We write $u_\Om$ an (arbitrary) choice of non-zero eigenfunction on $\Om$ associated to $\Lambda_1(\Om)$. Weinberger and Willms's argument may be decomposed in the following steps (we remind a detailed proof may be found in \cite[Prop 4.4]{K00} and \cite{AB03}). 
\begin{itemize}
\item[a) ] If $\Om\subset\R^2$ has finite measure, $\mathcal{C}^2$ boundary, and is a critical point of the functional $|\Om|\Lambda_1(\Om)$ with respect to shape derivatives, then $|\Delta u_\Om|$ must be constant in $\partial\Om$.
\item[b) ]If $\Om\subset\R^2$ is a bounded simply connected set and $|\Delta u_\Om|$ is constant on $\partial\Om$, then $\Delta u_\Om+\Lambda_1(\Om)u_\Om$ is constant in $\Om$. 
\item[c) ]If $\Om\subset\R^2$ is open with finite area such that $\Delta u_{\Om}+\Lambda_1(\Om)u_\Om$ is constant in $\Om$, then $\Lambda_1(\Om)\geq\Lambda_1(\Om^*)$, where $\Om^*$ is a disk with the same area as $\Om$.
\end{itemize}
These three steps imply that if $\Om$ is a minimizer of $|\Om|\Lambda_1(\Om)$ that is bounded, simply connected, with $\mathcal{C}^2$ boundary, then $\Om$ must be the disk.

As a direct consequence of Theorem \ref{mr_total_quasi}, we may lower the initial hypothesis to obtain this in two ways. We first prove that the first eigenfunction of a minimizers of the buckling eigenvalue problem \eqref{expr_Lambda1} under area constraint belong to the class of quasi-minimizers given in Definition \ref{def_quasimin}.

We remind that since $\Lambda_1(t\Om)=t^{-2}\Lambda_1(\Om)$, then up to a dilation,  minimizing $\Lambda_1(\Om)$ under an area constraint on $\Om$ is equivalent to minimizing either $|\Om|\Lambda_1(\Om)$ or $\Lambda_1(\Om)+|\Om|$ among every set.

\begin{lemma}\label{lem_buckling_quasimin}
Let $u\in H^2(\R^2)$ be a minimizer of $E(\cdot;\R^2)$ under the constraint $\int_{\R^2}|\nabla u|^2=1$. 
Then there exists $\ov{r}>0$ such that for any $p\in\R^2$, we have 
$u_{\ov{r}}\in\QM(\D_{p,1})$.
\end{lemma}
Similar computation may be found in \cite{S16} (as well as the existence of such a minimizer, obtained via a concentration-compactness procedure).
\begin{proof}
The function $u$ is a minimizer of the functional
$$v\in H^2(\R^2)\setminus\{0\}\mapsto\frac{\int_{\R^2}|\Delta v|^2}{\int_{\R^2}|\nabla v|^2}+|\Spt(v)|.$$
We write $\Lambda:=\int_{\R^2}|\Delta u|^2$. Let $v\in H^2(\R^2)$ such that $\{u\neq v\}\subset\D_{p,r}$ for some disk $\D_{p,r}$, suppose without loss of generality that
$$E(v;\D_{p,r})\leq E(u;\D_{p,r}),$$
otherwise the quasiminimality condition \eqref{eq_quasiminimality} is directly verified. Then
\begin{align*}
\left|1-\int_{\R^2}|\nabla v|^2\right|&=\left|\int_{\D_{p,r}}|\nabla u|^2-|\nabla v|^2\right|=\left|\int_{\D_{p,r}}(u-v)\Delta(u+v)\right|\\
&\leq \Vert u-v\Vert_{L^2(\D_{p,r})}\Vert \Delta (u+v)\Vert_{L^2(\D_{p,r})}\\
&\lesssim r^2\Vert \Delta(u-v)\Vert_{L^2(\D_{p,r})}\Vert \Delta (u+v)\Vert_{L^2(\D_{p,r})}\text{ by Poincaré inequality}\\
&\lesssim r^2\left(\Vert \Delta u\Vert_{L^2(\D_{p,r})}^2+\Vert \Delta v\Vert_{L^2(\D_{p,r})}^2\right)\\
&\lesssim r^2\left(\Vert \Delta u\Vert_{L^2(\R^2)}^2+|\Spt(u)|\right)\text{ since }E(v;\D_{p,r})\leq E(u;\D_{p,r}).
\end{align*}
So for a small enough $r>0$, we have $\int_{\R^2}|\nabla v|^2\neq 0$. Then by the minimality of $u$ we obtain
$$\Lambda+|\Spt(u)|\leq \frac{\Lambda+\int_{\D_{p,r}}\left(|\Delta v|^2-|\Delta u|^2\right)}{1+\int_{\D_{p,r}}\left(|\nabla v|^2-|\nabla u|^2\right)}+|\Spt(v)|.$$
This simplifies to
$$
E(u;\D_{p,r})-E(v;\D_{p,r})\leq\left(\Lambda+|\Spt(u)|-|\Spt(v)|\right)\int_{\D_{p,r}}\left(|\nabla u|^2-|\nabla v|^2\right).
$$
We bound the last term:
\begin{align*}
\int_{\D_{p,r}}\left(|\nabla u|^2-|\nabla v|^2\right)&=\int_{\D_{p,r}}\Delta(v+u)(v-u)\\
&\lesssim \Vert\Delta (v+u)\Vert_{L^2(\D_{p,r})}\Vert v-u\Vert_{L^2(\D_{p,r})}\\
&\lesssim \left(\Vert\Delta u\Vert_{L^2(\R^2)}+|\Spt(u)|^\frac{1}{2}\right)\Vert v-u\Vert_{L^2(\D_{p,r})}\text{ since }E(v;\D_{p,r})\leq E(u;\D_{p,r}).
\end{align*}
And $\Lambda+|\Spt(u)|-|\Spt(v)|\leq \Lambda+\pi r^2$. This proves the result, for some sufficiently small $\ov{r}$.
\end{proof}

\begin{corollary}\label{cor_buckling_1}
Let $\Om\subset\R^2$ be a minimizer of $\Om\mapsto\Lambda_1(\Om)$ among open sets of area $1$. Let $u_\Om$ be a buckling eigenfunction associated to $\Lambda_1(\Om)$. Then $\Om$ is bounded and $u_\Om$ satisfies the conclusion of Theorem \ref{mr_total_quasi}; in particular for any $p\in\partial\Om$ such that $\liminf_{r\to 0}\frac{|\Om\cap\D_{p,r}|}{|\D_{p,r}|}<1$, then there exists $r>0$ such that $\D_{p,r}\cap\Om$ is either a $\mathcal{C}^{1,\alpha}$ curve, or two $\mathcal{C}^{1,\alpha}$ curves meeting with an angle $t_1$.
\end{corollary}

\begin{proof}
Up to a dilation, we assume that $\Om$ is a minimizer (among open sets of finite area) of $\Om\mapsto \Lambda_1(\Om)+|\Om|$. We assume $u_\Om$ is normalized by $\int_{\Om}|\nabla u_\Om|^2=1$. This means that $u_\Om$ is a minimizer of $E(\cdot;\R^2)$ under the constraint $\int_{\R^2}|\nabla u_\Om|^2=1$. By Lemma \ref{lem_buckling_quasimin}, there exists $\ov{r}>0$ such that for every $p\in\R^2$ we have
$$(u_\Om)_{p,\ov{r}}\in\QM(\D_{1}).$$
By the non-degeneracy Lemma \ref{lem_nondeg_quasi}, there exists $c>0$ such that for any $p\in\R^2$, if $u(p)\neq 0$ then
$$
\int_{\D_{p,\ov{r}}}\left(|\nabla u_\Om|^2+u_\Om^2\right)\geq c.$$
Consider a sequence of disjoint disks $(\D_{p_i,\ov{r}})_{i=1,\hdots,N}$ (for some $N\in\N^*$) such that $u_\Om(p_i)\neq 0$ for every $i$, then applying the above inequality to each disk we obtain
$$\int_{\R^2}\left(|\nabla u_\Om|^2+|u_\Om|^2\right)\geq Nc.$$
Thus the number $N$ of such disjoint disks is bounded from above, meaning that $\Om$ must be a bounded set.

We now apply Theorem \ref{mr_total_quasi} to $(u_\Om)_{\ov{r}}$, which gives the rest of the conclusions. Note that $\partial\Om$ is, a priori, only a subset of $\partial\Spt(u_\Om)$.

\end{proof}

\begin{corollary}\label{cor_buckling_2}
Let $\Om\subset\R^2$ be a minimizer of $\Om\mapsto\Lambda_1(\Om)$ among open sets of area $1$. Let $u_\Om$ be a buckling eigenfunction associated to $\Lambda_1(\Om)$ and assume either that
\begin{equation}\label{eq_assumption_density_1}
\Om\text{ is simply connected and for all }p\in\partial\Om,\ \liminf_{r\to 0}\frac{|\D_{p,r}\cap\Om|}{|\D_{p,r}|}<\frac{t_1}{2\pi}(\approx 0.715),
\end{equation}
or
\begin{equation}\label{eq_assumption_density_2}
\Delta u_\Om+\Lambda_1(\Om)u_\Om\geq 0\text{ in }\Om\text{ and for all }p\in\partial\Om,\text{ we have } \liminf_{r\to 0}\frac{|\D_{p,r}\cap\Om|}{|\D_{p,r}|}<1.
\end{equation}
Then $\Om$ is a disk.
\end{corollary}

\begin{proof}
We first apply corollary \ref{cor_buckling_1}: $\Om$ is bounded, and $u_\Om$ verifies the conclusion of Theorem \ref{mr_total_quasi}.

If the first hypothesis \eqref{eq_assumption_density_1} is verified, then we know that every point of $\partial\Om$ belongs to $\mathcal{R}_{u_\Om}$. Since
$\Om$ is bounded and simply connected, then $\mathcal{R}_{u_\Om}$ is connected. $\Delta u_\Om$ is continuous up to the boundary with values in $\{-1,+1\}$, so by the connectedness of the boundary it is constant. By maximum principle, the (harmonic) function $\Delta u_\Om+\Lambda_1(\Om)u_\Om$ is constant in $\Om$: the step c) of Weinberger and Willms's argument applies and $\Om$ must be a disk.

If the second hypothesis \eqref{eq_assumption_density_2} is verified, then every point of $\partial\Om$ belongs to $\mathcal{R}_u\sqcup \mathcal{A}_u$. Assume there exists some $p\in\mathcal{A}_u$, then $\Delta u_\Om+\Lambda_1(\Om)u_\Om$ must change sign in any neighbourhood of $p$. Indeed, the blow-up sequence $(u_{\Om})_{p,r}$ converges in $H^2(\D_1)$ to a homogeneous solution of type II, denoted $\widehat{u}$, and $\Delta \widehat{u}=\lim_{r\to 0}\left(\Delta u_{p,r}+r^2 \Lambda u_{p,r}\right)$ changes sign. This contradicts our hypothesis, so $\mathcal{A}_u$ must be empty and every point of $\partial\Om$ is regular. 

Thus $\partial\Om\subset\mathcal{R}_{u_\Om}$, $\Delta u_\Om|_{\Om}$ extends continuously to $\partial\Om$ with value $\pm 1$. Since $u_\Om=0$ on $\partial\Om$ and $\Delta u_\Om+\Lambda_1(\Om)u_\Om\geq 0$, then $\Delta u_\Om+\Lambda_1(\Om)u_\Om=1$ in $\partial\Om$ and harmonic, so it is constant in $\Om$. We conclude with step c) of Weinberger and Willms's argument described above: $\Om$ is necessarily a disk.

\end{proof}

There are several obstacles to obtain a general proof of the optimality of the disk from this. First is proving the hypothesis \eqref{eq_assumption_density_1} would hold for a general minimizer ; either the simple connectedness should be attained as a \textit{consequence} of the optimality, from some unknown geometric argument, or it could be considered as a \textit{constraint} (in an attempt to prove the optimality of the disk at least among simply connected sets). If it is considered as a constraint, then our main Theorem \ref{mr_total_quasi} does not apply since we frequently use non-simply connected competitors (for instance as early as in the proof of Lemma \ref{lem_BMO} and Lemma \ref{lem_nondegeneracy}).

The density assumption is not obvious either ; although relying on heuristics in some parts, the optimal set for the drag of Richardson in \cite{R95} shows that angular points may occur in optimal shapes. Even assuming that we work among simply connected shapes, some geometric argument specific to the buckling problem would be needed to prove that angular points do not occur for minimizers. The same can be said for the assumption $\Delta u_\Om+\Lambda_1(\Om)u_\Om\geq 0$ in \eqref{eq_assumption_density_2}.

Finally, even assuming that $\partial\Om$ is connected, one would need to understand the behaviour of $\Delta u_\Om|_{\partial\Om}$ near junction point (here $\Delta u_\Om|_{\partial\Om}$ is expected to \textbf{not} change sign) and explosion points (which are expected to not exist).\bigbreak

\noindent{\bf Acknowledgements: }
This work was supported by the project ANR STOIQUES (ANR-24-CE40-2216) financed by the French Agence Nationale de la Recherche (ANR).

\bibliographystyle{plain}
\bibliography{biblio}
\end{document}